\documentclass[journal,comsoc,onecolumn,romanappendices]{IEEEtran}
\usepackage{amsmath,amssymb} 

\usepackage{graphicx}
\usepackage{amsmath,amssymb,amsthm} 

\usepackage[T1]{fontenc}
\usepackage[latin1]{inputenc}
\usepackage[scanall]{psfrag}
\usepackage{comment}
\usepackage{color}
\usepackage{setspace}
\usepackage{array}   
\usepackage{tikz}
\usepackage{booktabs} 
\usepackage{multirow}

\usepackage[section]{placeins}

\newcommand{\R}{{\mathbb{R}}}
\newcommand{\Ll}{{\mathbb{L}}}
\newcommand{\M}{{\mathbb{M}}}

\newcommand{\C}{\mathbb{C}}

\newcommand{\D}{\mathbb{D}}

\newcommand{\xx}{\mathbf{x}}

\renewcommand{\vec}[1]{ \boldsymbol{#1}}

\newcommand{\Ingard}{Ingard}
\newcommand{\Loewner}{L{\"o}wner}

\renewcommand{\Re}{{\rm Re\,}}
\renewcommand{\Im}{{\rm Im\,}}

\newcommand{\bbm}[1]{\left[\begin{matrix} #1 \end{matrix}\right]}
\newcommand{\sbm}[1]{\left[\begin{smallmatrix} #1
             \end{smallmatrix}\right]}

\newcommand{\spm}[1]{\left(\begin{smallmatrix} #1
             \end{smallmatrix}\right)}

\newcommand{\CTF}{\mathbf{G}}
\newcommand{\CTFK}{\mathbf{K}}
\newcommand{\DTF}{\mathbf{D}}

\newcommand{\abs}[1]{\left \vert #1 \right \vert}

\newcommand{\Range}[1]{{\rm range}\left (#1 \right )}
\newcommand{\Null}[1]{{\rm ker} \left (#1 \right )}
\newcommand{\BLO}{\mathcal L}
\newcommand{\norm}[1]{\left \|{#1} \right \|}

\DeclareMathOperator{\Chain} {CHAIN}
\DeclareMathOperator{\TI} {TI}
\DeclareMathOperator{\FI} {FI}
\DeclareMathOperator{\BI} {BI}
\DeclareMathOperator{\OF} {OF}
\DeclareMathOperator{\IF} {IF}
\DeclareMathOperator{\SR} {SR}

\newtheorem{thm}{Theorem}
\newtheorem{cor}[thm]{Corollary}
\newtheorem{prop}[thm]{Proposition}
\newtheorem{lemma}[thm]{Lemma}
\newtheorem{defn}{Definition}
\newtheorem{sta}{Standing Assumption}
\newtheorem{exm}{Example}
\newtheorem{remark}{Remark}

\DeclareOldFontCommand{\bf}{\normalfont\bfseries}{\mathbf}
\DeclareOldFontCommand{\rm}{\normalfont\rmfamily}{\mathrm}

\newcolumntype{C}{>{$}c<{$}} 

\title{Numerical modelling  of \\ coupled linear dynamical systems}
\author{\IEEEauthorblockN{Juha Kuortti\IEEEauthorrefmark{1} and Jarmo Malinen\IEEEauthorrefmark{1} and Tom Gustafsson\IEEEauthorrefmark{1}}\\
\IEEEauthorblockA{\IEEEauthorrefmark{1}Department of Mathematics and Systems Analysis, Aalto University}\\
\thanks {Manuscript received XX.XX.XX 
  Corresponding author: J.~Malinen (email:jarmo.malinen@aalto.fi)}}

\makeatletter
\def\name#1{\gdef\@name{#1\\}}
\makeatother
\name{{\em Juha Kuortti, Jarmo Malinen, Tom Gustafsson}}

\begin{document}

\maketitle

\bibliographystyle{IEEEtran}

\begin{abstract}
Numerical modelling of several coupled passive linear dynamical
systems (LDS) is considered. Since such component systems may arise
from partial differential equations, transfer function descriptions,
lumped systems, measurement data, etc., the first step is to
discretise them into finite-dimensional LDSs using, e.g., the finite
element method, autoregressive techniques, and interpolation.  The
finite-dimensional component systems may satisfy various types of
energy (in)equalities due to passivity that require translation into a
common form such as the scattering passive representation. Only then
can the component systems be coupled in a desired feedback
configuration by computing pairwise Redheffer star products of LDSs.

Unfortunately, a straightforward approach may fail due to
ill-posedness of feedback loops between component systems.
Adversities are particularly likely if some component systems have no
energy dissipation at all, and this may happen even if the fully
coupled system could be described by a finite-dimensional LDS. An
approach is proposed for obtaining the coupled system that is based on
passivity preserving regularisation.  Two practical examples are given
to illuminate the challenges and the proposed methods to overcome
them: the Butterworth low-pass filter and the termination of an
acoustic waveguide to an irrational impedance.

\end{abstract}

\section{Introduction}

In practical modelling work, various kinds of linear dynamical systems
need be interconnected. The ultimate purpose is to produce computer
software that is able to approximate the composite system behaviour in
frequency and time domains to a sufficient degree. Not only is the
discretisation of the component systems a challenge in itself (not to
be addressed in this work), but also the original, undiscretised
component systems may be represented in various mutually incompatible
ways. The purpose of this article is to show how tools from
mathematical systems theory can be used to overcome these challenges
in mathematical modelling.

Let us continue the discussion in terms of an example from acoustics
that will be treated in Section~\ref{AcTransLineSec} below. The
Webster's lossless horn model
\begin{equation} \label{eq:websterIntro}
    \frac{\partial^2 \phi}{\partial t^2}
    =  \frac{c^2}{\mathcal A(\chi)}\frac{\partial }{\partial \chi}\left(\mathcal A(\chi) \frac{\partial \phi}{\partial \chi} \right)
\quad \text{for all } \quad \chi \in (0,L)  \quad \text{and } \quad t \geq 0
\end{equation}
describes the longitudinal acoustics of a tubular acoustic waveguide
of length $L$ and the intersectional area $\mathcal A = \mathcal
A(\chi)$ where $c > 0$ denotes the speed of sound.  The solution $\phi
= \phi(t, \chi)$ is the \emph{velocity potential}, and the
(perturbation) volume velocity $i = i(t,\chi)$ and the sound pressure
$p = p(t,\chi)$ are given by $i = - \mathcal A(\chi)\tfrac{\partial
  \phi}{\partial \chi}$ and $p = \rho \tfrac{\partial \phi}{\partial
  \chi}$, respectively, where $\rho > 0$ is the density of the
medium. The external control for Eq.~\eqref{eq:websterIntro} takes
place through the boundary conditions
\begin{equation} \label{eq:webster_boundaryIntro}
    - \mathcal A(0) \frac{\partial \phi}{\partial \chi}(0,t) = i_1(t) \quad \text{and } \quad 
     \mathcal A(L) \frac{\partial \phi}{\partial \chi}(L,t) = i_2(t)  \quad \text{for } \quad t \geq 0
\end{equation}
where the acoustic volume velocities $i_1$ and $i_2$ represent input
signals of corresponding to currents at the ends of the
waveguide. There is an output signal at both ends of the waveguide,
given by
\begin{equation} \label{eq:webster_boundaryObsIntro}
\begin{aligned}
  p_1(t) = \rho  \frac{\partial \phi}{\partial t}(0,t)  \quad \text{and } \quad 
  p_2(t) = \rho  \frac{\partial \phi}{\partial t}(L,t)  \quad \text{for } \quad t \geq 0
\end{aligned}
\end{equation}
that represent sound pressures that are the acoustic counterparts of
voltages.

Now, consider the end $\chi = L$ of the waveguide in
Eqs.~\eqref{eq:websterIntro}--\eqref{eq:webster_boundaryObsIntro} to
be coupled to infinitely large exterior space where both the
parameters $c$ and $\rho$ remain the same. A classical and much used
model for the exterior space acoustics is provided by Morse and
\Ingard~\cite{Morse:1968} where they consider a piston (with diameter
$a > 0$) in a cylinder that opens to the 3D half space bounded by a
hard, perfectly reflecting wall. Instead of giving a time-domain model
such as Eq.~\eqref{eq:websterIntro},~\cite{Morse:1968} derives a
\emph{mechanical impedance}.  In terms of the \emph{acoustical
  impedance}, the piston model is given by the irrational analytic
function
\begin{equation}
  \label{eq:acImpIntro}
  Z (s) = Z_0 \left( 1 - \frac{c}{a s} \left (i J_1\left(-\frac{2
    a i}{c} s \right) + H_1\left(-\frac{2 a i}{c} s \right) \right )
  \right) \quad \text{for } \quad \Re{s} > 0
\end{equation}
where $Z_0 = \rho c/\mathcal A_0$, $\mathcal A_0 = \pi a^2$, is the
\emph{characteristic impedance} of an acoustic wave\-guide having a
constant cross-section area $\mathcal{A}_0$, and $J_1(\cdot)$ and
$H_1(\cdot)$ are the Bessel and Struve functions, respectively; see
\cite[Eqs.~(9.1.20)~and~(12.1.6)]{A-S:HMF}.

Rigorously treating the direct coupling of the two
infinite-dimensional \emph{conservative} dynamical systems described
by
Eqs.~\eqref{eq:websterIntro}--\eqref{eq:webster_boundaryObsIntro}~and~Eq.~\eqref{eq:acImpIntro}
appears to be quite difficult. If also the exterior space system of
Eq.~\eqref{eq:acImpIntro} was given as a PDE in terms of a boundary
control system in time domain, then both the systems in
Eqs.~\eqref{eq:websterIntro}--\eqref{eq:webster_boundaryObsIntro}~and~Eq.~\eqref{eq:acImpIntro}
could thus be described using impedance conservative \emph{strong
  boundary nodes} due to their finite-dimensional signals; see
\cite[Definition~4.4~and~Theorem~4.7]{M-S:IPCBCS}. Then their
composition could be understood as a \emph{transmission graph} whose
impedance conservativity and internal well-posedness follows from
\cite[Theorem~3.3]{A-M:CPBCS}. Another popular framework for treating
such couplings is provided by the \emph{port-Hamiltonian systems} in,
e.g., \cite{S-J:PHST} and the references therein. The objective of
this article is, however, more practical: to propose methods for
coupling finite-dimensional, (spatially) discretised versions of
systems that are suitable for numerical simulations.  We seek to
approximate
Eqs.~\eqref{eq:websterIntro}--\eqref{eq:webster_boundaryObsIntro} by
the finite-dimensional linear dynamical system $\Sigma_{ac}$ given by
\begin{equation*}
  \begin{cases}
    x'(t) & = A_{ac} x(t) +  B_{ac} \sbm{i_1(t) \\ i_2(t)}, \\
    \sbm{p_1(t) \\ p_2(t)} & = C_{ac} x(t) \quad \text{for } t \geq 0,
  \end{cases}
\end{equation*}
and Eq.~\eqref{eq:acImpIntro} by the transfer function
\begin{equation*}
  \hat p_2(s) = \left (D_{ex}  + C_{ex} (s - A_{ex})^{-1} B_{ex} \right ) \hat i_2(s)
  \quad \text{for} \quad \Re{s} >  0
\end{equation*}
of another finite-dimensional system $\Sigma_{ex}$. Possible choices
for the approximations, \emph{Finite Element Method (FEM)} and
\emph{L{\"o}wner interpolation}, are discussed in
Sections~\ref{FEMSubSec} and~\ref{LownerSubSec}, respectively. The
required feedback connection between $\Sigma_{ac}$ and $\Sigma_{ex}$
can --- at least in principle --- be computed as the \emph{Redheffer
  star product} (see,~e.g.,~\cite[Section~4]{HK:CSAH},
\cite[Section~10]{Z-D-G:ROC}, \cite[Chapter~XIV]{F-F:CLA}) of certain
\emph{externally Cayley transformed} versions of systems $\Sigma_{ac}$
and $\Sigma_{ex}$ as explained in Section~\ref{RedhefferSec} below.

Unfortunately, complications due to the lack of well-posedness of the
explicitly treated feedback loop make it sometimes impossible to
directly compute the closed loop system.  Within passive systems,
these complications are typically showstoppers for those systems that
are, in fact, conservative. Indeed, conservative systems lack all
energy dissipative mechanisms that could help feedback loops to
satisfy a version of the Nyquist stability criterion.  That the system
described by
Eqs.~\eqref{eq:websterIntro}--\eqref{eq:webster_boundaryObsIntro} is,
indeed, (impedance) conservative follows from
\cite[Corollary~5.2]{A-L-M:AWGIDDS} recalling
\cite[Definition~3.2]{M-S:IPCBCS}.  Our approach is to artificially
regularise such systems to be \emph{properly passive} (see
Definition~\ref{ProperlyImpPassDef}) by adding artificial resistive
losses scaled by a regularisation parameter $\varepsilon$, carrying
out the feedback connection using the Redheffer star product for
$\varepsilon > 0$, and finally extirpating the singular terms at
$\epsilon = 0$ while letting $\varepsilon \to 0$ in order to remove
the regularisation. A tractable example of this process is given in
Section~\ref{ButterworthSec} whereas the resistive regularisation is
used for spectral tuning in Section~\ref{AcTransLineSec}.

Both the models in
Eqs.~\eqref{eq:websterIntro}~and~\eqref{eq:acImpIntro} were originally
derived by theoretical considerations which is not always feasible or
even necessary. A sufficient approximation of time- or
frequency-domain behaviour can often be obtained by measurements,
leading to \emph{empirical models} whose quality is typically not
assessed by, say, mathematical error estimates rather than by
\emph{validation experiments}. In time domain, autoregressive
techniques such as \emph{Linear Prediction (LP, or LPC)} can be used
to estimate the parameters of a (discrete time) rational filter from
measured signals, however, often under some \emph{a priori} model
assumptions. For example, the filter transfer function is always
all-pole in \cite{JM:LPATR,Makhoul:SLP:1975}, and this is relevant for
transimpedances of transmission lines (such as the one defined by
Eqs.~\eqref{eq:websterIntro}--\eqref{eq:webster_boundaryObsIntro}) and
their counterparts consisting of discrete components (such as the
passive circuit for the fifth order Butterworth filter with impedance
given by Eq.~\eqref{eq:fullPiImpedance}). The subsequent realisation
of the rational filter transfer function as a discrete time linear
system can be carried out by using, e.g., the controllable canonical
realisation (see~\cite[Theorem~10.2]{PAF:LSOHS},
\cite[Section~4.4.2]{Antoulas:ALD:2005} \cite[Section~3]{Z-D-G:ROC}).
The transformation to a continuous time system, if necessary, is best
carried out using the inverse \emph{internal Cayley transformation}
(see,~e.g.,~\cite{H-M:CTATDS}, \cite[Section~12.2]{OS:WPLS}), and the
systems can finally be coupled using properly passive regularisation
and the Redheffer star product.

Considering the empirical modelling in frequency domain, direct
impedance measurements from physical circuits could be used as
interpolation data for L{\"o}wners method; see,
e.g.,~\cite[Section~4.5]{Antoulas:ALD:2005}.  In high frequency work
on electronic devices, one would prefer using scattering parameter
data to start with, produced by a \emph{Vector Network Analyser
  (VNA)}, for interpolation in a similar manner. In all cases, the
outcome would be a quadruple of four matrices $A, B, C$ and $D$ giving
rise to dynamical systems in Eq.~\eqref{eq:systemDyn} and transfer
functions in Eq.~\eqref{eq:ContinuousTimeTF} below. Numerical
performance may require additional \emph{dimension reduction} by,
e.g., interpolation or balanced realisations;
see,~e.g.,~\cite[Section~7]{Antoulas:ALD:2005},~\cite[Section~10]{BDOA:OC:2007},~\cite{KG:AOH}.

The purpose of this article is to present an economical toolbox of
mathematical systems theory techniques that is --- at least within
reasonable approximations, regularisations, and validations --- rich
enough for creating numerical time-domain solvers of physically
realistic passive linear feedback systems that are composed of more
simple passive components. The proposed toolbox is introduced in a
fairly self-contained manner, and it consists of basic realisation
algebra and system diagrams (Section~\ref{SystemSection}), internal
and external transformations of realisations
(Section~\ref{TransformSec}), and passivity preserving regularisation
methods for treating possible singularities in system matrices
(Section~\ref{RegSec}).  Our focus is to show how these tools can be
fruitfully used in simulations of two physically motivated
applications in Section~\eqref{ApplicationsSec}. The inconvenient
singularities may appear as three kinds of showstoppers:
\begin{enumerate}
  \item \label{ShowStopper1} \textbf{Realisability:} There are
    \emph{impedance} conservative systems in finite dimensions that
    cannot be described in terms of realisation theory of
    Section~\ref{BackgroundSec}.
  \item \label{ShowStopper2} \textbf{Well-posedness:} There are
    feedback configurations of \emph{scattering} conservative systems
    that are not well-posed, and, hence, cannot be treated within
    finite-dimensional systems as such.
  \item \label{ShowStopper3} \textbf{Technical issues:} There may be a
    technical, restrictive but removable assumption that is required
    only by the mathematical apparatus.
\end{enumerate}
A trivial example of a non-realisable impedance conservative,
non-well-posed physical system is the impedance $Z_L(s) := s L$ of a
single inductor or the admittance $A_C(s) := sC$ of a single capacitor
since $\lim_{\abs{s} \to \infty}{Z_L(s)} = \lim_{\abs{s} \to
  \infty}{A_C(s)} = \infty$ whereas $\lim_{\abs{s} \to
  \infty}{\CTF_{\Sigma}(s)} = D$ in Eq.~\eqref{eq:ContinuousTimeTF}
for any finite-dimensional system $\Sigma$. The most trivial example
of a non-well-posed feedback loop is provided by $\CTF(s) =
(s-1)/(s+1)$ and $\CTFK(s) \equiv 1$ that are both transfer functions
of a finite-dimensional scattering conservative systems whereas the
closed loop transfer function $(s - 1)/2$ is not a transfer function
of any finite-dimensional system. An example of a technical issue can
be found in Section~\ref{RedhefferSec} if one attempts to compute the
Redheffer star product using the chain transformation and
Theorem~\ref{RedhefferCascadeThm}: There is an extra invertibility
condition required by the intermediate chain transformation which is
not required by the Redheffer star product as discussed right after
Eq.~\eqref{eq:RedhefferComponents}. The challenge in using
finite-dimensional realisation theory for practical modelling is to
avoid these three kinds of showstoppers by an expedient use of what
mathematical systems theory offers and regularise component systems
when all other attempts fail.

As a side product, some extensions and clarifications of the
underlying mathematical systems theory framework are indicated:
Theorem~\ref{thm:SpringConnection} and
Corollary~\ref{thm:SpringConnectionCor} for second order passive
systems, the extended definition of the external Cayley transformation
allowing arbitrary characteristic impedances of coupling channels in
Section~\ref{ExtCayleySubSec},
Propositions~\ref{ExtCayleyPassivityProp}~and~\ref{ReciprocalPassivityProp}
for dealing with the passivity properties of this extension as well as
both the reciprocal transformations, and
Theorem~\ref{RedhefferScatteringThm} for well-posedness of a feedback
connection of two impedance passive systems. An almost elementary
proof of Theorem~\ref{RedhefferCascadeThm} on the computation of
Redheffer star products using chain scattering is provided in terms of
system diagram rules and their state space counterparts in
Sections~\ref{DiagrammeSec}~and~\ref{FundamentalSec}. The new
Definition~\ref{ProperlyImpPassDef} of proper passivity is due to the
requirements of the regularisation process introduced in
Section~\ref{RegSec}.

\section{\label{BackgroundSec} Background on systems and passivity}

In this article, we consider continuous time finite-dimensional linear
(dynamical) systems described by the state space equations
\begin{equation}
  \label{eq:systemDyn}
\begin{cases}
    x'(t) & = A x(t) +  B u(t), \\
    y(t) & = C x(t) +   D u(t),
\end{cases}
\quad \text{or, briefly,} \quad \bbm{x'(t) \\ y(t)} = \bbm{A & B \\ C
  & D} \bbm{x(t) \\ u(t)} \quad \text{for} \quad t \in I.
\end{equation}
The quadruple $\Sigma = \sbm{ A & B \\ C & D }$ defining
Eq.~\eqref{eq:systemDyn} is identified with the linear system.  The
temporal domain $I \subset \R$ is an interval, and it is not necessary
to specify it for the purposes of this article except in
Corollary~\ref{thm:SpringConnectionCor}.
\begin{sta}\label{FirstSA}
  We assume that all linear dynamical systems $\Sigma$ are real and
  finite-dimensional in the sense that the dimensions of the
  submatrices in $\Sigma$ make Eq.~\eqref{eq:systemDyn} well-defined.
  We assume that all signals in systems are real and sufficiently
  smooth for the classical solvability Eq.~\eqref{eq:systemDyn}.
  Furthermore, all vector norms and inner products are assumed to be
  euclidean, and all matrix norms are induced by the euclidean vector
  norm.
\end{sta}
We call $A$ the semigroup generator, $D$ the feedthrough matrix, and
$B$, $C$, input and output matrices, respectively.  The input signal
$u(\cdot)$, the state trajectory $x(\cdot)$, and the output signal
$y(\cdot)$ are all column vectors. We assume that $D$ is always a
square matrix, and thus the input and output signals are of the same
dimension.

System $\Sigma = \sbm{ A&B \\ C & D }$ is called \emph{impedance
  passive} if the functions in Eq.~\eqref{eq:systemDyn} satisfy the
energy inequality
\begin{equation*}
  \frac{d}{dt} \norm{ x(t) }^2 \leq 2 \left <u(t), y(t) \right > \quad \text{for all } t \in I.
\end{equation*}
If this inequality is satisfied as an equality, the system is then
called \emph{impedance conservative}. Both of these properties can be
checked in terms of a Linear Matrix Inequality (LMI):
\begin{prop}\label{ImpPassiveProp}
  Let $\Sigma = \sbm{ A&B \\ C & D }$ be linear system. Then $\Sigma$
  is impedance passive if and only if
  \begin{equation*} \label{eq:ImpPassiveProp}
    \bbm{A^T +A & B-C^T \\ B^T -C & -D^T-D } \leq  \bbm{0 & 0 \\ 0 & 0}.
  \end{equation*}
It is impedance conservative if and only if the inequality holds as an
equality.
\end{prop}
\noindent This is the finite-dimensional version of
\cite[Theorem~4.2(vi)]{Staffans:PCCT:2002}.

A passive system is sometimes represented as a second order
multivariate system as in Section~\ref{AcTransLineSec} where Finite
Element discretisation is used.
\begin{thm} \label{thm:SpringConnection}
Let $M$, $P$, and $K$ be symmetric positive definite $m \times m$
matrices of which $M$ and $K$ are invertible. Let $F$ be a $m \times
k$ matrix and $Q_1$, $Q_2$ $m \times m$ matrices. Then the following
holds:
  \begin{enumerate}
  \item \label{thm:SpringConnectionClaim1} The second order coupled system of
    ODEs
    \begin{equation} \label{eq:SpringConnection}
      \begin{cases}
        & M z''(t) + P z'(t) + K z(t) = F u(t),  \\
        & y(t) = Q_1 z(t) + Q_2 z'(t), \quad t \in I
      \end{cases}
    \end{equation}
    defines a linear system $\Sigma_i$ (with the state space of
    dimension $n = 2m$) by
    \begin{equation} \label{eq:SpringConnection2}
      \begin{cases}
        x'(t) & = \bbm{0 & K^{1/2}M^{-1/2 }\\ -M^{-1/2 }K^{1/2} & - M^{-1/2 }PM^{-1/2}} x(t) + \frac{1}{\sqrt{2}} \bbm{0\\
          M^{-1/2} F }u(t)\\
        y(t)  & = \sqrt{2}   \bbm{Q_1 K^{-1/2} & Q_2 M^{-1/2 } } x(t), \quad t \in I
      \end{cases}
    \end{equation}
    where $x(t):= \frac{1}{\sqrt{2}}\bbm{K^{1/2}  z(t) \\ M^{1/2} z'(t)}  $.
  \item \label{thm:SpringConnectionClaim2}  $\Sigma_i$ is
    impedance passive if and only if $Q_1 = 0$ and $Q_2 =  \tfrac{1}{2} F^T$.
  \item \label{thm:SpringConnectionClaim3}  $\Sigma_i$ is
    impedance conservative if and only if $P = Q_1 = 0$ and $Q_2 =
    \tfrac{1}{2} F^T$.
  \end{enumerate}
\end{thm}
\noindent Observe that if $M,P,K$ are positive scalars in a damped
mass-spring system, then $\norm{x(t)}^2 = \frac{1}{2} \norm{K^{1/2}
  z(t)}^2 + \frac{1}{2} \norm{M^{1/2} z'(t)}^2$ which is the sum of
the potential and kinetic energies. For this reason, the matrices $M$
and $K$ are called \emph{mass} and \emph{stiffness matrices},
respectively, and they define the physical energy norm of the system
requiring the normalisation $1/\sqrt{2}$ in equations. That $Q_1 = 0$
is a necessary condition for impedance passivity reflects the fact
that such physical systems must have co-located sensors and actuators
in the sense of, e.g., \cite{Curtain:RSC:2000}. A scattering
conservative analogue of Theorem~\ref{thm:SpringConnection} is given
in \cite[Theorems~1.1~and~1.2]{W-T:HTGCWPLSOOTA} in infinite
dimensions.
\begin{proof}
  Claim~\eqref{thm:SpringConnectionClaim1}: The transfer function of
  the system described by Eq.~\eqref{eq:SpringConnection} is given by
\begin{equation*}
  \CTF(s) = \left (Q_1 + s Q_2 \right ) \left (s^2 M + s P + K \right )^{-1} F
  = \left (\frac{Q_1}{s^2} + \frac{Q_2}{s} \right ) \left ( M + \frac{P}{s} + \frac{K}{s^2} \right)^{-1} F. 
\end{equation*}
Since $M$ is invertible, the matrix $M + \frac{ P}{s} + \frac{K}{s^2}$
is invertible for all $s$ with $\abs{s}$ large enough, implying $D_i
:= \lim_{\abs{s} \to \infty}{\CTF(s)} = 0$. Hence, the feedthrough
matrix $D_i$ vanishes for any realisation modelling
Eq.~\eqref{eq:SpringConnection}. Otherwise, it is a matter of
straightforward computations to see that
Eqs.~\eqref{eq:SpringConnection}~and~\eqref{eq:SpringConnection2} are
equivalent; see the proof of Corollary~\ref{thm:SpringConnectionCor}
where the invertibility of $K$ is not assumed.

Claim~\eqref{thm:SpringConnectionClaim2}: By
Proposition~\ref{ImpPassiveProp} and the invertibility of $K$ and $M$,
the system $\Sigma_i$ is impedance passive if and only if
\begin{equation} \label{eq:SpringConnection3}
  \bbm{\underbrace{\bbm{0}}_{m \times m} & \underbrace{\bbm{0}}_{m \times m} &  Q_1^T \\
    \underbrace{\bbm{0}}_{m \times m} &  P  &  Q_2^T -  \frac{1}{2} F  \\
    Q_1  &      Q_2 -  \frac{1}{2} F^T   & \underbrace{\bbm{0}}_{m \times m}}
  \geq \underbrace{\bbm{0}}_{3 m \times 3 m}.
\end{equation}
 To prove the nontrivial direction, assume that $\Sigma_i$ is
 impedance passive. Then both $P \geq 0$ and $\sbm{ P & T \\ T^T & 0 }
 \geq 0$ where $T := Q_2^T -\frac{1}{2}
 F$. Lemma~\ref{lemma:SpringConnection} implies that $T = 0$, i.e.,
 $Q_2 = \frac{1}{2} F^T$. Eq.~\eqref{eq:SpringConnection3} with $T =
 0$ implies $\sbm{ 0 & Q_1^T \\ \ Q_1 & 0} \geq 0$, and $Q_1 = 0$
 follows from Lemma~\ref{lemma:SpringConnection}.

 Claim~\eqref{thm:SpringConnectionClaim3}: This can be directly seen
 from Eq.~\eqref{eq:SpringConnection3} which is satisfied as an
 equality in the impedance conservative case by
 Proposition~\ref{ImpPassiveProp}.
\end{proof}

In many cases, the stiffness matrix $K \geq 0$ in
Eq.~\eqref{eq:SpringConnection} fails to be injective even though the
mass matrix $M > 0$ is invertible.  The system in
Section~\ref{FEMSubSec} is an example of this since the acoustic
velocity potential is defined only up to an additive constant, and
nothing in the observed physics depends on such a constant. Note that
if $Q_1 = 0$ (a necessary condition for passivity), there is nothing
in Eq.~\eqref{eq:SpringConnection2} that requires invertibility of
$K$. This motivates a variant of Theorem~\ref{thm:SpringConnection}:
\begin{cor} \label{thm:SpringConnectionCor}
  Let $M$, $P$, and $K$ be symmetric positive definite $m \times m$
  matrices, and assume that $M$ is invertible. Let $F$ be a $m \times k$
  matrix and $u(\cdot) \in C([0,\infty); \R^k)$. Then the
    following holds:
    \begin{enumerate}
    \item \label{thm:SpringConnectionCorClaim1}
      The linear system  $\Sigma_i$
      associated to the differential equation
      \begin{equation} \label{eq:SpringConnectionCor2}
    \begin{cases}
      x'(t) & = \bbm{0 & K^{1/2}M^{-1/2 }\\ -M^{-1/2 }K^{1/2} & - M^{-1/2 }PM^{-1/2}} x(t)
      +  \bbm{0\\  M^{-1/2 } F}u(t)\\
      y(t)  & =    \bbm{0 &  F^T M^{-1/2 } } x(t), \quad t \geq 0,
    \end{cases}
      \end{equation}
      is impedance passive.
    \item \label{thm:SpringConnectionCorClaim2} Let $\alpha \in \Range{K}$
      and $\beta \in \R^m$ be arbitrary.  Then the function $x \in
      C^1([0,\infty);\R^{2m})$, $x(t) = \sbm{x_1(t) \\ x_2(t)}$, satisfies
        Eq.~\eqref{eq:SpringConnectionCor2} with $x(0) = \sbm{\alpha
          \\ \beta}$ if and only if the function $z \in
        C^2([0,\infty);R^m)$, given by
          \begin{equation} \label{eq:SpringConnectionCor3}
            z(t) =  \int_0^t {M^{-1/2}x_2(\tau) \, d \tau} + \gamma, \quad  t \geq 0,
          \end{equation}
          satisfies $x_1(t) = K^{1/2} z(t)$ and
          \begin{equation} \label{eq:SpringConnectionCor4}
            \begin{cases}
              & M z''(t) + P z'(t) + K z(t) = F u(t),  \\
              & y(t) = F^T z'(t), \quad t \geq 0, \quad \text{and} \\
              & z(0) = \gamma, \quad  z'(0) = M^{-1/2} \beta
            \end{cases}
          \end{equation}
          for some (hence, for all)  $\gamma \in \R^m$ satisfying $K^{1/2}
          \gamma = \alpha$.
    \end{enumerate}
\end{cor}
\noindent
Note that Eq.~\eqref{eq:SpringConnectionCor2} is equivalent with
Eq.~\eqref{eq:SpringConnection2} if $Q_1 = 0$ and $Q_2 = \tfrac{1}{2}
F^T$ apart from the normalisations by $1/\sqrt{2}$. In particular, the
transfer functions given by
Eqs.~\eqref{eq:SpringConnection2}~and~\eqref{eq:SpringConnectionCor2}
are the same. To simulate the input/output behaviour of these systems,
it is possible to use $\alpha = \beta = \gamma = 0$.
\begin{proof}
Claim~\eqref{thm:SpringConnectionCorClaim1} follows the same way as
Claim~\eqref{thm:SpringConnectionClaim2} of
Proposition~\ref{ImpPassiveProp} since nothing in the proof depends on
the invertibility of $K$. For the rest of the proof, fix $\alpha \in
\Range{K} = \Range{K^{1/2}}$ and $\beta \in \R^m$, and let $\gamma$ be
arbitrary such that $K^{1/2} \gamma = \alpha$ holds.

Claim~\eqref{thm:SpringConnectionCorClaim2}, necessity:   From
Eq.~\eqref{eq:SpringConnectionCor3} we get $x_2 = M^{1/2} z'$. Since
$x_1 = K^{1/2} z$, it follows that
\begin{equation*}
  \begin{aligned}
    & \bbm{0 & K^{1/2}M^{-1/2 }\\ -M^{-1/2 }K^{1/2} & - M^{-1/2 }PM^{-1/2}} \bbm{x_1(t) \\ x_2(t)}
    +  \bbm{0\\  M^{-1/2 } F}u(t) \\
    & = \bbm{K^{1/2}M^{-1/2 } x_2(t) \\ -M^{-1/2 }K^{1/2} x_1(t) - M^{-1/2 }PM^{-1/2} x_2(t) + M^{-1/2 } F u(t)} \\
    & = \bbm{K^{1/2} z'(t) \\ M^{-1/2 } \left (-K z(t) - P z'(t) +  F u(t) \right )}
    = \bbm{K^{1/2} z'(t) \\ M^{1/2 } z''(t) } = \bbm{x_1'(t) \\ x_2'(t)}
  \end{aligned}
\end{equation*}
where we used the assumption that Eq.~\eqref{eq:SpringConnectionCor4}
holds.  Obviously, $\bbm{0 & F^T M^{-1/2 } } x = F^T M^{-1/2 } x_2 =
F^T z' = y$ and $x(0) = \sbm{K^{1/2} z(0) \\ M^{1/2} z'(0)} =
\sbm{K^{1/2} \gamma \\ M^{1/2} M^{-1/2}\beta} = \sbm{\alpha \\ \beta}$
from which the necessity part follows.

Claim~\eqref{thm:SpringConnectionCorClaim2}, sufficiency: Assume that
$x(t) = \sbm{x_1(t) \\ x_2(t)}$ satisfies
Eq.~\eqref{eq:SpringConnectionCor2} with the initial condition $x(0) =
\sbm{\alpha \\ \beta}$, and define $z$ by
Eq.~\eqref{eq:SpringConnectionCor3} satisfying $z' = M^{-1/2} x_2$ as
well as the initial conditions $z(0) = \gamma$ and $z'(0) = M^{-1/2}
\beta$.  From the top row of the first equation in
Eq.~\eqref{eq:SpringConnectionCor2} we conclude that $x_1' = K^{1/2}
M^{-1/2} z_2 = K^{1/2} z'$, and thus $x_1 = K^{1/2} z + \delta$ for
some $\delta \in \R^{m}$. Now, $\alpha = x_1(0) = K^{1/2} z(0) +
\delta = K^{1/2} \gamma + \delta = \alpha + \delta$, and hence $\delta
= 0$. We have now concluded that $x = \sbm{x_1 \\x_2} = \sbm{K^{1/2} z
  \\ M^{1/2} z'}$, and the differential equation in
Eq.~\eqref{eq:SpringConnectionCor4} follows from
Eq.~\eqref{eq:SpringConnectionCor2} by the computation given in the
necessity part of this claim. Since the observation equations in
Eqs.~\eqref{eq:SpringConnectionCor2}~and~\eqref{eq:SpringConnectionCor4}
are equivalent, the proof is complete.
\end{proof}
It remains to give a technical lemma for the proof of
Theorem~\ref{thm:SpringConnection}:
\begin{lemma} \label{lemma:SpringConnection}
  Let $P$ and $U$ be $m \times m$ matrices with $P \geq 0$. Then
  $\sbm{ P & U \\ U^T & 0 } \geq 0$ if and only if $U = 0$.
\end{lemma}
\begin{proof}
  Only the ``only if'' part requires proof. By a change of basis, we
  may assume without loss of generality that $U = \bbm{U_0 & 0}$ where
  $m_1 := \mathop{rank}(U) > 0$ and $U_0 = \bbm{\vec{u}_{1} & \ldots &
    \vec{u}_{m}}^T$ is an injective $m \times m_1$ matrix.

  Now, removing some of the row vectors $\vec{u}_{j_k} \in \R^{m_1}$
  of $U_0$ for $k = 1, \ldots, m - m_1$ from $U_0$, we get a $m_1
  \times m_1$ \emph{invertible} matrix $\tilde U$. Similarly removing
  $j_k$'th column and row vectors from $P$ for all $k = 1, \ldots, m -
  m_1$ we obtain $m_1 \times m_1$ matrix $\tilde P \geq 0$ so that $Q
  := \sbm{\tilde P & \tilde U \\ \tilde U^T & 0 }$ is a compression of
  $\sbm{ P & U \\ U^T & 0 }$ into a $2 m_1$-dimensional subspace with
  the additional property that $\tilde U$ is invertible.

  By the Schur complement, we observe that for any $\varepsilon \neq
  0$ $\mathop{det}{\sbm{\tilde P & \tilde U \\ \tilde U^T &
      \varepsilon I }} = \mathop{det}{\left ( \varepsilon \tilde P -
    \tilde U \tilde U^T \right )}$, and letting $\varepsilon \to 0$
  implies $\mathop{det}{Q} = \mathop{det}{\left ( - \tilde U \tilde
    U^T \right )} = (-1)^{m_1} \mathop{det}{\left (\tilde U \tilde U^T
    \right )}$ where $\mathop{det}{\left (\tilde U \tilde U^T \right
    )}> 0$ by positivity and invertibility.  If $m_1$ is odd, it
  directly follows that $\mathop{det}{Q} < 0$ which contradicts $Q
  \geq 0$ and, hence, $\sbm{ P & U \\ U^T & 0 }> 0$.

  It remains to consider the case when $m_1$ is even. Because
  $\mathop{det}(\tilde U) \neq 0$, some of its $(i,j)$ minors for $1
  \leq i,j \leq m_1$ is nonvanishing. Define the invertible $(m_1 - 1)
  \times (m_1 - 1)$ matrix $\hat U$ by removing the $i$'th row and
  $j$'th column from $\tilde U$.  Further, define $\hat P \geq 0$ by
  removing the $i$'th row and column from $\tilde P$. Thus, the matrix
  $\hat Q := \sbm{\hat P & \hat U \\ \hat U^T & 0 }$ is a compression
  of $Q$ into $2(m_1 - 1)$-dimensional subspace where $m_1 - 1$ is now
  odd. The above argument shows that $\hat Q \geq 0$ does not hold,
  and hence the same holds for $Q$ and $\sbm{ P & U \\ U^T & 0 }$,
  too.
\end{proof}

Another important class are \emph{scattering passive} systems.  System
$\Sigma = \sbm{ A&B \\ C & D }$ is scattering passive if the signals
in Eq.~\eqref{eq:systemDyn} satisfy the energy inequality
\begin{equation} \label{eq:ScatteringEnergyInEq}
  \frac{d}{dt} \norm{ x(t) }^2 \leq \norm{u(t)}^2 - \norm{y(t)}^2 \quad \text{for all} \quad t \in I.
\end{equation}
If this inequality is satisfied as an equality, then $\Sigma$ is
\emph{scattering conservative}.

A simple characterisation of scattering passivity in terms of LMI's
such as Eq.~\eqref{ImpPassiveProp} does not exists. However, a
formulation involving the resolvent of $A$ is given in
\cite[Theorem~11.1.5]{OS:WPLS} or
\cite[Proposition~5.2]{M-S-W:HTCCS}; see also
Proposition~\ref{ScatteringPassiveProp} below. However, scattering
conservative finite-dimensional systems can be characterised
concisely:
\begin{prop} \label{ScatteringConservativeProp}
  A linear system $\Sigma = \sbm{ A &B \\ C & D}$ is scattering
  conservative if and only if
\begin{equation*}
   A + A^T = - C^T C = - B B^T, \quad C = - D B^T \quad \text{and}
   \quad D^T D = I.
\end{equation*}
\end{prop}
\noindent This follows from \cite[Eqs.~(1.4)--(1.5)]{M-S-W:HTCCS}.

We occasionally need also discrete time systems as time discretised
versions of $\Sigma$ as in Section~\ref{ResultsSubSec}. Also these
systems are defined in terms of quadruples of matrices $\phi =
\spm{A_d &B_d \\ C_d & D_d}$ associated with the difference equations
\begin{equation}
  \label{eq:systemDynDisc}
  \begin{aligned}
    & \begin{cases}
        x_{j+1} & = A_d x_j +  B_d u_j, \\
        y_j & = C_d x_j +   D_d u_j,
      \end{cases}
\quad \text{or, briefly,} \quad \bbm{x_{j + 1} \\ y_j } = \bbm{A_d &B_d \\ C_d &
  D_d} \bbm{x_j \\ u_j } \\
&  \quad \quad \quad \text{for all } j \in J \subset \{ \ldots -1, 0, 1, 2, \ldots \}.
  \end{aligned}
\end{equation}
We call system $\phi$ \emph{discrete time scattering passive} if
\begin{equation} \label{eq:DiscreteTimeEnergyBalance}
  \norm{ x_{j+1} }^2 -   \norm{ x_{j} }^2 \leq \norm{ u_{j} }^2 - \norm{ y_{j} }^2 \quad \text{for all } j \in J,
\end{equation}
and \emph{discrete time impedance passive} if
\begin{equation} \label{eq:DiscreteTimeEnergyBalanceImp}
  \norm{ x_{j+1} }^2 - \norm{ x_{j} }^2 \leq 2 \left < u_{j}, y_{j}
  \right > \quad \text{for all } j \in J.
\end{equation}
Moreover, such $\phi$ is \emph{scattering [impedance] conservative} if
the respective inequality is satisfied as an equality.
\begin{prop} \label{ScatteringPassivePropDiscrete}
  Let $\phi = \spm{A_d &B_d \\ C_d & D_d}$ a discrete time linear
  system.  Then $\phi$ is scattering passive if and only if
\begin{equation} \label{eq:ScatteringPassivePropDiscrete}
  \bbm{A_d &B_d \\ C_d & D_d}^T \bbm{A_d &B_d \\ C_d & D_d} \leq
\bbm{I & 0 \\ 0 & I}.
\end{equation}
Similarly, $\phi$ is impedance passive if and only if
\begin{equation} \\ \label{eq:ImpedancePassivePropDiscrete}
  \bbm{I - A_d^T A_d & C_d^T - A_d^T B_d \\
    C_d - B_d^T A_d   & D_d + D_d^T - B_d^T B_d } \geq
\bbm{0 & 0 \\ 0 & 0}.
\end{equation}
The system is scattering or impedance conservative if and only if the
respective inequality holds as an equality.
\end{prop}
\noindent Indeed, 
Eqs.~\eqref{eq:ScatteringPassivePropDiscrete}~and~\eqref{eq:ImpedancePassivePropDiscrete}
are equivalent with
Eq.~\eqref{eq:DiscreteTimeEnergyBalance}~and~\eqref{eq:DiscreteTimeEnergyBalanceImp}, 
respectively, by Eq.~\eqref{eq:systemDynDisc}.

We can now give a characterisation for passive continuous time systems
$\Sigma = \sbm{A & B \\ C & D}$ in terms of discrete time systems:
\begin{prop} \label{ScatteringPassiveProp}
  Let $\Sigma = \sbm{ A &B \\ C & D}$ a linear system whose internal
  Cayley transform defined in Section~\ref{InternalCayleySec} below is
  denoted by $\phi_\sigma = \spm{A_\sigma & B_\sigma \\ C_\sigma &
    D_\sigma}$ for $\sigma > 0$.  Then the following are equivalent:
  \begin{enumerate}
  \item \label{ScatteringPassivePropClaim1}
    $\Sigma$ is scattering [impedance] passive;
  \item \label{ScatteringPassivePropClaim2}
    $\phi_\sigma$ is discrete time scattering [impedance] passive for some $\sigma > 0$; and
  \item \label{ScatteringPassivePropClaim3} $\phi_\sigma$ is discrete
    time scattering [impedance] passive for all $\sigma > 0$.
     \end{enumerate}
The equivalences remain true if the word ``passive'' is replaced by ``conservative''.
\end{prop}
\noindent The scattering passive part is a finite-dimensional special
case of \cite[Proposition~5.2]{M-S-W:HTCCS} or
\cite[Theorem~3.3(v)]{Staffans:PCCT:2002}. For the impedance passive
part, see \cite[Theorem~4.2(v)]{Staffans:PCCT:2002} or use
Proposition~\ref{ImpPassiveProp} with the identity
\begin{equation*}
\begin{aligned}
 - & \bbm{\sqrt{2 \sigma }\left ( \sigma - A^T \right )^{-1} & 0 \\ B^T \left ( \sigma - A^T \right )^{-1} & I}
\bbm{A +A^T &  B - C^T   \\ B^T - C   & - D - D^T}
 \bbm{\sqrt{2 \sigma }\left ( \sigma - A \right )^{-1} &  \left ( \sigma - A \right )^{-1}B \\ 0 & I} \\
= & \bbm{I - A_\sigma^T A_\sigma & C_\sigma^T - A_\sigma^T B_\sigma \\
  C_\sigma - B_\sigma^T A_\sigma & D_\sigma + D_\sigma^T - B_\sigma^T B_\sigma}\quad\text{for all } \sigma > 0.
\end{aligned}
\end{equation*}

\section{\label{SystemSection} Transfer functions, realisations, and signals}

As discussed in the introduction, practical applications may require
treating time-domain and frequency-domain models in a same
framework. Passive linear systems $\Sigma = \sbm{A & B \\C & D}$ were
reviewed in time domain in Section~\ref{BackgroundSec}, and it remains
to give the frequency-domain description in terms of their transfer
functions
\begin{equation} \label{eq:ContinuousTimeTF}
  \CTF_{\Sigma}(s) := D + C \left (s - A \right )^{-1} B \quad \text{for } s \notin \sigma(A).
\end{equation}
Given a matrix-valued rational function $g(s)$, we call $\Sigma$ a
\emph{realisation} of $g(s)$ if $g(s) = \CTF_{\Sigma}(s)$ for
infinitely many $s \in \C$.  Manipulating realisations is one way of
carrying out computations on rational functions (such as required by
feedback systems analysis) in terms of matrix computations. 

The first step is to describe the associative, typically
non-commutative algebra (with unit) of rational transfer functions
$\CTF_{\Sigma}(s)$ where the addition, scalar multiplication, and
product of elements stand out as the elementary operations.  Of
course, the input/output signal dimension $m$ of system $\Sigma$ must
be the same within such algebra for these operation to be universally
feasible. Since the same rational transfer function has an infinite
number of realisation, of which only some are controllable and
observable (i.e., have a minimal state space), the algebraic structure
must be described in terms of equivalence classes as described in
Section~\ref{RealAlgSection} below. Symbolic and numerical
computations must be carried out in terms of representatives of these
classes, using specific formulas and system diagrams for realisations
introduced in Sections~\ref{DiagrammeSec}~and~\ref{FundamentalSec}.

\subsection{\label{RealAlgSection} Realisation algebra}

Rational $m \times m$ matrix-valued rational functions constitute an
algebra that can be described in terms of equivalence classes of
realisations.  This provides us a way of carrying out practical
computations with transfer functions by using numerical linear algebra
on their conveniently chosen realisations.  We proceed to give a
description of the rules of calculation involved.

We denote by $\Sigma = \sbm{A & B \\ C & D}$, $\Sigma_p =
\sbm{A_p&B_p\\C_p&D_p}$, and $\Sigma_q= \sbm{A_q&B_q\\C_q&D_q}$ linear
systems with $m$-dimensional input and output signals.  All the
matrices are assumed to be real, and the transfer functions of
$\Sigma_p$ and $\Sigma_q$ are given by $\CTF_p(s) =D_p +
C_p(sI-A_p)^{-1}B_p$ and $\CTF_q(s) =D_q + C_q(sI-A_q)^{-1}B_q$ where
we take the liberty of using complex valued $s$.
\begin{defn}
  \label{EquivalenceClassDef}
  \begin{enumerate}
  \item Linear systems $\Sigma_p$ and $\Sigma_q$ are \emph{I/O
    equivalent} if $\CTF_p(s) = \CTF_q(s)$ for infinitely many $s \in
    \C$.  I/O equivalence is denoted by $\Sigma_p \sim \Sigma_q$.
  \item The equivalence class containing $\Sigma_p$ is denoted by
    $\left[ \Sigma_p \right] := \left\lbrace\Sigma \, : \,
    \Sigma\sim\Sigma_p \right\rbrace$.
  \item For $m \geq 1$, we denote 
    \begin{equation*}
      \aleph_m := \{ \left [\Sigma \right ] \,
    : \, \text{The input and output signals of } \Sigma \text{ are } m
    \text{ dimensional } \}.
    \end{equation*}
  \end{enumerate}
\end{defn}
Thus, the equivalence class $\left[ \Sigma_p \right]$ and the transfer
function $\CTF_p(s)$ are in one-to-one correspondence.  In any
nontrivial equivalence class, say $\left[ \Sigma' \right]$, there are
infinitely many systems $\Sigma \in \left[ \Sigma' \right]$ that are
\emph{minimal} in the sense that they are observable and controllable;
i.e.,
\begin{equation*}
  \mathop{rank}{\bbm{B & A B & \ldots A^{n-1} B}} =
  \mathop{rank}{\bbm{C^* & A^* C^* & \ldots A^{*(n-1)} C^*}} = n
\end{equation*}
where $A$ is a $n \times n$ matrix.  Given a transfer function
$\CTF_\Sigma(\cdot)$ of a minimal $\Sigma$, the number $n$ is called
the \emph{McMillan degree} of $\CTF_\Sigma(\cdot)$. It is well known
that two minimal systems $\Sigma_p = \sbm{A_p&B_p\\C_p&D_p}$ and
$\Sigma_q= \sbm{A_q&B_q\\C_q&D_q}$ having the same transfer function
are \emph{state space isomorphic} in the sense that $A_p = T^{-1} A_q
T$, $B_p = T^{-1} B_q$, $C_p = C_q T$, and $D_p = D_q$ for some
invertible matrix $T$. A non-minimal system $\Sigma$ can always be
reduced to some minimal system by standard linear algebra means. For
these facts, see any classical text on algebraic control theory such
as \cite{PAF:LSOHS} and \cite{K-F-A:TMST}.
Minimisation and state space isomorphism typically changes all of the
original matrices $A$, $B$, and $C$ (that may bear resemblance to,
e.g., the physical parameters of the problem) to an unrecognisable
form which diminishes the appeal of them in applications.
\begin{defn}
  \label{BasicSystemOpsDef}
  Let $\Sigma_p$ and $\Sigma_q$ be linear systems with $m$-dimensional
  state space.
  \begin{enumerate}
  \item For any $c \in \R$ the \emph{scalar multiple} of $\Sigma_p$
    is
    \begin{equation*}
      \label{eq:scalar}
      c\Sigma_p = \bbm{A_p& B_p \\cC_p & cD_p}. 
    \end{equation*} 

  \item The \emph{parallel sum} $+$ for $\Sigma_p$ and $\Sigma_q$ is 
    \begin{equation*}  
      \Sigma_p + \Sigma_q =\
      \bbm{\begin{matrix} A_p & 0 \\ 0 & A_q\end{matrix} &
        \begin{matrix}B_p\\ B_q \end{matrix} \\
      \begin{matrix} C_p &  C_q\end{matrix} & D_p+D_q}.
    \end{equation*}
  \item The \emph{cascade product} $*$ for $\Sigma_p$ and $\Sigma_q$ is 
    \begin{equation*}
      \Sigma_p * \Sigma_q 
      = \bbm{
        \begin{matrix} A_p & B_pC_q \\ 0 & A_q \end{matrix} &
        \begin{matrix} B_p D_q \\ B_q \end{matrix} \\
        \begin{matrix} C_p & D_p C_q \end{matrix} &
        \begin{matrix} D_p D_q \end{matrix}
      }.
    \end{equation*}
  \end{enumerate}
\end{defn}
If both $\Sigma_p$ and $\Sigma_q$ are minimal, then so are $\Sigma_p +
\Sigma_q$ and $c\Sigma_p$. However, the system $\Sigma_p * \Sigma_q$
can fail to be minimal since zero/pole cancellations may take place in
the product of transfer functions.  Obviously, $\CTF_{\Sigma_p *
  \Sigma_q}(s) = \CTF_{\Sigma_p}(s) \CTF_{\Sigma_q}(s)$,
$\CTF_{\Sigma_p + \Sigma_q}(s) = \CTF_{\Sigma_p}(s) +
\CTF_{\Sigma_q}(s)$, and $\CTF_{c \Sigma_p}(s) = c\CTF_{\Sigma_p}(s)$
for all but a finite number of $s \in \C$.  Hence, the operation in
Definition~\ref{BasicSystemOpsDef} can be extended to any $S_m$ by
setting
\begin{equation} \label{eq:MathOpsClasses}
  c \left [ \Sigma \right ] := \left [ c \Sigma \right ], \quad 
  \left [ \Sigma_p \right ] + \left [ \Sigma_q \right ] := 
  \left [ \Sigma_p + \Sigma_q \right ], \text{ and } 
  \left [ \Sigma_p \right ] * \left [ \Sigma_q \right ] := 
  \left [ \Sigma_p * \Sigma_q \right ]. \quad 
\end{equation}
\begin{prop}
  Let $m \geq 1$ and $\aleph_m$ as in
  Definition~\ref{EquivalenceClassDef}.  Equipped with the operations
  in Eq.~\eqref{eq:MathOpsClasses}, the set $\aleph_m$ becomes an
  associative algebra over $\R$ with unit (i.e., $(\aleph_m, + )$ is a
  real vector space, and $(\aleph_m, +, * )$ is a ring.)
\end{prop}
\noindent The required properties can be checked either directly from
Eq.~\eqref{eq:MathOpsClasses} in terms of elements of equivalence
classes, or by observing the one-to-one correspondence with rational
matrix-valued transfer functions known to be an algebra.

\subsection{\label{DiagrammeSec} System diagrams and splitted signals}

We proceed by splitting the $m$-dimensional input and the output
signals of $\Sigma = \sbm{A& B \\ C & D}$. The need for such splitting
arises from the fact that we consider the system $\Sigma$ represent a
\emph{two-port} of \emph{four-pole} following~\cite{HK:CSAH}. Since
the objective is to couple such systems in various ways, rules for
drawing \emph{system diagrams} are proposed.\footnote{The rules for
  drawing the system diagrams differ from those used in \cite{HK:CSAH}
  to emphasise the directions of the signals consistently with their
  dynamical equations.} It appears that much of the proofs for results
on couplings and feedbacks can be carried out simply by considering
diagrams.

\begin{figure}[h]
  \centering
  \includegraphics[scale=1.5]{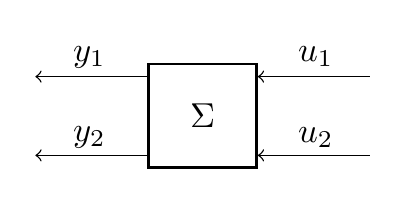}
  \includegraphics[scale=1.5]{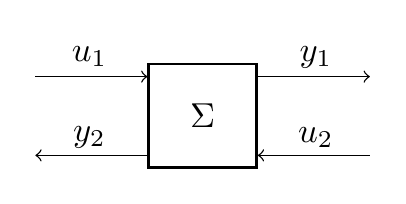}
  \includegraphics[scale=1.5]{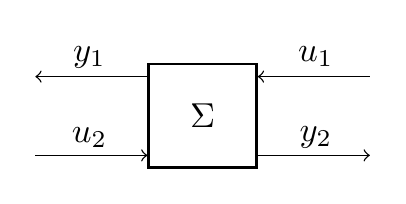}
  \includegraphics[scale=1.5]{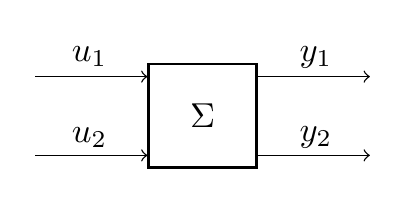}
  \caption[Equivalent systems]{Four equivalent diagrams describing the same system
    $\Sigma = \sbm{A&B \\ C& D}$ associated to Eq.~\eqref{eq:systemDyn}.
    The first of the diagrams is drawn in  \emph{standard form} as
    defined in the text.}
  \label{fig:SystemDiagram}
\end{figure}

\begin{sta}\label{SecondSA}
  We assume that the common dimension $m$ of the input and output
  signals $u(\cdot)$, $y(\cdot)$ in Eq.~\eqref{eq:systemDyn} is even,
  and the signals are splitted
  \begin{equation*}
    u(t) = \bbm{u_1(t) \\ u_2(t)} \quad \text{and} \quad y(t) =
    \bbm{y_1(t) \\ y_2(t)}
  \end{equation*}
  where each of the signals $u_1(\cdot)$, $u_2(\cdot)$, $y_1(\cdot)$,
  and $y_2(\cdot)$ are column vector valued of the same dimension $m/2$.
\end{sta}

The behaviour described by Eq.~\eqref{eq:systemDyn} can be illustrated
in terms of \emph{system diagrams} shown in
Fig.~\ref{fig:SystemDiagram}. Each of the four signals $u_1(\cdot)$,
$u_2(\cdot)$, $y_1(\cdot)$, and $y_2(\cdot)$ in these diagrams has two
mathematical properties: a signal is either \textrm{(i)} \emph{input}
or \emph{output} signal, and either \textrm{(ii)} \emph{top} or
\emph{bottom} signal. That a signal is input is indicated by the arrow
pointing to the frame in diagram.  Otherwise, the signal is output of
the system. That a signal is top is indicated by drawing it to top row
of the diagram. We say that the diagram is in \emph{standard form}
when the signals $u_1(\cdot)$, $u_2(\cdot)$, $y_1(\cdot)$, and
$y_2(\cdot)$ are in the same relative positions as they appear in
their dynamical Eqs.~\eqref{eq:systemDyn}. The diagrams in
Fig.~\ref{fig:SystemDiagram} describe the same dynamical system, and
the top left panel describes it in standard form. The following
\emph{coupling rule} is a restriction for coupling in system diagrams:
two inputs or two outputs cannot be coupled.\footnote{In
  Section~\ref{RedhefferSec} the \emph{colour} (in fact, red or blue)
  of signals is also introduced together with the \emph{colour
    rule}. Colour is a semantic property that describes the underlying
  physics of the realisations.}

The system diagrams do not assume even the linearity of the underlying
dynamical systems. Extremely complicated networks of dynamical systems
can be described in terms of system diagrams ; see, e.g.,
Fig~\ref{fig:RedhefferSystem}.  The parallel and cascade connections
of Definition~\ref{BasicSystemOpsDef} are shown in
Fig.~\ref{fig:SystemDiagramPasCas}.

\begin{figure}[h]
  \centering
  \raisebox{-0.5\height}{\includegraphics[scale=1.3]{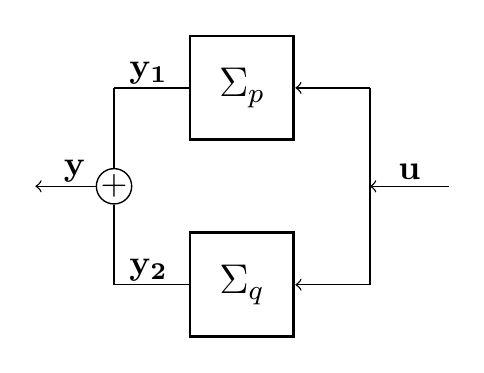}}
  \raisebox{-0.5\height}{\includegraphics[scale=1.3]{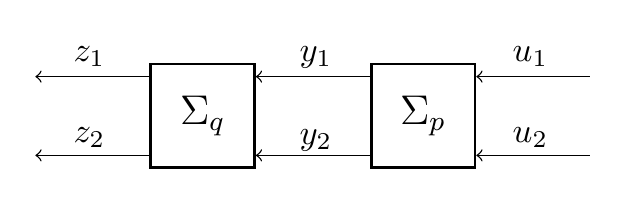}}
  \caption{Parallel and cascade connections in terms of system
    diagrams drawn in standard form. The resulting realisations are
    denoted by $\Sigma_p + \Sigma_q$ and $\Sigma_p * \Sigma_q$ with
    formulas given in Definition~\ref{BasicSystemOpsDef}.  }
  \label{fig:SystemDiagramPasCas}
\end{figure}

\subsection{\label{FundamentalSec} Fundamental operations of realisations}

\begin{figure}[h]
  \centering
  \includegraphics[scale=1.5]{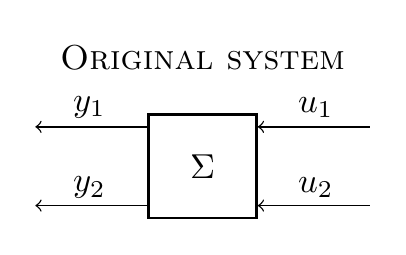}
  \includegraphics[scale=1.5]{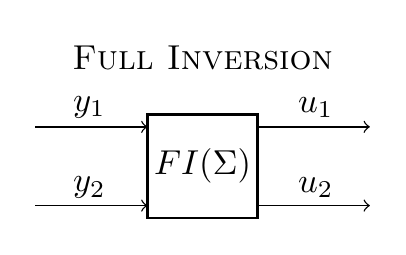}
  \includegraphics[scale=1.2]{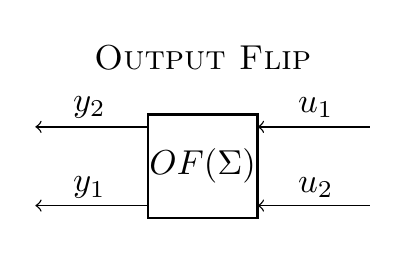}
  \includegraphics[scale=1.2]{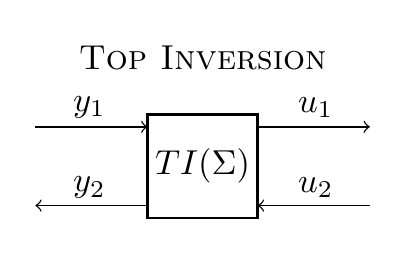}
  \includegraphics[scale=1.2]{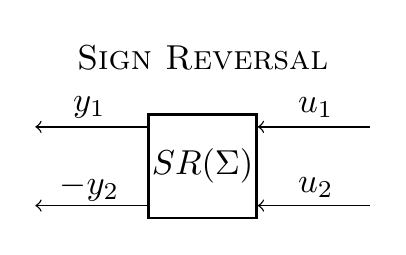}
  \caption[Transformed systems]{The representations of the three basic
    transformations of the original system $\Sigma = \sbm{A&B \\ C&
      D}$ given in its standard diagrammatic form (top left). The
    mathematical relations between the signals $u_1, u_2, y_1$, and
    $y_2$ in all of these systems $\Sigma$, $\OF(\Sigma)$,
    $\TI(\Sigma)$, and $\SR(\Sigma)$ are \emph{by definition} the same.}
  \label{fig:SystemTransformationDiagram}
\end{figure}

It appears that all couplings and feedback configurations required in
this article are combinations of four elementary transformations of
the original system $\Sigma = \sbm{A&B \\ C& D}$ described
diagrammatically in Fig.~\ref{fig:SystemTransformationDiagram}. In
terms of the state space representation of the linear system $\Sigma$,
these transformations are as follows:
\begin{enumerate}
\item {\bf Full Inversion ($\FI$)}
     \begin{equation}
      \label{eq:invsystem}
    \bbm{A&B \\ C& D} = \Sigma \mapsto \FI(\Sigma) := \bbm{A-BD^{-1}C & BD^{-1}
      \\ -D^{-1}C & D^{-1}}.
     \end{equation}
\item {\bf Output Flip ($\OF$)}
    \[
    \bbm{A&B \\ \sbm{C_1 \\ C_2} & \sbm{D_1\\ D_2}} = \Sigma \mapsto \OF(\Sigma) := \bbm{A& B
      \\ \sbm{0&I\\I&0 } \sbm{C_1 \\ C_2} & \sbm{0&I\\I&0 } \sbm{D_1\\ D_2}}.
  \]
\item {\bf Top Inversion ($\TI$)}

  \[
    \bbm{A & \sbm{B_{1} & B_{2}} \\ \sbm{C_1 \\ C_2} & \sbm{D_{{11}} & D_{{12}} \\
        D_{{21}} & D_{{22}}} }  = \Sigma \mapsto \TI(\Sigma) :=
    \bbm{A-B_{1}D_{{11}}^{-1}C_{1} & \sbm{-B_{1}D_{{11}}^{-1}  & B_2 - B_{1}D_{{11}}^{-1}D_{12}}\\
      \sbm{-D_{{11}}^{-1}C_{1} \\C_{2}-D_{{21}}D_{{1}}^{-1}C_{1}} &\sbm{-D_{{11}}^{-1}& -D_{{11}}^{-1} D_{{12}} \\
        -D_{{12}}D_{{11}}^{-1}&D_{{22}}- D_{{21}}D_{{11}}^{-1}D_{{12}}}}.
  \]

\item{{\bf Sign Reversal ($\SR$)}}
    \[
    \bbm{A&B \\ \sbm{C_1 \\ C_2} & \sbm{D_1\\ D_2}}
    = \Sigma  \mapsto \SR(\Sigma) := \bbm{A& B \\ \sbm{I& 0\\0 & -I } \sbm{C_1 \\ C_2} & \sbm{I & 0\\0 & -I } \sbm{D_1\\ D_2}}.
  \]
\end{enumerate}
The realisation formula for $\FI(\Sigma)$ is called \emph{external
  reciprocal transformation} in Section~\ref{ERTSec}. It is the only
one of the four transformations that does not require the splitting of
inputs to components $u_1, u_2$ or outputs to $y_1, y_2$. Observe that
both $\FI(\Sigma)$ and $\TI(\Sigma)$ are only defined for $\Sigma$ for
which the result is well-defined.

The compositions of these operations on realisations are denoted by
$\circ$.  Obviously, $\Sigma = \OF(\OF(\Sigma)) = \TI(\TI(\Sigma)) =
\SR(\SR(\Sigma))$; i.e., $\OF^{-1} = \OF$, $\TI^{-1} = \TI$, and $\SR^{-1} =
\SR$.  Two further similar operations can be defined in terms of these,
namely the \emph{Bottom Inversion} defined as $\BI := \TI \circ
\FI = \FI \circ \TI$ whose realisation formula is plainly
\begin{equation*}
  \BI(\Sigma) =  \bbm{A-B_{2}D_{{22}}^{-1}C_{2}  & \bbm{B_{1}-B_{2}D_{{22}}^{-1}D_{{21}}  & -B_{2}D_{{22}}^{-1}} \\
                     \bbm{C_{1}-D_{{21}}D_{{22}}^{-1}C_{2} \\ -D_{{22}}^{-1}C_{2} } &
                     \bbm{D_{{11}}- D_{{12}}D_{{22}}^{-1}D_{{21}} & -D_{{12}}D_{{22}}^{-1} \\ -D_{{22}}^{-1}D_{{21}} &-D_{{22}}^{-1}} \\
}.
\end{equation*}
Clearly, $\FI = \TI \circ \BI = \BI \circ \TI$.  The \emph{Input Flip}
given by $\IF := \FI \circ \OF \circ \FI$ has the realisation formula
\begin{equation*}
    \bbm{A&B \\ C& D} = \Sigma  \mapsto \IF(\Sigma) := \bbm{A& \sbm{B_1 & B_2} \sbm{0&I\\I&0 } \\ C& \sbm{D_1 & D_2} \sbm{0&I\\I&0 }}.
\end{equation*}

\section{\label{TransformSec} Transformations of realisations}

\subsection{\label{InternalCayleySec} Internal Cayley and reciprocal transformations}

Let $\Sigma = \sbm{A & B \\ C & D}$ be any continuous time system. We
define for any $\sigma > 0$ the matrices
\begin{equation} \label{eq:IntCayleyDef}
\begin{aligned}
  A_\sigma & := \left (\sigma + A \right )\left (\sigma - A \right )^{-1}, \quad 
  B_\sigma  := \sqrt{2 \sigma} \left (\sigma - A \right )^{-1} B, \\
  C_\sigma & :=  \sqrt{2 \sigma} \left (\sigma - A \right )^{-1} C, \quad 
  D_\sigma  :=  D + C \left (\sigma - A \right )^{-1} B = \CTF_{\Sigma}(\sigma)
\end{aligned}
\end{equation}
where $\CTF_{\Sigma}(\cdot)$ is the transfer function of $\Sigma$.
The discrete time system $\phi_\sigma = \spm{A_\sigma & B_\sigma \\ C_\sigma & D_\sigma}$
obtained this way is known as the \emph{internal Cayley transform} of
$\Sigma$. The discrete time transfer function of $\phi_\sigma$ is given by
$\DTF_\sigma(z) = D_\sigma + z C_\sigma\left (I + z A_\sigma \right )^{-1} B_\sigma$, and we
have the correspondence
\begin{equation} \label{eq:IntCayleyTF}
  \DTF_\sigma(z) = \CTF_\Sigma \left ( \sigma\frac{1 - z}{1 + z} \right )
\quad \text{for} \quad z \in \D.
\end{equation}
Conversely, if $-1 \notin \sigma(A_\sigma)$, we have
\begin{equation*}
\begin{aligned}
  A & = - \sigma \left (I + A_\sigma \right )^{-1} \left (I - A_\sigma \right ) , \quad B = \sqrt{2 \sigma} \left (I + A_\sigma \right )^{-1} B_\sigma, \\
  C & = \sqrt{2 \sigma} C_\sigma  \left (I + A_\sigma \right )^{-1}, \quad D = D_\sigma - C_\sigma (I + A_\sigma)^{-1} B_\sigma
\end{aligned}
\end{equation*}
since $(\sigma - A)^{-1} = \frac{1}{2 \sigma} (I + A_\sigma)$.  The
internal Cayley transform can be interpreted as the Crank--Nicolson
time discretisation scheme (also known as Tustin's method) as
explained in \cite{H-M:CTATDS} or as a spectral discretisation method
as explained in \cite{A-G:MSIVPLDEIH,G-M:CTSIVPFODE}; see also
\cite[Section~12.3]{OS:WPLS}. That this time discretisation scheme
respects passivity was already indicated in
Proposition~\ref{ScatteringPassiveProp}.

It remains to mention the \emph{internal reciprocal system} $\Sigma_-
= \sbm{A_- & B_- \\ C_- & D_-}$ of $\Sigma = \sbm{A & B \\ C & D}$.
If the main operator $A$ is invertible, $\Sigma_-$ is defined by
\begin{equation} \label{eq:InternalReciprocals}
\begin{aligned}
  A_- & := A^{-1}, \quad B_- := A^{-1}B,  C_- := -CA^{-1}, \quad \\
  D_- & := \CTF_\Sigma(0) = D - CA^{-1}B.
\end{aligned}
\end{equation}
Obviously, $\left ( \Sigma_-\right )_- = \Sigma$ and
$\CTF_{\Sigma_-}(s) = \CTF_{\Sigma} \left (1 / s \right )$.  The
reciprocal system is studied
in~\cite{Curtain:RLSTR:2003},\cite[Section~12.4]{OS:WPLS}, and it is
useful for interchanging high and low frequency contributions in
system responses when carrying out dimension reduction based on a
desired frequency passband.

\subsection{\label{ExtCayleySubSec} External transformations}

Four fundamental operations on state space realisations $\Sigma$ were
introduced in Section~\ref{FundamentalSec}. Three further combinations
of these operations have an essential role in feedbacks of linear
dynamical systems. We proceed to introduce them next, and we also
discuss further the Full Inversion transformation in
Section~\ref{ERTSec}.

Applications produce two variants of linear systems $\Sigma$ that
impose different kinds of restriction on couplings of signals:
\textrm{(i)} systems whose signals have two different the
\emph{physical dimensions}, and \textrm{(ii)} systems whose signals
have the same physical dimensions but different \emph{physical
  directions}. Systems of the first kind have transfer functions that
typically represent acoustical or electric impedances or
admittances. The systems of the second kind transfer energy through
their inputs and outputs in three-dimensional space. Both kinds of
mathematical systems may be used to describe the same physical
configuration but requirements due to, e.g., measurement and
instrumentation make different descriptions more preferable.

From now on, we add a purely semantic property to signals in system
diagrams: the \emph{colour} which is either red or blue.  System
diagrams are always required to satisfy the following \emph{colour
  rule}: two signals of different colour cannot be coupled. Depending
on the context, the colour of a signal may either refer to the
physical dimension or the direction of energy flow in the underlying
physics. The colour rule helps keeping track of the underlying physics
in the Redheffer star products in Section~\ref{RedhefferSec} even though
mathematics itself is ``colour blind''.

\begin{figure}[h]
  \centering
  \includegraphics[scale=1.5]{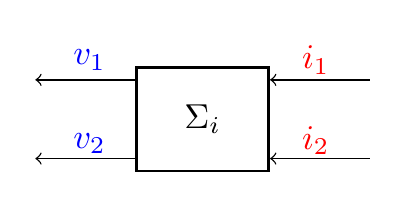}
  \includegraphics[scale=1.5]{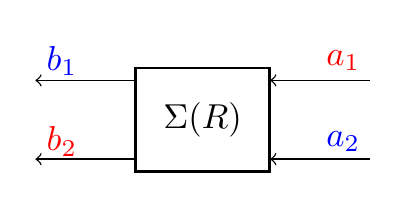}
  \caption{An impedance system with two inputs and two outputs
    (left). A scattering system $\Sigma(R)$ constructed of $\Sigma_i$
    with the resistance matrix $R$ (right). The relation between
    signals is given in Eq.~\eqref{eq:ImpScatteringSigs}.  Colour of
    the signals in left panel refers to the physical dimension and in
    the right channel to direction of the energy flow.}
  \label{fig:OrigSystem}
\end{figure}

\subsubsection{\label{ExtCayleySec} External Cayley transformations}

Let $\Sigma_i = \sbm{A_i & B_i \\ C_i & D_i}$ be a system whose input
and output spaces are $m$-dimensional.  Moreover, let $R := \sbm{ R_1
  & 0 \\ 0 & R_2}$ be positive, invertible, $m \times m$
\emph{resistance matrix}. Define now the matrices
\begin{equation}
  \label{eq:imp2scat}
  \begin{aligned}
    A & := A_i - B_i \left (D_i + R \right )^{-1} C_i, \quad
    B := \sqrt{2} B_i \left (D_i    + R \right )^{-1} R^{1/2} \\
    C & := \sqrt{2} R^{1/2} \left (D_i + R \right )^{-1} C_i, \quad
    D :=  I - 2 R^{1/2} \left ( D_i + R \right )^{-1} R^{1/2},
\end{aligned}
\end{equation}
comprising the system $\Sigma(R) := \sbm{A & B \\ C & D}$ where it is
assumed that $D_i + R$ is invertible; see Proposition~\ref{AnyROKProp}
below. Conversely, we have
\begin{equation}
  \label{eq:scat2imp}
  \begin{aligned}
    A_i & = A + B (I-D)^{-1}C, \quad
    B_i = \sqrt{2} B \left(I-D\right)^{-1} R^{1/2} \\
    C_i &= \sqrt{2} R^{1/2} \left(I-D\right)^{-1} C, \quad
    D_i  = R^{1/2} \left (I-D\right )^{-1} \left (I+D \right ) R^{1/2}.
  \end{aligned}
\end{equation}
For reasons explained in Section~\ref{CharSec}, we call $\Sigma_i$ and
$\Sigma(R)$ the \emph{impedance system} and \emph{scattering system with
  coupling channel resistance} $R$, respectively.  Observe that the
discretisation parameter $\sigma > 0$ in Eq.~\eqref{eq:IntCayleyDef} for
the internal Cayley transform and the resistance matrix $R > 0$ for
the external Cayley transform play somewhat analogous roles.

Any choice of $R > 0$ is acceptable for an impedance passive system
even though some values of $R$ are more desirable than others:
\begin{prop} \label{AnyROKProp}
  If $\Sigma_i = \sbm{A_i & B_i \\ C_i & D_i}$ is an impedance passive
  system, then $\Sigma(R)$ defined by Eq.~\eqref{eq:imp2scat} exists
  for all invertible $R > 0$. Moreover, the feedthrough matrix $D$ of
  $\Sigma(R)$ satisfies $1 \notin \sigma(D)$.
\end{prop}
\noindent There exists an impedance conservative $\Sigma_i$ for which
$-1 \in \sigma(D)$ since $D_i = 0$ is possible, and in some physically
motivated applications such as Example~\ref{piTopologyExample} it is
even typical.
\begin{proof}
  We have for all $m$-vectors $u$
\begin{equation*}
\begin{aligned}
  & 2 \Re \left < (R + D_i)u, u \right > = \left < (R + D_i)u, u \right > + \left < u, (R + D_i) u \right > \\
 & = 2 \left <R u , u \right > + \left < (D_i + D_i^T) u,  u \right > \geq 2 \left <R u , u \right > > 0
\end{aligned}
\end{equation*}
since $D_i + D_i^T \geq 0$ by Proposition~\ref{ImpPassiveProp}. Thus
$R + D_i$ is an invertible $m \times m$ matrix, and the first claim
follows. That $1 \notin \sigma(D)$ follows from the last equation in
Eq.~\eqref{eq:imp2scat}.
\end{proof}

Denoting the input and output signals of $\Sigma_i$ and $\Sigma(R)$ in
their dynamical equations (analogously with Eqs.~\eqref{eq:systemDyn})
by $\sbm{i_1 \\ i_2}$, $\sbm{v_1 \\ v_2}$, $\sbm{a_1 \\ a_2}$,
$\sbm{b_1 \\ b_2}$, respectively, we have the relations
\begin{equation} \label{eq:ImpScatteringSigs}
   \bbm{a_1 \\ a_2} = \frac{R^{-1/2}}{\sqrt{2}}  \bbm{v_1 + R_1 i_1 \\ v_2 + R_2 i_2}
  \quad \text{and} \quad \\
  \bbm{b_1 \\ b_2} = \frac{R^{-1/2}}{\sqrt{2}}  \bbm{v_1 - R_1 i_1 \\ v_2 - R_2 i_2}.
\end{equation}
If the physical dimension of $i_1$ is current and $v_1$ is voltage,
then the dimension of $v_1 + R_1 i_1$ is voltage. It follows, for
example, that $\left | a_1 \right |^2 = \tfrac{1}{2 R_1} \left | v_1 +
R_1 i_1 \right |^2$, and its dimension is thus power. The same holds
for the all other signals $a_2, b_1, b_2$ of the scattering system
$\Sigma(R)$.

In practice, both impedance and scattering measurements are used for
passive circuits. Direct impedance measurements are impractical for,
e.g, high frequency work often carried out using Vector Network
Analysers (see, e.g., \cite{Anderson:1997:SPD},
\cite[Section~12]{Orfanidis:EWA:2016}) that are based on scattering
parameters instead of voltages and currents. The external Cayley
transformation with resistance matrix $R$ is plainly a translation of
these frameworks in state space. Scattering systems $\Sigma(R)$ are
also directly eligible for Redheffer star products introduced in
Section~\ref{RedhefferSec}.

\subsubsection{\label{ERTSec} External reciprocal transformation}

The flow inverted system $\Sigma_f = \sbm{A_f & B_f \\ C_f & D_f}$ of
$\Sigma = \sbm{A & B \\ C & D}$ is obtained by $\Sigma_f = \FI(\Sigma)$
as introduced in Section~\ref{FundamentalSec}. It follows that
\begin{equation*}
  \CTF_{\Sigma_f}(s) = \CTF_{\Sigma}(s)^{-1}\,\text{for } s \notin
  \sigma(A-BD^{-1}C),
\end{equation*}
and we call $\Sigma_f$ the \emph{external reciprocal transform} of
$\Sigma$.  This transformation is possible if and only if the
feedthrough $D$ or, equivalently, $D_f$ is an invertible matrix.
Indeed, we observe that
\begin{equation}
  \label{eq:invsystemDyn}
\begin{cases}
    x'(t) & = A_f x(t) +  B_f y(t), \\
    u(t)  & = C_f x(t) +   D_f y(t)
\end{cases}
\end{equation}
holds if and only if Eq.~\eqref{eq:systemDyn} holds.  If the transfer
function $\CTF_{\Sigma}(s)$ is an impedance of a passive circuit, then
$\CTF_{\Sigma_f}(s)$ is the admittance of the same circuit. From
purely mathematical systems theory point of view, the impedance and
admittance are completely analogous concepts. However, for some
circuits, either impedance, or admittance, or even both of them may
not be realisable by a finite-dimensional system which is a
restriction on how one should write the modelling equations.

We conclude that the external Cayley and the reciprocal
transformations connect systems of scattering, impedance, and
admittance type without further restrictions whenever a technical
assumption concerning the feedthrough matrix holds:
\begin{prop} \label{FInvProp}
  Let $\Sigma = \sbm{A & B \\ C & D}$ be a linear system and $R > 0$
  be an invertible matrix. Then the following holds:
\begin{enumerate}
  \item The (impedance) system $\Sigma_i = \sbm{A_i & B_i \\ C_i &
    D_i}$ given by Eq.~\eqref{eq:scat2imp} and its external reciprocal
    transform, the (admittance) system $\left (\Sigma_i \right )_f$
    exist if and only if $\pm 1 \notin \sigma(D)$.
  \item Defining the (scattering) system $\Sigma(R)$ in terms of
    $\Sigma_i$ and Eqs.~\eqref{eq:imp2scat}, we have $\Sigma(R) =
    \Sigma$.
    \item The matrix $D$ is block diagonal in the same way as $R$ if
      and only if the matrix $D_i$ is block diagonal in the same way
      as $R$.
\end{enumerate}
\end{prop}
\noindent This follows by inspection of Eq.~\eqref{eq:scat2imp} and
the definition of $\Sigma_f$.

\subsubsection{\label{HybridSec} Hybrid transformation}


The two examples given in Section~\ref{ApplicationsSec} use the
external Cayley transformation to produce physically and
mathematically realistic Redheffer products. There is yet another way
of transforming an impedance passive system $\Sigma_i = \sbm{A_i & B_i
  \\ C_i & D_i}$ into the \emph{hybrid system} $\Sigma_h$ (see
\cite[Example~4.1]{HK:CSAH}) so as to make the Redheffer product of
two such (otherwise compatible) systems a physically realistic
feedback connection.  The corresponding operation for realisations is
called the \emph{hybrid transformation}, which we now introduce for
the sake of completeness.  The hybrid transformation is illustrated in
Fig.~\ref{fig:HybridSystem} in terms of an impedance system
\begin{equation} \label{ImpedanceSystemSplittedEq}
  \Sigma_i  =
  \bbm{A_i & \bbm{B_{i_1} & B_{i_2}} \\ \bbm{C_{i_1} \\ C_{i_2} } &   \bbm{D_{i_{11}} & D_{i_{12}} \\
      D_{i_{21}} & D_{i_{22}} }}
\end{equation}
associated to differential equations
\begin{equation} \label{eq:systemDynImp}
   \begin{cases}
     x'(t) & = A_i x(t) + B_i u(t), \\ y(t) & = C_i x(t) + D_i u(t)
   \end{cases}
\end{equation}
with the (current) input signal $u(t) = \bbm{i_1(t) & i_2(t)}^T$ and
the (voltage) output signal $y(t) = \bbm{v_1(t) & v_2(t)}^T$.
\begin{figure}[h]
  \centering
  \includegraphics[scale=1.5]{system_impedance.pdf}
  \includegraphics[scale=1.5]{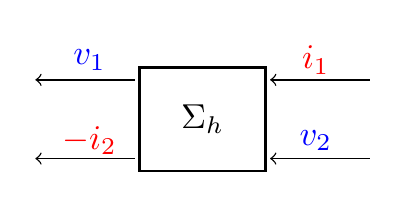}
   \includegraphics[scale=0.7]{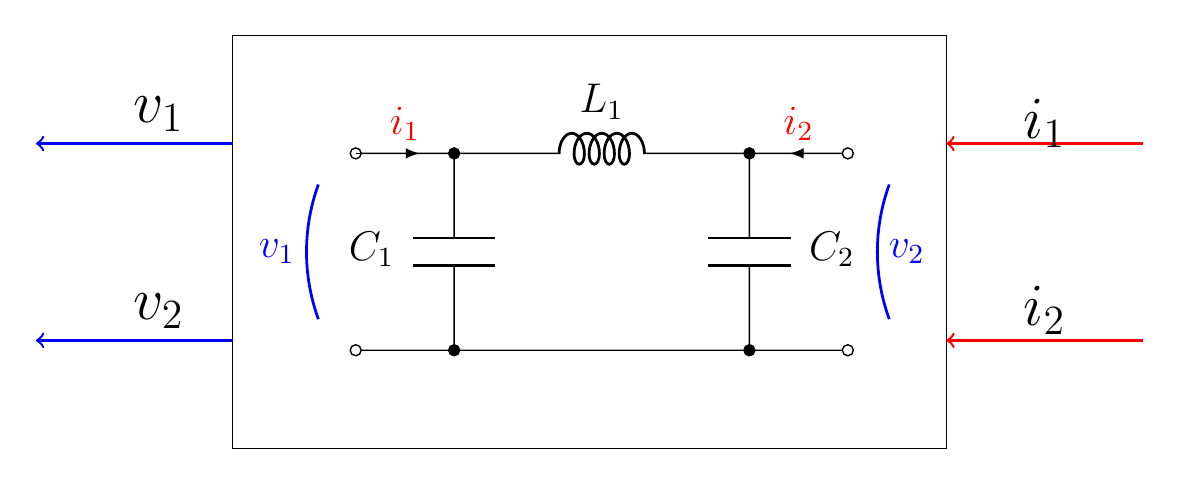}
  \caption{An impedance system $\Sigma_i$ (left top panel) and its
    hybrid transform $\Sigma_h$ (right top panel) in diagrammatic
    form. A $\pi$-topology circuit defining an impedance passive
    system to be used in Section~\ref{ButterworthSec} (bottom
    panel). In all of these diagrams, the relations between $i_1, i_2,
    v_1,$ and $v_2$ are the same. The positive directions of currents
    in the \emph{circuit diagram} point always into the circuit,
    following the convention of Ohm's law and contrary to the one
    followed in \cite{Orfanidis:EWA:2016}. When chaining such circuits
    together using the Redheffer star product, the directions of the
    currents in couplings must be compatible with Kirchhoff's
    laws. This is achieved by including a sign reversal for $i_2$ in
    the hybrid transformation. Recall that the signal arrows in
    \emph{system diagrams} separate system inputs from outputs, and
    they are not related to directions of currents in circuit
    diagrams.}
  \label{fig:HybridSystem}
\end{figure}

To compute the realisation for $\Sigma_h$ in terms of $\Sigma_i$, the
new output component $-i_2$ is solved from
Eqs.~\eqref{eq:systemDynImp} which is possible if and only if
$D_{i_{22}}$ is invertible.  Then the original output component $v_2$
becomes an input component. Straightforward computations lead to the
realisation
\begin{equation}
  \label{eq:hybrid}
    \Sigma_h = 
  \bbm{A_i-B_{i_2}D_{i_{22}}^{-1}C_{i_2} &
    \bbm{B_{i_1}-B_{i_2}D_{i_{22}}^{-1}D_{i_{21}}  & B_{i_2}D_{i_{22}}^{-1}} \\
    \bbm{C_{i_1}-D_{i_{21}}D_{i_{22}}^{-1}C_{i_2} \\ D_{i_{22}}^{-1}C_{i_2}} &
    \bbm{D_{i_{11}}- D_{i_{12}}D_{i_{22}}^{-1}D_{i_{21}} & D_{i_{12}}D_{i_{22}}^{-1}\\D_{i_{22}}^{-1}D_{i_{21}} &-D_{i_{22}}^{-1} }}.
\end{equation}
Recall that the external reciprocal transform $\Sigma_f$ of $\Sigma_i$
is the full flow inversion given by Eq.~\eqref{eq:invsystem}, and the
transfer function of $\Sigma_f$ models circuit admittance if
$\Sigma_i$ is a model for impedance.  The hybrid transformation is a
partial flow inversion with an extra sign reversal. We leave it to the
reader to derive the realisation formula for the inverse hybrid
transformation $\Sigma_h \mapsto \Sigma_i$.

Given $\Sigma_i$, both the external Cayley transform $\Sigma(R)$ for
$R > 0$ and the hybrid transform $\Sigma_h$ can, at least in
principle, be used for computing Redheffer star products of two
systems. It remains to compare these two approaches.  Compared to the
external Cayley transformation, the benefit of the hybrid
transformation is that a resistance matrix $R$ is not required. The
invertibility requirement of $D_{i_{22}}$ is quite severe in
physically realistic systems whereas any impedance passive $\Sigma_i$
has an external Cayley transform $\Sigma(R)$ for any suitable
resistance block matrix $R > 0$.  In fact, the hybrid transformation
is unusable for all impedance conservative \emph{real}
finite-dimensional systems with two-dimensional signals as shown in
Example~\ref{CuteCounterExample}. Hence, the hybrid transformation is
not further developed in this article apart from a few notes. 

\subsubsection{\label{sec:chain-scattering} Chain transformation}

It remains to introduce the last transformation of realisations,
namely the chain transformation that is introduced in
\cite[Eq.~(4.19)~in~Section~4.2]{HK:CSAH} for solving the $H^{\infty}$
control problem. The benefit of the chain transformation is that
the rather complicated Redheffer star product can be represented as
the simple cascade product (see Definition~\ref{BasicSystemOpsDef}) of
chain transforms as shown in Theorem~\ref{RedhefferCascadeThm}.  We
plainly define
\begin{equation} \label{ChainScatteringDef}
  \Chain(\Sigma) := \TI(\OF(\Sigma)) \quad \text{for} \quad   \Sigma = \bbm{A & \bbm{B_{1} & B_{2}} \\
    \bbm{C_{1} \\ C_{2}} & \bbm{D_{{11}} & D_{{12}} \\
       D_{{21}} & D_{{22}} \\}}
\end{equation}
whenever the invertibility conditions required by $\Chain(\Sigma)$ are
satisfied. By direct computations, we get 
\begin{equation}
  \label{eq:chain}
  \Chain(\Sigma) = \bbm{A - B_1 D_{21}^{-1} C_2 & \bbm{B_2 - B_1D_{21}^{-1}D_{22} & B_1D_{21}^{-1}}
    \\  \bbm{C_1 - D_{11}D_{21}^{-1}C_2  \\ -D_{21}^{-1}C_2 } &
    \bbm{D_{12}- D_{11}D_{21}^{-1}D_{22} & D_{11}D_{21}^{-1} \\ -D_{21}^{-1}D_{22} & D_{21}^{-1}}},
\end{equation}
and, hence, $\Chain(\Sigma)$ exists if and only if $D_{21}$ is
invertible. Conversely, the inverse operation for chain transformation
satisfies $\Chain^{-1} = \left (\TI \circ \OF \right )^{-1} = \OF^{-1}
\circ \TI^{-1} = \OF \circ \TI$. In terms of realisations, we have
\begin{equation}
  \label{eq:dechain}
  \Chain^{-1}(\Sigma_c) = \bbm{
    A_c - B_{c_2} D_{c_{22}}^{-1} C_{c_2}  & \bbm{B_{c_2} D_{c_{22}}^{-1} & B_{c_1} - B_{c_2}D_{c_{22}}^{-1}D_{c_{21}}} \\
    \bbm{C_{c_1} - D_{c_{12}}D_{c_{22}}^{-1}C_{c_2} \\ -D_{c_{22}}^{-1}C_{c_2}} & \bbm{D_{c_{12}}D_{c_{22}}^{-1} &
      D_{c_{11}}-D_{c_{12}}D_{c_{22}}^{-1}D_{c_{21}} \\
      D_{c_{22}}^{-1} & - D_{c_{22}}^{-1}D_{c_{21}}}}
\end{equation}
for
\begin{equation*}
\Sigma_c = \bbm{A_c & \bbm{B_{c_1} & B_{c_2}} \\
    \bbm{C_{c_1} \\ C_{c_2}} & \bbm{D_{c_{11}} & D_{c_{12}} \\
       D_{c_{21}} & D_{c_{22}} \\}}.
\end{equation*}

As pointed out in Section~\ref{HybridSec}, computing Redheffer
products is meaningful either for the external Cayley transform
$\Sigma(R)$ for $R > 0$ or the hybrid transform $\Sigma_h$ of an
impedance passive $\Sigma_i$ splitted as in
Eq.~\eqref{ImpedanceSystemSplittedEq}. 
\begin{prop} \label{ChainApplicableProp}
Let $\Sigma_i$ given by Eq.~\eqref{ImpedanceSystemSplittedEq} be an
impedance passive system with its external Cayley transform
$\Sigma(R)$ for an invertible $R = \sbm{R_1 & 0 \\ 0 & R_2} >0$.  By
$\Sigma_h$ denote the hybrid transform of $\Sigma_i$ (if it exists).
Then the following holds:
\begin{enumerate}
\item \label{ChainApplicablePropClaim1} $\Chain(\Sigma(R))$ exists if and only if 
$D_{i_{21}}$ is invertible.
\item \label{ChainApplicablePropClaim2} Both $\Sigma_h$ and
  $\Chain(\Sigma_h)$ exist if and only if both $D_{i_{21}}$ and
  $D_{i_{22}}$ are invertible.
\end{enumerate}
\end{prop}
\begin{proof}
Claim~\eqref{ChainApplicablePropClaim1}: The proof is based on the fact that
any block matrix $\sbm{\alpha & \beta \\ \gamma & \delta}$ with square blocks
and invertible $\delta$ satisfies
\begin{equation*}
  \bbm{\alpha & \beta \\ \gamma & \delta}^{-1} 
  = \bbm{ (\alpha - \beta \delta^{-1} \gamma)^{-1}  & -(\alpha - \beta \delta^{-1} \gamma)^{-1} \beta \delta^{-1} \\ 
   -\delta^{-1}  \gamma (\alpha - \beta \delta^{-1}  \gamma)^{-1} & (\delta - \gamma \alpha^{-1}\beta)^{-1} }
\end{equation*}
where $\alpha$ and both the Schur complements $\alpha - \beta \delta^{-1} \gamma$ and
$\delta - \gamma \alpha^{-1}\beta$
are invertible as a consequence.
 
By Eq.~\eqref{eq:imp2scat}, the feedthrough operator of $\Sigma(R)$ is
given by $D = I - 2 R^{1/2} \left ( D_i + R \right )^{-1} R^{1/2}$,
and we are interested in the bottom left block, say $D_{21}$, of it.
Now, $D_{21}$ is invertible if and only if the bottom left block of
\begin{equation} \label{ChainApplicablePropEq1}
R^{1/2} \left ( D_i + R \right )^{-1} R^{1/2} = 
\bbm{R_1^{1/2} & 0 \\ 0 & R_2^{1/2}} \bbm{D_{i_{11}} + R_1 & D_{i_{12}} \\ D_{i_{21}} & D_{i_{22}} + R_2 }^{-1}
\bbm{R_1^{1/2} & 0 \\ 0 & R_2^{1/2}}
\end{equation}
is invertible. Observe that $D_i^T + D_i \geq 0$ by impedance
passivity of $\Sigma_i$, and $R > 0$ is invertible by assumption. The
invertibility of $D_i + R$ follows as in the proof of
Proposition~\ref{AnyROKProp}. Since $D_i^T + D_i \geq 0$, we have
$D_{i_{22}}^T + D_{i_{22}} \geq 0$, and hence also $D_{i_{22}} + R_2$
is invertible. Since both $R_1^{1/2}$ and $R_2^{1/2}$ are positive
invertible matrices, we only need consider the bottom left block of
$(D_i + R)^{-1}$ in Eq.~\eqref{ChainApplicablePropEq1} which is given
by
\begin{equation*}
  - (D_{i_{22}} + R_2)^{-1} D_{i_{21}}(D_{i_{11}} + R_1 - D_{i_{12}}(D_{i_{22}} + R_2)^{-1} D_{i_{21}} )^{-1}
\end{equation*}
using the Schur complements. Its invertibility is equivalent with the invertibility of  
$D_{i_{21}}$ as claimed. 

Claim~\eqref{ChainApplicablePropClaim2} is seen to hold by inspection.
\end{proof}

It is unfortunate that many physically motivated impedance passive
systems have vanishing feedthrough operators $D_i = 0$ making
straightforward chain transformation infeasible. Both of the
applications in Section~\ref{ApplicationsSec} are of this kind.

\section{\label{CharSec} Passivity of the transformed systems}

The external Cayley transformation is impedance/scattering passivity
preserving for any value of the resistance matrix, too.
\begin{prop} \label{ExtCayleyPassivityProp}
Let $\Sigma(R) = \sbm{ A &B \\ C & D}$ and $\Sigma_i = \sbm{ A_i & B_i
  \\ C_i & D_i}$ be linear systems that are related by
Eqs.~\eqref{eq:imp2scat}--\eqref{eq:scat2imp} where $R> 0$ is
invertible. Then the following are equivalent:
\begin{enumerate}
   \item \label{ExtCayleyPassivityPropClaim1} $\Sigma_i$ is impedance passive;
   \item \label{ExtCayleyPassivityPropClaim2} $\Sigma(R)$ is scattering passive for some invertible $R > 0$; and
   \item \label{ExtCayleyPassivityPropClaim3} $\Sigma(R)$ is scattering passive for all invertible $R > 0$.
\end{enumerate}
The equivalences remain true if the word ``passive'' is replaced by ``conservative''.
\end{prop}
\begin{proof}
We prove first the implication \eqref{ExtCayleyPassivityPropClaim2}
$\Rightarrow$ \eqref{ExtCayleyPassivityPropClaim1}. Let $\tilde R > 0$
be such that $\Sigma(\tilde R) = \sbm{ \tilde A & \tilde B \\ \tilde C
  & \tilde D}$ is scattering passive where
\begin{equation*}
  \begin{aligned}
   \tilde A & := A_i - B_i \left (D_i + \tilde R \right )^{-1} C_i, \quad
   \tilde B := \sqrt{2} B_i \left (D_i    + \tilde R \right )^{-1} \tilde R^{1/2} \\
   \tilde C & := \sqrt{2} \tilde R^{1/2} \left (D_i + \tilde R \right )^{-1} C_i, \quad
   \tilde D :=  I - 2 \tilde R^{1/2} \left ( D_i + \tilde R \right )^{-1} \tilde R^{1/2}.
\end{aligned}
\end{equation*}
Because the external Cayley transformation with resistance $R = I$
maps between impedance and scattering passive systems by \cite[Theorem
  5.2]{Staffans:PCCT:2002}, the system motivated by
Eq.~\eqref{eq:scat2imp} with $R = I$
\begin{equation}
\begin{aligned}
  \Sigma_i(\tilde R) & = \bbm{\tilde A_i & \tilde B_i \\ \tilde C_i & \tilde D_i}
  : = \bbm{\tilde A + \tilde B (I- \tilde D)^{-1} \tilde C  & \sqrt{2} \tilde B \left(I- \tilde D\right)^{-1}
    \\ \sqrt{2} \left(I - \tilde D\right)^{-1} \tilde C & \left (I-\tilde D\right )^{-1} \left (I+ \tilde D \right )} \\
  & = \bbm{A_i & B_i  \tilde R^{-1/2} \\ \tilde R^{-1/2}  C_i  & \tilde R^{-1/2} D_i \tilde R^{-1/2} }
\end{aligned}
\end{equation}
is impedance passive. Hence,
\begin{equation*}
  \bbm{\tilde A_i & \tilde B_i \\ \tilde C_i & \tilde D_i}
  = \bbm{I  & 0 \\ 0 & \tilde R^{-1/2}} \bbm{A_i & B_i \\ C_i & D_i} \bbm{I & 0 \\ 0 & \tilde R^{-1/2}},
\end{equation*}
and it follows from Proposition~\ref{ImpPassiveProp} that $\Sigma_i$
is impedance passive.

That \eqref{ExtCayleyPassivityPropClaim1} $\Rightarrow$
\eqref{ExtCayleyPassivityPropClaim3} follows by reading the above
given reasoning in converse direction with an arbitrary $R > 0$ in
place of $\tilde R$.  The final implication
\eqref{ExtCayleyPassivityPropClaim3} $\Rightarrow$
\eqref{ExtCayleyPassivityPropClaim2} is trivial.

By inspection, the same arguments hold if the word ``passive'' is
consistently replaced by the word ``conservative'' with the only
difference that the LMI in Proposition~\ref{ImpPassiveProp} is then
satisfied as equalities.
\end{proof}

\begin{prop} \label{ReciprocalPassivityProp}
Let $\Sigma = \sbm{ A &B \\ C & D}$ be a linear system.  Then the
following holds:
\begin{enumerate}
  \item \label{ReciprocalPassivityPropClaim1}
 If the internal reciprocal transform $\Sigma_-$ of $\Sigma$ exists, it is impedance passive
if and only if $\Sigma$ is impedance passive.
\item \label{ReciprocalPassivityPropClaim2} If the external reciprocal
  transform $\Sigma_f$ of $\Sigma$ exists, it is impedance passive if and only
  if $\Sigma$ is impedance passive.
\end{enumerate}
Both the claims remain true if the word ``passive'' is replaced by
``conservative''.
\end{prop}
\begin{proof}
  Claim~\eqref{ReciprocalPassivityPropClaim1}: Recalling
  Eq.~\eqref{eq:InternalReciprocals}, we have
   \begin{equation*}
     \bbm{A_-^T +A_- & B_- - C_-^T \\ B_-^T -C_- & -D_-^T-D_- } =
     \bbm{A^{-T} +A^{-1} & A^{-1}B + A^{-T} C^T \\
       B^T A^{-T} + CA^{-1}   & -D^T - D + B^T A^{-T} C^T +  CA^{-1}B}.
   \end{equation*}
   Considering the left hand side, we carry out the following
  congruence transformation:
  \begin{equation*}
    \begin{aligned}
      &  \bbm{A^T & 0 \\ B^T & -I}
      \begin{bmatrix}
        A^{-T} +A^{-1} & A^{-1}B + A^{-T} C^T \\
        B^T A^{-T} + CA^{-1}   & -D^T - D + B^T A^{-T} C^T +  CA^{-1}B
      \end{bmatrix}
      \bbm{A & B \\ 0 & -I} \\
      = &
      \bbm{A^T & 0 \\ B^T & -I}
      \begin{bmatrix}
        A^{-T}A + I        &  (A^{-T} +A^{-1})B  - (  A^{-1}B + A^{-T} C^T) \\
        B^T A^{-T}A  + C   &  (B^T A^{-T} + CA^{-1}) B  + D^T + D - B^T A^{-T} C^T -  CA^{-1}B
      \end{bmatrix} \\
      = &
      \bbm{A^T & 0 \\ B^T & - I}
      \begin{bmatrix}
        A^{-T}A + I        &  A^{-T} (  B- C^T )  \\
        B^T A^{-T}A  + C   &  D^T + D   - B^T A^{-T} ( C^T - B)
      \end{bmatrix} \\
      = &
      \begin{bmatrix}
        A + A^{T}        &    B - C^T   \\
        B^T (A^{-T}A + I)  - (B^T A^{-T}A  + C)   & B^T  A^{-T} (  B - C^T )  - D^T - D   + B^T A^{-T} ( C^T - B)
      \end{bmatrix}
      \\
      = &
      \begin{bmatrix}
        A + A^{T}        &    B - C^T   \\
        B^T - C   &  - D^T - D
      \end{bmatrix}.
    \end{aligned}
  \end{equation*}
  Since $\sbm{A_-^T +A_- & B_- - C_-^T \\ B_-^T -C_- & -D_-^T-D_- }$
  and $\sbm{A^T +A & B - C^T \\ B^T -C & -D^T-D }$ are simultaneously
  negative or vanish, the claim follows from
  Proposition~\ref{ImpPassiveProp}.

  Claim~\eqref{ReciprocalPassivityPropClaim2}: Now we need to study the block matrix
   \begin{equation*}
     \bbm{A_f^T +A_f & B_f - C_f^T \\ B_f^T - C_f & - D_f^T - D_f } =
     \bbm{A^T - C^T D^{-T} B^T + A - BD^{-1} C  &  BD^{-1} + C^T D^{-T} \\
       D^{-1} C + D^{-T} B^T   &  - D^{-T}- D^{-1}}
\end{equation*}
by Eq.~\eqref{eq:invsystem}. It remains to figure out another
congruence transformation:
\begin{equation*}
\begin{aligned}
     & \bbm{ I & C^T  \\ 0  & D^T  }
     \bbm{A^T - C^T D^{-T} B^T + A - BD^{-1} C  &  BD^{-1} + C^T D^{-T} \\
       D^{-1} C + D^{-T} B^T   &  - D^{-T}- D^{-1}}  \bbm{ I & 0  \\  C & D  }  \\
     = &   \bbm{ I & C^T  \\ 0  & D^T  }
     \bbm{A^T + A- C^T D^{-T} B^T  - BD^{-1} C  + (BD^{-1} + C^T D^{-T}) C &  B + C^T D^{-T} D  \\
       D^{-1} C + D^{-T} B^T - ( D^{-T} + D^{-1}) C    & - D^{-T}D - I  }  \\
     = &   \bbm{ I & C^T  \\ 0  & D^T  }
     \bbm{A^T + A- C^T D^{-T} B^T + C^T D^{-T} C &  B + C^T D^{-T} D  \\
       D^{-T} ( B^T -   C )   & - D^{-T}D - I  }  \\
     = &
     \bbm{A^T + A- C^T D^{-T} ( B^T - C ) + C^T D^{-T} ( B^T - C) &  B + C^T D^{-T} D + C^T (- D^{-T}D - I )  \\
          B^T -   C  & - D - D^{T}  }   \\
     = &
     \bbm{A^T + A  &  B - C^T   \\ B^T -   C  & - D - D^{T}  }.
\end{aligned}
\end{equation*}
Again, $\sbm{A_f^T +A_f & B_f - C_f^T \\ B_f^T -C_f & -D_f^T-D_f }$
and $\sbm{A^T +A & B - C^T \\ B^T -C & -D^T-D }$ are simultaneously
negative or vanish which completes the proof by Proposition~\ref{ImpPassiveProp}.
\end{proof}

If the hybrid transform $\Sigma_h$ an impedance passive $\Sigma_i$
exists, then $\Sigma_h$ can be given an equivalent
passivity/conservativity notion can be characterised in terms of LMI's
as well. Similarly, the chain transforms $\Sigma_c$ of $\Sigma_h$ or
the scattering passive $\Sigma(R)$, if they exists, can be given such
equivalent passivity notions. The conservativity of chain transforms
for a (lossless) scattering conservative systems is treated in
\cite[Chapter~4.4]{HK:CSAH} in terms of $J$-losslessness.  Since the
mathematical formulations of these variants is immaterial for the
purpose of this article, we leave the details for an interested
reader.

\begin{figure}[h]
  \centering
  \includegraphics[scale=1.0]{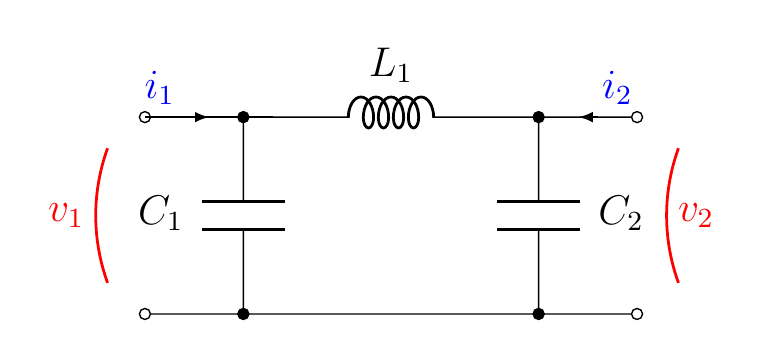}
  \caption{A lossless $\pi$-topology circuit for
    Example~\ref{piTopologyExample}, consisting of two capacitances
    $C_1, C_2$, and one inductance $L_1$.}
  \label{fig:piTopology}
\end{figure}

There are elementary physically motivated examples where, e.g., the
conditions of Proposition~\ref{ReciprocalPassivityProp} are not
satisfied.
\begin{exm} \label{piTopologyExample}
  The governing equations for the LC circuit shown in
  Fig.~\ref{fig:piTopology} are
  \begin{equation}
    \label{eq:KCL}
    p_1-p_2 = L_2 \left (i_1-i_3 \right )', \quad
    p_1 = \frac{1}{C_1}\int_0^t i_3 d\tau, \quad \text{and} \quad
    p_2 = \frac{1}{C_3}\int_0^t \left ( i_1 + i_2 -i_3 \right )d\tau.
  \end{equation}
  By algebraic manipulations, these signals satisfy
  \begin{equation}
    \label{eq:redOrd}
      p_1' = p_a,  \quad 
      C_1 p_a' - i_1' = \frac{p_2-p_1}{L_1}, \quad
      C_2 p_2' = i_1 + i_2 - C_1p_a
  \end{equation}
  where $p_a := p_1'$. Using the scattering type of signals given in
  Eq.~\eqref{eq:ImpScatteringSigs}, we get (after carrying out a
  similarity transformation in the state space) the scattering
  conservative model $\Sigma(R) = \sbm{A & B \\ C & D}$ for any
  positive $R := \sbm{ R_1 & 0 \\ 0 & R_2}$ where
  \begin{equation*}
    A := \bbm{-\frac{1}{{R_1} C_1 } & \frac{1}{\sqrt{{L_1} {C_1}}} & 0 \\
      -\frac{1}{\sqrt{{L_1}  {C_1} }} & 0 & \frac{1}{\sqrt{{L_1} {C_2}}} \\
      0 & -\frac{1}{\sqrt{{L_1} {C_2} }} & -\frac{1}{{R_2}{C_2} } }, \quad
    B :=  \bbm{\sqrt{\frac{2}{R_1 C_1}} & 0 \\
      0 & 0 \\
      0  & \sqrt{\frac{2}{R_2 C_2}}},
  \end{equation*}
  $C := B^T$, and $D := -I$.  The corresponding impedance system
  $\Sigma_i = \sbm{A_i & B_i \\ C_i & D_i}$ can be produced by the
  inverse external Cayley transformation of Eq.~\eqref{eq:scat2imp}
  with $D_i = 0$. The system $\Sigma_i$ is impedance conservative by
  Proposition~\ref{ImpPassiveProp}, and $\Sigma(R)$ is then scattering
  conservative by Proposition~\ref{ExtCayleyPassivityProp}. This is
  all of the good news that there are about this example.

  The internal reciprocal transform of $\Sigma_i$, given by
  Eq.~\eqref{eq:InternalReciprocals} for $\Sigma(R)$ instead of
  $\Sigma_i$, does not exists since $A_i$ given by
  Eq.~\eqref{eq:scat2imp} is not invertible. The external reciprocal
  transform of $\Sigma_i$, namely the admittance system given by the
  $\FI$ operation Eq.~\eqref{eq:invsystem}, does not exists since $D_i$
  is not invertible. Neither does the hybrid system, given by
  Eq.~\eqref{eq:hybrid}, exist since $(D_i)_{22} = 0$ is not
  invertible.

  Furthermore, the transform $\Chain(\Sigma(R))$ does not exists since
  $D_{21} = 0$.  Both the internal and external reciprocal transforms
  of $\Sigma(R)$ exist.
\end{exm}
It is instructive to observe that even more general scattering
\emph{conservative} systems with \emph{two-dimensional signals} can
never be transformed to hybrid form assuming that their impedance
descriptions are possible to begin with.
\begin{exm} \label{CuteCounterExample}
  Let $\Sigma =\sbm{A&B\\C&D}$ be a scattering conservative system
  with $m = 2$.  Such systems (with real matrices) are characterised
  by the equations
  \begin{equation}
    \label{eq:conservativeCondition}
    \begin{aligned}
      A+A^T& = -C^TC=-BB^T, \quad    D^T D = DD^T =I, \\
      B^T &= -D^T C, \quad C = -DB^T.
    \end{aligned}
  \end{equation}
  see, e.g.,~\cite[Proposition~1.4]{M-S-W:HTCCS}. Then the feedthrough
  matrix $D \in \R^{2 \times 2}$ is an orthogonal matrix which is
  always of one of the following two types:
  \[
  \text{ Either }  \qquad S = \begin{bmatrix}
    \rho & \tau \\  \tau & -\rho
  \end{bmatrix} \quad \text{(reflection)}
  \qquad \text{ or }  \qquad
  Q =
  \begin{bmatrix}
    \rho & -\tau \\  \tau & \rho
    \end{bmatrix} \quad \text{(pure rotation)}
  \]
  where $\rho^2+\tau^2=1$, $\mathop{det}(S) = -1$ (with eigenvalues
  $\pm1$), and $\mathop{det}(Q) = 1$ (with complex conjugate
  eigenvalues unless $\tau = 0$). Thus, the matrix $D$ has two
  possibilities: namely,
  \begin{equation*}
    \begin{cases}
      & \mathop{det}(D) = -1 \quad \Rightarrow \quad I - D \text{ is
        not invertible; or } \\ & \mathop{det}(D) = 1 \quad
      \Rightarrow \quad I - D \text{ is invertible if and only if }
      \rho \in [-1,1).
    \end{cases}
  \end{equation*}
  Recalling Eq.~\eqref{eq:scat2imp}, we may produce the impedance
  system $\Sigma_i = \sbm{A_i &B_i \\ C_i & D_i}$ by the inverse
  external Cayley transformation only if $\mathop{det}(D) = 1$ but
  $\rho \in [-1,1)$.  In this case, $D_i$ satisfies
  \begin{equation} \label{CuteCounterExampleEq1}
    D_i = (I-D)^{-1}(I+D)= \frac{\tau}{1-\rho}
    \left[
      \begin{array}{cc}
        0 & -1 \\
        1 & 0 \\
      \end{array}
    \right].
  \end{equation}
  (Without loss of generality, we have set $R = I$ in
  Eq.~\eqref{eq:scat2imp}.)  For $\rho \in (-1, 1)$, the element
  $\left ( D_i \right )_{22} = 0$ is not invertible (hence, the hybrid
  transform of $\Sigma_i$ does not exist) even though $D_i$ is
  invertible and the admittance system $\Sigma_f$ exists since $\tau
  \neq 0$. Moreover, even $\Chain(\Sigma)$ exists for $\rho \in (-1,
  1)$.

  Observe that Example~\ref{piTopologyExample} is the special case of
  Eq.~\eqref{CuteCounterExampleEq1} where $\rho = -1$ and $\tau = 0$.
\end{exm}

Based on
Examples~\ref{piTopologyExample}~and~\ref{CuteCounterExample},
impedance conservative systems $\Sigma_i$ in finite dimension may not
allow external transformations other than the external Cayley
transformation $\Sigma(R)$. Even $\Sigma(R)$ may fail the condition
required for defining Redheffer products. A more desirable subclass of
impedance passive systems is characterised as follows:
\begin{defn} \label{ProperlyImpPassDef}
  An impedance passive $\Sigma_i = \sbm{A_i & B_i \\ C_i & D_i}$ is
  \emph{properly} impedance passive if the matrix $D_i^T+D_i \geq 0$ is
  invertible.
\end{defn}
Obviously, the impedance conservative systems, denoted by $\Sigma_i$,
described in
Examples~\ref{piTopologyExample}~and~\ref{CuteCounterExample} are
impedance passive but not properly so. The external Cayley and
reciprocal transforms of a properly impedance passive system have a
nice description:

\begin{thm}
  \label{ProperlyImpPassThm}
  Let $\Sigma_i = \sbm{A_i &B_i \\ C_i & D_i}$ be an impedance passive
  system whose external Cayley transform $\Sigma(R) = \sbm{A & B \\ C
    & D}$ is given by Eq.~\eqref{eq:imp2scat} where $R > 0$
  is invertible.  Then the following conditions are equivalent:
  \begin{enumerate}
  \item \label{ProperlyImpPassPropClaim1}
    $\Sigma_i$ is properly impedance passive; 
  \item \label{ProperlyImpPassPropClaim2}
    $\Sigma(R)$ is scattering passive, and the matrix $I - D^T
    D$ is invertible and positive; 
  \item \label{ProperlyImpPassPropClaim3} $\Sigma(R)$ is scattering
    passive, $\norm{D} < 1$, and $\sigma(D) \subset \D := \{ z \in \C \, : \, |z| < 1
    \}$ holds; 
  \item \label{ProperlyImpPassPropClaim4} $\Sigma_i$ is properly
    impedance passive, and its feedthrough matrix $D_i$ is invertible;
    and
  \item \label{ProperlyImpPassPropClaim5} the external reciprocal
    transform of $\Sigma_i$ exists, and it is properly impedance
    passive.
  \end{enumerate}
\end{thm}
\noindent By Claim~\eqref{ProperlyImpPassPropClaim3} and
Propositions~\ref{ScatteringConservativeProp}~and~\ref{ExtCayleyPassivityProp},
an impedance conservative system cannot be properly impedance passive.
\begin{proof}
\eqref{ProperlyImpPassPropClaim1} $\Leftrightarrow$
\eqref{ProperlyImpPassPropClaim2}: By
Proposition~\eqref{ExtCayleyPassivityProp}, $\Sigma_i$ is impedance
passive if and only if $\Sigma(R)$ is scattering passive.  To verify
the equivalence, only the feedthrough matrices $D_i$ and $D$ remains
to be considered.  By Proposition~\ref{AnyROKProp}, the matrix
$(I-D)^{-1}$ and its transpose $(I-D^T)^{-1}$ exist.  From
Eq.~\eqref{eq:scat2imp} we see that $R^{-1/2} D_i R^{-1/2} = \left
(I-D\right )^{-1} \left (I+D \right ) = I + 2 D \left (I-D\right
)^{-1} = 2 \left (I-D\right )^{-1} - I$. Thus
\begin{equation} \label{ProperlyImpPassPropEq1}
  \begin{aligned}
    & R^{-1/2} \left (D_i^T + D_i \right ) R^{-1/2} 
    = 2 \left( (I-D^T)^{-1} + (I-D)^{-1} - I \right)\\
    & = 2 (I-D^T)^{-1} \left[(I-D) + (I-D^T) - (I-D^T)(I-D) \right](I-D)^{-1}\\
    & = 2 (I-D^T)^{-1} \left[2I - D -D^T - (I - D - D^T + D^T D) \right](I-D)^{-1}\\
    & = 2 (I-D^T)^{-1} \left[I-D^TD \right](I-D)^{-1}.
  \end{aligned}
\end{equation}
  Since the invertibility and positivity of $D_i^T + D_i$ is
  equivalent with that of $I-D^TD$, the equivalence follows.

\eqref{ProperlyImpPassPropClaim2} $\Rightarrow$
\eqref{ProperlyImpPassPropClaim3}: We have $1 \notin \sigma(D)$ by
Proposition~\ref{AnyROKProp} and $I - D^T D \geq 0$ for any scattering
passive $\Sigma(R)$. If $I - D^T D$ is invertible and nonnegative, the
we have $\norm{u } > \norm{D u }$ for all $u \neq 0$. Since $D$ operates in 
a finite-dimensional space where the unit ball is compact, 
we have $\norm{D} < 1$ and hence $\sigma(D) \subset \D$.

\eqref{ProperlyImpPassPropClaim3} $\Rightarrow$
\eqref{ProperlyImpPassPropClaim4}: Again, $\Sigma_i$ is impedance
passive if and only if $\Sigma(R)$ is scattering passive.  Since, in
particular, $\pm 1 \notin \sigma(D)$ holds, the invertibility of $D_i$
follows from the last equation in \eqref{eq:scat2imp}. That $D_i^T +
D_i$ is invertible follows now from
Eq.~\eqref{ProperlyImpPassPropEq1}.

\eqref{ProperlyImpPassPropClaim4} $\Rightarrow$
\eqref{ProperlyImpPassPropClaim5}: The external reciprocal transform
$\left ( \Sigma_i \right )_f$ exists by
claim~\eqref{ProperlyImpPassPropClaim2} and
Proposition~\ref{FInvProp}, and its feedthrough matrix is
$D_i^{-1}$. The system $\left ( \Sigma_i \right )_f$ is impedance
passive by claim \eqref{ReciprocalPassivityPropClaim2}~of~
Proposition~\ref{ReciprocalPassivityProp}.  Also, the matrix $D_i^T$
is invertible, and we have $D_i^{-T} + D_i^{-1} = D_i^{-T} \left
(D_i^T + D_i \right ) D_i^{-1}$ where $D_i^T + D_i$ is invertible by
assumption.  Thus $D_i^{-T} + D_i^{-1}$ is invertible, and $\left (
\Sigma_i \right )_f$ properly is impedance passive.

\eqref{ProperlyImpPassPropClaim4} $\Rightarrow$
\eqref{ProperlyImpPassPropClaim1}: Trivial.

\eqref{ProperlyImpPassPropClaim5} $\Rightarrow$
\eqref{ProperlyImpPassPropClaim4}: Since $\left ( \Sigma_i \right )_f$
exists, $D_i$ is invertible. As the original system satisfies
$\Sigma_i = \left ( \left ( \Sigma_i \right )_f \right )_f$, the claim
follows from the already proved implication
\eqref{ProperlyImpPassPropClaim4} $\Rightarrow$
\eqref{ProperlyImpPassPropClaim5}.
\end{proof}

\section{\label{RedhefferSec} Redheffer star product}

We proceed to study two-directional feedback couplings.  We allow
external inputs and outputs in addition to those that are internal to
the feedback loop. The fundamental structure of such couplings for
systems $\Sigma_p$ and $\Sigma_q$ is the \emph{Redheffer star product}
$\Sigma_p \star \Sigma_q$.  We will ultimately produce a realisation
formula for $\Sigma_p \star \Sigma_q$ but we first consider it
plainly as a feedback configuration shown in
Fig.~\ref{fig:RedhefferDef}.
\begin{figure}[h]
  \centering
  \raisebox{-0.5\height}{\includegraphics[scale=0.8]{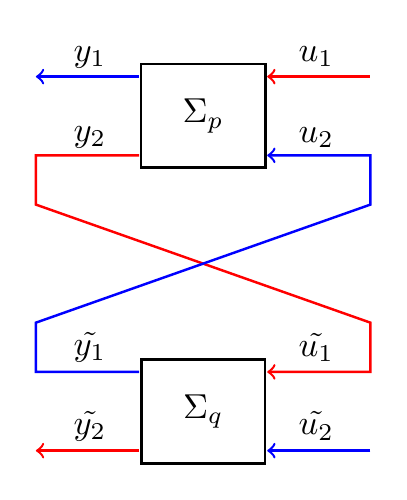}} \hspace{2cm}
  \raisebox{-0.5\height}{\includegraphics[scale=0.8]{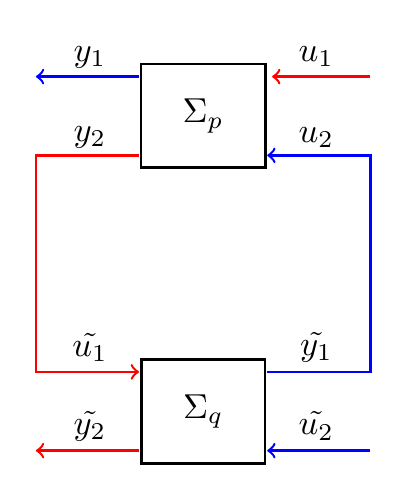}}
  \caption{The feedback configuration describing the Redheffer star
    product $\Sigma_p \star \Sigma_q$ of systems $\Sigma_p$ and
    $\Sigma_q$ diagrammatically drawn in the standard form (left) and
    an equivalent representation (right).  If the systems $\Sigma_p$
    and $\Sigma_q$ are considered as scattering system, the red
    signals refer to the right going energy wave, and the blue signals
    refer to the left going energy wave.  If the systems are of hybrid
    form, the colour separates between current and voltage
    signals.  \label{fig:RedhefferDef}}
\end{figure}
Starting from impedance passive systems, both the scattering systems
in Section~\ref{ExtCayleySec} and the hybrid systems in
Section~\ref{HybridSec} are eligible for Redheffer star products with
other systems of the same kind. We have, in essence, two ways to treat
the same feedback connection. The hybrid transformation requires an
extra invertibility condition whereas any impedance passive system can
be externally Cayley transformed without such
restrictions.\footnote{Even then, the hybrid transformation could well
  be preferable to external Cayley transformation in some particular
  application.}

However, there are additional requirement on systems $\Sigma_p$ and
$\Sigma_q$ to be suitable for forming the feedback system $\Sigma_p
\star \Sigma_q$. Firstly, the signal pairs $(u_2 , \tilde{y}_1)$ and
$(\tilde{u}_1 , y_2)$ in Fig.~\ref{fig:RedhefferDef} must be of
compatible mathematical and physical dimensions. The latter is here
reflected by the semantic requirement that the colours of signals in
couplings in Fig~\ref{fig:RedhefferDef} are not allowed to mix. Also,
the feedback loop in Fig.~\ref{fig:RedhefferDef} may fail to be
\emph{well-posed} in the sense that it cannot be described by any
finite-dimensional state space system at all; see
Definition~\ref{WellPosedFBDef} below.

The feedback connection in Fig.~\ref{fig:RedhefferDef} alone does not
uniquely define a state space for $\Sigma_p \star \Sigma_q$.
Following \cite[Chapter~4]{HK:CSAH}, it is sometimes possible to
compute \emph{one} state space realisation for $\Sigma_p \star
\Sigma_q$ by reducing it to the cascade product of chain transformed
systems; see Definition~\ref{BasicSystemOpsDef},
Section~\ref{sec:chain-scattering}, and the following state space
variant of \cite[Eqs.~(4.7)--(4.8)~in~Section~4.1]{HK:CSAH}:
\begin{thm} \label{RedhefferCascadeThm}
  Let 
\begin{equation} \label{eq:RedhefferProductParts}
\Sigma_p =\bbm{
  A_p & \bbm{B_{p_1} & B_{p_2}} \\
  \bbm{C_{p_1}\\  C_{p_2}} & \bbm{D_{p_{11}} & D_{p_{12}} \\ D_{p_{21}} & D_{p_{22}}}
}
\qquad \text{ and }
\qquad
\Sigma_q =\bbm{
  A_q & \bbm{B_{q_1} & B_{q_2}} \\
  \bbm{C_{q_1}\\  C_{q_2}} & \bbm{D_{q_{11}} & D_{q_{12}} \\ D_{q_{21}} & D_{q_{22}}}
}
\end{equation}
be systems whose signals in Fig.~\ref{fig:RedhefferDef} are
(dimensionally) feasible for the Redheffer star product $\Sigma_p
\star \Sigma_q$.  Assume further that the matrices
  \begin{equation} \label{eq:fullInvertibilityCondition}
    D_{p_{21}}, \quad  D_{q_{21}},  \quad \text{ and } \quad \Delta_1 := I - D_{p_{22}}D_{q_{11}}
  \end{equation}
  are invertible.
  \begin{enumerate}
  \item \label{RedhefferCascadeThmClaim1} The chain transforms
    $\Chain(\Sigma_p)$ and $\Chain(\Sigma_q)$ defined in
    Eq.~\eqref{ChainScatteringDef} exist.
  \item \label{RedhefferCascadeThmClaim2} There exists a state space
    realisation, denoted by $\Sigma_p \star \Sigma_q$, such that
    $\Chain(\Sigma_p \star \Sigma_q)$ exists and satisfies
    \begin{equation}
      \label{eq:RedhefferCascadeThm}
      \Chain \left(\Sigma_p \star \Sigma_q  \right ) = \Chain(\Sigma_p) * \Chain(\Sigma_q)
    \end{equation}
    holds where $*$ denotes the cascade product of realisations.
  \item \label{RedhefferCascadeThmClaim3} The input and output signals
    of $\Sigma_p \star \Sigma_q$ have the same relations as the
    external signals $\sbm{u_1 \\ \tilde{u}_2}$ and $\sbm{y_1 \\ \tilde{y}_2}$ in
    Fig.~\ref{fig:RedhefferDef}. 
  \end{enumerate}
  \end{thm}
\begin{proof}
Claim~\eqref{RedhefferCascadeThmClaim1} follows from the invertibility of
$D_{p_{21}}$ and  $D_{q_{21}}$; see Eq.~\eqref{eq:chain}.

Claim~\eqref{RedhefferCascadeThmClaim2}: We see from
Eqs.~\eqref{eq:chain}--\eqref{eq:dechain} that the cascade product
system $\Chain(\Sigma_p) * \Chain(\Sigma_q)$ is inverse chain
transformable. Considering the feedthrough matrices of
$\Chain(\Sigma_p)$ and $\Chain(\Sigma_q)$, we get for the feedthrough
of $\Chain(\Sigma_p) * \Chain(\Sigma_q)$ the expression
\begin{equation*}
\begin{aligned}
  &  \bbm{D_{p_{12}}- D_{p_{11}}D_{p_{21}}^{-1}D_{p_{22}} & D_{p_{11}} D_{p_{21}}^{-1} \\ -D_{p_{21}}^{-1}D_{p_{22}} & D_{p_{21}}^{-1}}
     \bbm{D_{q_{12}}- D_{q_{11}}D_{q_{21}}^{-1}D_{q_{22}} & D_{q_{11}} D_{q_{21}}^{-1} \\ -D_{q_{21}}^{-1}D_{q_{22}} & D_{q_{21}}^{-1}} \\
  & = \bbm{* & * \\ * & D_{p_{21}}^{-1} \left(I - D_{p_{22}}  D_{q_{11}} \right) D_{q_{21}}^{-1}}
  = \bbm{* & * \\ * & D_{p_{21}}^{-1} \Delta_1  D_{q_{21}}^{-1}}
\end{aligned}
\end{equation*}
where the asterisks denote irrelevant entries. By assumptions, the
bottom right block is invertible, and this is enough by
Eq.~\eqref{eq:dechain} to prove the existence of $\Sigma$ such that
$\Chain \left(\Sigma \right ) = \Chain(\Sigma_p) * \Chain(\Sigma_q)$.

Since Claim~\eqref{RedhefferCascadeThmClaim3} concerns only the
signals of the feedback system, it would be an unnecessary
complication to verify it in terms of state space
realisations. Instead, the proof is indicated in terms of system
diagrams in Fig.~\ref{fig:RedhefferDef}~and~\ref{fig:RedhefferSystem}
by reading them from left to right, and from top to bottom.  The
transformations and the rules of system diagrams given in
Section~\ref{SystemSection} are used together with the
definition $\Chain = \TI \circ \OF$.
\begin{figure}[ht]
  \centering
  \raisebox{-0.5\height}{\includegraphics[scale=0.8]{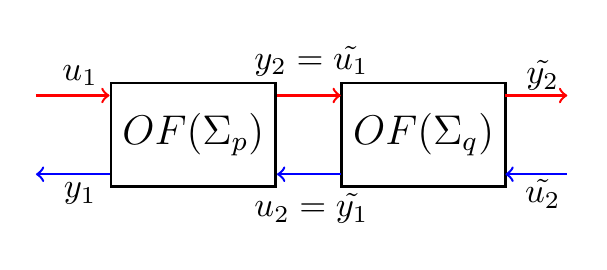}}
  \raisebox{-0.5\height}{\includegraphics[scale=0.8]{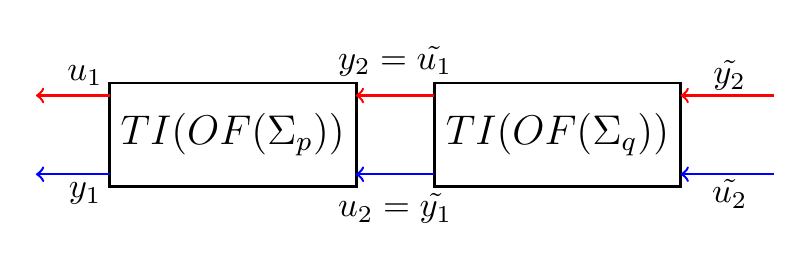}} \newline
  \raisebox{-0.5\height}{\includegraphics[scale=0.8]{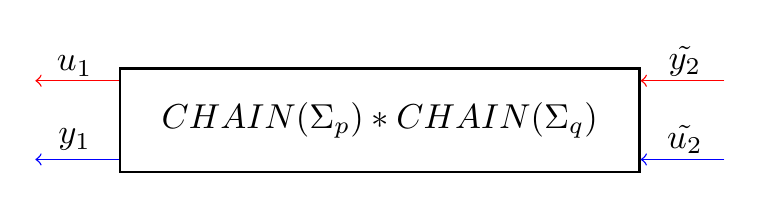}}
  \raisebox{-0.5\height}{\includegraphics[scale=0.8]{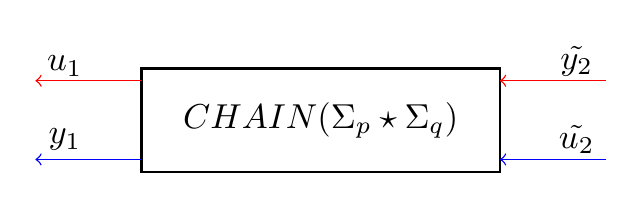}} \newline
  \raisebox{-0.5\height}{\includegraphics[scale=0.8]{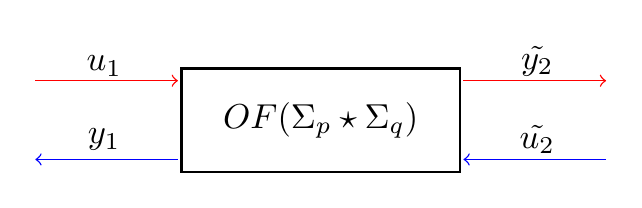}}
  \raisebox{-0.5\height}{\includegraphics[scale=0.8]{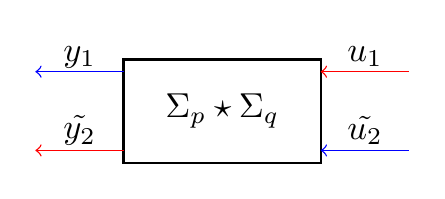}}
  \caption{Equivalent system diagrams involving $\OF$ and $\TI$
    operators and their composition $\Chain = \TI \circ \OF$ of
    Section~\ref{FundamentalSec}. The middle row represents
    Claim~\eqref{RedhefferCascadeThmClaim2}~of~Theorem~\ref{RedhefferCascadeThm}.
    Together with the diagrams in Fig.~\ref{fig:RedhefferDef}, these
    constitute the proof of
    Claim~\eqref{RedhefferCascadeThmClaim3}~of~Theorem~\ref{RedhefferCascadeThm}
    since the same relations between signals $u_1, u_2, y_1, y_2,
    \tilde{u}_1, \tilde{u}_2, \tilde{y}_1$, and $\tilde{y}_2$ in all
    of them.
  \label{fig:RedhefferSystem}}
\end{figure}
\end{proof}
It remains to present a state space formula for $\Sigma_p \star
\Sigma_q$.  Defining $\Sigma_p$ and $\Sigma_q$ by
Eq.~\eqref{eq:RedhefferProductParts} and assuming that the matrices
Eq.~\eqref{eq:fullInvertibilityCondition} together with $\Delta_2 := I
- D_{q_{11}}D_{p_{22}}$ are invertible, we get from
Eq.~\eqref{eq:chain} and Eq.~\eqref{eq:RedhefferCascadeThm} the
expression
\begin{equation} \label{eq:RedhefferStateSpace}
  \Sigma_p \star \Sigma_q = \bbm{\bbm{A_{11} & A_{12} \\ A_{21} & A_{22}} & \bbm{B_{11} & B_{12} \\ B_{21} & B_{22}}
    \\ \bbm{C_{11} & C_{12} \\ C_{21} & C_{22}} & \bbm{D_{11} & D_{12} \\ D_{21} & D_{22}} }
\end{equation}
where the component block matrices are given by
\begin{equation}
  \label{eq:RedhefferComponents}
  \begin{aligned}
    \bbm{A_{11} & A_{12} \\ A_{21} & A_{22}} & := \bbm{A_p + B_{p_2}\Delta_2^{-1}D_{q_{11}}C_{p_2} & B_{p_2}\Delta_2^{-1}C_{q_1} \\
      B_{q_1}\Delta_1^{-1}C_{p_2} & A_q
      +B_{q_1}\Delta_1^{-1}D_{p_{22}}C_{q_1}} \\
    \bbm{B_{11} & B_{12} \\ B_{21} & B_{22}} & := \bbm{B_{p_1} + B_{p_2}\Delta_2^{-1}D_{q_{11}}D_{p_{21}} & B_{p_2}\Delta_2^{-1}D_{q_{12}}\\
      B_{q_1}\Delta_1^{-1}D_{p_{21}}& B_{q_2}
      +B_{q_1}\Delta_1^{-1}D_{p_{22}}D_{q_{12}}}\\
    \bbm{C_{11} & C_{12} \\ C_{21} & C_{22}} & :=  \bbm{C_{p_1}+D_{p_{12}} \Delta_2^{-1} D_{q_{11}} C_{p_2} & D_{p_{12}}\Delta_2^{-1}C_{q_{1}} \\
      D_{q_{21}}\Delta_1^{-1}C_{p_2} & C_{q_2} + D_{q_{21}}\Delta_1^{-1}D_{p_{22}}C_{q_1} }\\
    \bbm{D_{11} & D_{12} \\ D_{21} & D_{22}} & := \bbm{D_{p_{11}} + D_{p_{12}}\Delta_2^{-1} D_{q_{11}} D_{p_{21}} & D_{p_{12}}\Delta_2^{-1}D_{q_{12}}\\
        D_{q_{21}}\Delta_1^{-1}D_{p_{21}} & D_{q_{22}} + D_{q_{21}}\Delta_1^{-1}D_{p_{22}}D_{q_{12}}}.
  \end{aligned}
\end{equation}
Observe that $\Delta_2^{-1} D_{q_{11}} = D_{q_{11}} \Delta_1^{-1}$ and
$\Delta_1^{-1} D_{p_{22}} = D_{p_{22}} \Delta_2^{-1}$ if both
$\Delta_1$ and $\Delta_2$ are invertible. This amounts to some
non-uniqueness in how Eqs.~\eqref{eq:RedhefferComponents} can be
written.
\begin{remark} \label{DiagonalityInheritableRem}
  It is worth noting that for two systems $\Sigma_p$ and $\Sigma_q$
  with (block) diagonal feedthroughs (consistent with
  Standing~Assumption~\ref{SecondSA}), also $\Sigma_p \star \Sigma_q$
  inherits the same (block) diagonality.
\end{remark}
The first observation on Eq.~\eqref{eq:RedhefferComponents} is that
the formulas are well-defined even if the matrices $D_{p_{21}}$ and
$D_{q_{21}}$ were not invertible. These invertibility assumptions are
only required for computing the realisations $\Chain(\Sigma_p)$ and
$\Chain(\Sigma_q)$ which is, in fact, not necessary for describing the
Redheffer feedback connection itself.  Conversely, even if both
$\Chain(\Sigma_p)$ and $\Chain(\Sigma_q)$ did exists, the existence of
the linear system $\Sigma_p \star \Sigma_q$ is not always guaranteed
through
Eqs.~\eqref{eq:RedhefferStateSpace}--\eqref{eq:RedhefferComponents}
since the invertibility of $\Delta_1$ and $\Delta_2 = I -
D_{q_{11}}D_{p_{22}}$ is, in addition, required.
\begin{defn} \label{WellPosedFBDef}
  The feedback loop in Fig.~\ref{fig:RedhefferDef} is
  \emph{well-posed} if both of the matrices $\Delta_1 = I - D_{p_{22}}
  D_{q_{11}}$ and $\Delta_2 = I - D_{q_{11}}D_{p_{22}}$ in
  Eq.~\eqref{eq:fullInvertibilityCondition} are invertible.
\end{defn}

It is clear from
Eq.~\eqref{eq:RedhefferStateSpace}--\eqref{eq:RedhefferComponents}
that the mappings from the external inputs $(u_1, \tilde{u}_2)$ to
external outputs $(y_1, \tilde{y}_2)$ in Fig.~\ref{fig:RedhefferDef}
are well-posed (in the sense that their relation can be represented by
a transfer function of a finite-dimensional linear system) for
invertible $\Delta_1$ and $\Delta_2$.  We may add external
perturbations to the feedback loop by setting $u_2 = \tilde{y}_1 +
v_2$ and $\tilde{u}_1 = y_2 + \tilde{v}_1$, and read the internal
outputs $\tilde{y}_1$ and $y_2$.  Even now the mapping $(u_1, v_2,
\tilde{v}_1, \tilde{u}_2) \mapsto (y_1, y_2, \tilde{y}_1,
\tilde{y}_2)$ is similarly well-posed for invertible $\Delta_1$ and
$\Delta_2$.

We need economical conditions to check the well-posedness of the
feedback loop.
\begin{lemma} \label{RedhefferNecThm} 
  Make the same assumption and use the same notation as in
  Theorem~\ref{RedhefferCascadeThm}.  Necessary conditions for the
  well-posedness of the feedback loop in Fig.~\ref{fig:RedhefferDef}
  are that
  \begin{enumerate}
  \item \label{RedhefferNecThmClaim1}
    the matrices  $D_{p_{22}}$ and $D_{q_{11}}$ satisfy
    $\norm{D_{p_{22}}} + \norm{D_{q_{11}}} < 2$; or
  \item \label{RedhefferNecThmClaim2} 
    one of the matrices $\Delta_1$ and $\Delta_2$ is invertible.
  \end{enumerate}
\end{lemma}
\begin{proof}
  Claim~\eqref{RedhefferNecThmClaim1} follows because both of the
  matrices are contractive, and at least one of them strictly so. Thus
  $\norm{D_{p_{22}}D_{q_{11}}} < 1$ and $\norm{D_{q_{11}}D_{p_{22}}} <
  1$ implying the invertibility of $\Delta_1$ and $\Delta_2$ by the
  usual Neumann series argument.

  Claim~\eqref{RedhefferNecThmClaim2} follows by showing that
  $\Delta_1$ is invertible if and only if $\Delta_2$ is invertible.
  Assume that the square matrix $\Delta_1$ is not invertible, i.e.,
  $\Delta_1 u = u - D_{p_{22}} D_{q_{11}} u = 0$ for some $u \neq
  0$. Thus $u = D_{p_{22}} D_{q_{11}} u$, and it follows that $v :=
  D_{q_{11}} u \neq 0$. Now $v = D_{q_{11}} D_{p_{22}} D_{q_{11}} u =
  D_{q_{11}} D_{p_{22}} v$ and hence $\Delta_2 v = 0$.  We have shown
  that $\Null{\Delta_1} \neq \{ 0 \}$ if and only if $\Null{\Delta_2}
  \neq \{ 0 \}$ which completes the proof.
\end{proof}

\begin{prop} \label{RedhefferSystemPassiveProp}
  Let $\Sigma_p$ and $\Sigma_q$ be two scattering passive
  [conservative] systems such that the feedback loop in
  Fig.~\ref{fig:RedhefferDef} is well-posed. Then the Redheffer star
  product $\Sigma_p \star \Sigma_q$ given by
  Eqs.~\eqref{eq:RedhefferStateSpace}-- \eqref{eq:RedhefferComponents}
  is a scattering passive [respectively, conservative] system.
\end{prop}
\begin{proof}
  Let $\sbm{u_1 \\ \tilde{u}_2}$ be a twice continuously
  differentiable input signal and $\sbm{x_{p_0} \\ x_{q_0}}$ an
  initial state of suitable dimensions for $\Sigma_p \star \Sigma_q$.
  It can be seen by a fairly long computation that the Redheffer star
  product in
  Eqs.~\eqref{eq:RedhefferStateSpace}--\eqref{eq:RedhefferComponents}
  is so defined that the output signals $y_1, \tilde{y}_2$ and the
  state trajectories $x_p, x_q$ in the dynamical equations
\begin{equation*}
  \bbm{x_p'(t) \\ \bbm{y_1(t) \\ y_2(t)}} 
  = \Sigma_p \bbm{x_p(t) \\ \bbm{u_1(t) \\ u_2(t)}} \quad \text{ and } \quad 
  \bbm{x_q'(t) \\ \bbm{\tilde{y}_1(t) \\ \tilde{y}_2(t)}} 
  = \Sigma_q \bbm{x_q(t) \\ \bbm{\tilde{u}_1(t) \\ \tilde{u}_2(t)}} \quad \text{ for } t > 0 
\end{equation*}
with the initial conditions $x_p(0) = x_{p_0}$, $x_q(0) = x_{q_0}$ and
coupling equations $u_2 = \tilde{y}_1, \tilde{u}_1 = y_2$ in
Fig.~\ref{fig:RedhefferDef} are equivalent with the dynamical
equations
\begin{equation*}
  \bbm{\bbm{x_p'(t) \\ x_q'(t)} \\ \bbm{ y_1(t) \\ \tilde{y}_2(t)}} 
  = \Sigma_p \star \Sigma_q  \bbm{\bbm{x_p(t) \\ x_q(t)} \\ \bbm{u_1(t) \\ \tilde{u}_2(t)}} \quad \text{ for } t > 0 
\end{equation*}
with the initial condition $\sbm{x_p(0) \\ x_q(0)} = \sbm{x_{p_0}
  \\ x_{q_0}}$.  Since both $\Sigma_p$ and $\Sigma_q$ are assumed to
be scattering passive, their integrated energy inequalities
\begin{equation*}
\begin{aligned}
  \norm{x_p(T)}^2 - \norm{x_{p_0}}^2 & \leq \int_0^T { \left ( \norm{u_1(t)}^2 + \norm{u_2(t)}^2 - \norm{y_1(t)}^2 - \norm{y_2(t)}^2 \right ) \, dt} \quad \text{ and } \\
  \norm{x_q(T)}^2 - \norm{x_{q_0}}^2 & \leq \int_0^T { \left ( \norm{\tilde{u}_1(t)}^2 + \norm{\tilde{u}_2(t)}^2 - \norm{\tilde{y}_1(t)}^2 - \norm{\tilde{y}_2(t)}^2  \right ) \, dt}
\end{aligned}
\end{equation*}
hold for all $T$; see Eq.~\eqref{eq:ScatteringEnergyInEq}. Adding
these inequalities and using the coupling equations to cancel out
the internal signals gives
\begin{equation*}
  \norm{\bbm{x_p(T)\\x_q(T) }}^2 - \norm{\bbm{x_{p_0} \\ x_{q_0}}}^2  
    \leq \int_0^T { \left ( \norm{\bbm{u_1(t) \\ \tilde{u}_2(t)}}^2 -  \norm{\bbm{y_1(t) \\ \tilde{y}_2(t)}}^2  \right ) \, dt}
\end{equation*}
which proves passivity. The proof for conservativity follows by
replacing all inequalities in proof with equalities.
\end{proof}

Properly impedance passive systems can now be treated through their
external Cayley transforms.
\begin{thm} \label{RedhefferScatteringThm} 
  Let $\Sigma_{p,i}$ and $\Sigma_{q,i}$ be impedance passive systems
  such that the signals are (dimensionally) compatible as shown in
  Fig.~\ref{fig:RedhefferDef}. Assume that at least one
    of systems $\Sigma_{p,i}$ and $\Sigma_{q,i}$ is properly impedance
    passive.  Define the external
  Cayley transformed scattering passive systems
  \begin{equation*}
    \Sigma_{p} :=  \Sigma_{p,i}(R_p) \quad \text{ and } \quad  \Sigma_{q} :=  \Sigma_{q,i}(R_q)
  \end{equation*}
  where $R_p = \sbm{ R_{p,1} & 0 \\ 0 & R_{p,2}}$, $R_q = \sbm{
    R_{q,1} & 0 \\ 0 & R_{q,2}}$ satisfying $R_{p,2} = R_{q,1}$ are
  invertible, positive resistance matrices. Then the following holds:
  \begin{enumerate}
  \item \label{RedhefferScatteringThmClaim1} The internal feedback
    loop in Fig.~\ref{fig:RedhefferDef} is well-posed, and the
    Redheffer star product $\Sigma_p \star \Sigma_q$ given by
    Eqs.~\eqref{eq:RedhefferStateSpace}--\eqref{eq:RedhefferComponents}
    exists.
  \item \label{RedhefferScatteringThmClaim2} There exists an impedance
    passive system $\Sigma_i'$ whose external Cayley transform
    satisfies $\Sigma'(R_{p,q}) = \Sigma_p \star \Sigma_q$ with
    $R_{p,q} := \sbm{ R_{p,1} & 0 \\ 0 & R_{q,2}}$.
  \end{enumerate}
\end{thm}
\begin{proof}
  We write $\Sigma_p = \sbm{A_p&B_p\\C_p&D_p}$ and $\Sigma_q = \sbm{A_q
    & B_q \\ C_q & D_q}$ where
  \begin{equation*}
  D_p = \bbm{D_{p_{11}} & D_{p_{12}} \\ D_{p_{21}} & D_{p_{22}}} \quad
  \text{ and } \quad D_q = \bbm{D_{q_{11}} & D_{q_{12}} \\ D_{q_{21}}
    & D_{q_{22}}}.
  \end{equation*}
  To fix notions, we assume that $\Sigma_{p,i}$ properly impedance
  passive.

  Claim~\eqref{RedhefferScatteringThmClaim1}: We have $\norm{D_p } <
  1$ by Theorem~\ref{ProperlyImpPassThm} and $\norm{D_q } \leq 1$ by
  passivity.  The claim follow from claim
  \eqref{RedhefferNecThmClaim1} of Lemma~\ref{RedhefferNecThm}.

  Claim~\eqref{RedhefferScatteringThmClaim2}: Considering the
  feedthrough operator $D := \sbm{D_{11} & D_{12} \\ D_{21} & D_{22}}$
  of $\Sigma_p \star \Sigma_q$ in Eq.~\eqref{eq:RedhefferComponents},
  we see that $D \sbm{u_1 \\ \tilde{u}_2} = \sbm{y_1 \\ \tilde{y}_2}$
  is equivalent with
  \begin{equation} \label{RedhefferScatteringThmEq1} 
   \bbm{y_1 \\ y_2} = D_p \bbm{u_1 \\ u_2}, \quad \bbm{\tilde{y}_1
     \\ \tilde{y}_2} = D_q \bbm{\tilde{u}_1 \\ \tilde{u}_2}, \quad u_2
   = \tilde{y}_1 \quad \text{ and } \quad \tilde{u}_1 = y_2
  \end{equation}
  following the feedback configuration of
  Fig.~\ref{fig:RedhefferDef}. For contradiction, assume that $I - D$
  is not invertible. Then we must have $D \sbm{u_1 \\ \tilde{u}_2} =
  \sbm{u_1 \\ \tilde{u}_2}$ for some non-vanishing vector $\sbm{u_1
    \\ \tilde{u}_2}$.  This together with
  Eq.~\eqref{RedhefferScatteringThmEq1} gives
  \begin{equation} \label{RedhefferScatteringThmEq2} 
   \bbm{u_1 \\ \tilde{u}_1} = D_p \bbm{u_1 \\ u_2}, \quad \text{ and }
   \quad \bbm{u_2 \\ \tilde{u}_2} = D_q \bbm{\tilde{u}_1
     \\ \tilde{u}_2}.
  \end{equation}
  Since $\Sigma_{p,i}$ is properly impedance passive,
  Theorem~\ref{ProperlyImpPassThm} implies that $\norm{D_p} < 1$.
  Since $\Sigma_{q,i}$ is impedance passive, the system $\Sigma_q$ is
  scattering passive satisfying $\norm{D_q} \leq 1$. Thus,
  Eq.~\eqref{RedhefferScatteringThmEq2} implies
  \begin{equation*} 
    \norm{u_1}^2 + \norm{\tilde{u}_1}^2 < \norm{u_1}^2 + \norm{u_2}^2
    \quad \text{ and } \quad \norm{u_2}^2 + \norm{\tilde{u}_2}^2 \leq
    \norm{\tilde{u}_1}^2 + \norm{\tilde{u}_2}^2.
  \end{equation*}
  Hence $\norm{\tilde{u}_1}^2 < \norm{u_2}^2$ and $\norm{u_2}^2 \leq
  \norm{\tilde{u}_1}^2$ which is impossible. We conclude that $I-D$ is
  invertible, and $\Sigma_p \star \Sigma_q$ is an external Cayley
  transform of some $\Sigma_i'$ by Eq.~\eqref{eq:scat2imp}.  Moreover,
  the system $\Sigma_p \star \Sigma_q$ is scattering passive by
  Proposition~\ref {RedhefferSystemPassiveProp}, and the impedance
  passivity of $\Sigma_i'$ follows from
  Proposition~\ref{ExtCayleyPassivityProp}.
\end{proof}

\section{\label{RegSec} Regularisation}

In finite dimensions, an impedance conservative system $\Sigma_i =
\sbm{A_i &B_i \\ C_i & D_i}$ is characterised by
\begin{equation*}
  A_i + A_i^T = 0, \quad B_i^T = C_i, \quad \text{ and } \quad D_i +
  D_i^T = 0
\end{equation*}
by Proposition~\ref{ImpPassiveProp}.  Because $D_i$ does not appear in
the first two equations, the system $\Sigma' := \sbm{A_i &B_i \\ C_i &
  0}$ is impedance conservative if $\Sigma_i$ is. A circuit theory
example satisfying $D_i = 0$ is provided in
Section~\ref{piTopologyExample} below. Thus, the invertibility of
$D_i$, or any of its parts in splitting $D_i = \sbm{D_{i_{11}} &
  D_{i_{12}} \\ D_{i_{21}} & D_{i_{22}}}$, is not a generic property
of impedance conservative systems. Unfortunately, some kind of
invertibility is required for computing the external reciprocal
transform $\left ( \Sigma_i \right )_f$ by Eq.~\eqref{eq:invsystem},
hybrid transform $\Sigma_h$ by Eq.~\eqref{eq:hybrid}, or the chain
transform of $\Sigma_h$ or $\Sigma(R)$ where $\Sigma(R)$ is the
external Cayley transform of $\Sigma_i$; see
Proposition~\ref{ChainApplicableProp}.  Even though $\Sigma(R)$ exists
and is scattering passive for any impedance passive $\Sigma_i$ and $R
> 0$, the Redheffer product of two such systems may fail to be defined
unless the conditions of Lemma~\ref{RedhefferNecThm} are
satisfied. Indeed, if $D_i = 0$ and $\Sigma(R) = \sbm{A & B \\ C &
  D}$, then $D = -I$ by Eqs.~\eqref{eq:imp2scat} which is an
ingredient of a non-well-posed feedback loop.

To compute feedback systems consisting of general impedance passive
systems $\Sigma_i$, one could take one of these approaches:
\begin{enumerate} 
  \item The external Cayley transform $\Sigma(R)$ of $\Sigma_i$ is
    regularised so as to make the Redheffer star products of such
    similar systems feasible while preserving scattering passivity.
  \item The system $\Sigma_i$ is regularised in a way that preserves
    impedance passivity, so that the hybrid transform $\Sigma_h$ and
    the Redheffer star products of such similar systems are feasible.
\end{enumerate}
Both of these approaches can be taken by using a \emph{shift and
  invert procedure} on $\Sigma_i$.  The first step is the replacement
of $D_i$ by $D_i + \varepsilon I$ for some, purely resistive
perturbation $\varepsilon > 0$. Clearly $\Sigma_i(\varepsilon) :=
\sbm{A_i & B_i \\ C_i & D_i + \varepsilon I}$ is \emph{properly}
impedance passive if $\Sigma_i$ is impedance passive by
Proposition~\ref{ImpPassiveProp}. The second inversion step may be any
of the following: \rm{(i)} the external Cayley transformation,
\rm{(ii)} the hybrid transformation, or even \rm{(iii)} the external
reciprocal transformation, depending on what kind of system is
desirable for modelling purposes. In this article, we concentrate on
the external Cayley transformation of $\Sigma_i(\varepsilon)$, given
by
\begin{equation} \label{ExtCaleyRegularisedEq}
  \Sigma(R, \varepsilon) = \bbm{A_i - B_i \left (D_i + \varepsilon I + R
    \right )^{-1} C_i & \sqrt{2} B_i \left (D_i + \varepsilon I + R
    \right )^{-1} R^{1/2} \\ \sqrt{2} R^{1/2} \left (D_i + \varepsilon I
    + R \right )^{-1} C_i & I - 2 R^{1/2} \left ( D_i + \varepsilon I + R
    \right )^{-1} R^{1/2}}.
\end{equation}
By Theorem~\ref{RedhefferScatteringThm}, the Redheffer star product
$\Sigma_p \star \Sigma(R, \varepsilon)$ is well-defined for all
$\varepsilon > 0$ and scattering passive systems $\Sigma_p$ such that
the feedback loop in Fig.~\ref{fig:RedhefferDef} is possible with
$\Sigma_q = \Sigma(R, \varepsilon)$.

\begin{remark} \label{LimitObjectRem}
We have $\lim_{\epsilon \to 0+}{\Sigma(R, \varepsilon)} = \Sigma(R)$
from Eq.~\eqref{ExtCaleyRegularisedEq} where $\Sigma(R)$ is the
external Cayley transform of $\Sigma_i$.\footnote{When taking a limit
  of finite-dimensional systems, we use any of the equivalent matrix
  norms for the block matrix representing the system.} However, it is
not as straightforward to make sense of the limit object
\begin{equation} \label{LimitObjectRemEq}
  \lim_{\epsilon \to 0+}\left (\Sigma_p \star \Sigma(R, \varepsilon) \right )
\end{equation}
for the reason that the natural limit candidate $\Sigma_p \star
\Sigma(R)$ may not be well-defined by
Eqs.~\eqref{eq:RedhefferStateSpace}--\eqref{eq:RedhefferComponents}.
In this case, the matrix elements of $\Sigma_p \star \Sigma(R,
\varepsilon)$ contain nonnegative powers of $1/\varepsilon$ whose
effect on the \emph{transfer function} of $\Sigma_p \star \Sigma(R,
\varepsilon)$ may be vanishing as $\epsilon \to 0+$. There does not
seem to exist an universal method for describing the limit object in
Eq.~\eqref{LimitObjectRemEq} if the feedback loop of
Fig.~\ref{fig:RedhefferDef} is not well-posed without
regularisation. However, one special case is treated in
Section~\ref{ButterworthSec}.
\end{remark}

It is worth noting that Eqs.~\eqref{eq:RedhefferComponents} simplify
considerably if $D_i$ is (block) diagonal following the splitting of
Standing~Assumption~\ref{SecondSA}.  So as to the system given by
Example~\ref{piTopologyExample}, we get the scattering passive
realisation
\begin{equation} \label{eq:regularisedPiScattering}
  \Sigma(R, \varepsilon) =
  \bbm{ \sbm{  \frac{-1}{ C_1(R_1 + \varepsilon) } & \frac{1}{\sqrt{L_1 C_1}} & 0 \\
  \frac{-1}{\sqrt{L_1C_1}} & 0 & \frac{1}{\sqrt{L_1C_2}} \\
  0 &  \frac{1}{\sqrt{L_1C_2}} & \frac{-1}{C_2(R_2 + \varepsilon)}}&
    \sbm{\sqrt{\frac{2}{(R_1+\varepsilon)C_1}}& 0 \\ 0& 0 \\ 0& \sqrt{\frac{2}{(R_2+\varepsilon)C_2}} } \\
    \sbm{\sqrt{\frac{2}{(R_1+\varepsilon)C_1}}& 0 & 0\\ 0 & 0& \sqrt{\frac{2}{(R_2+\varepsilon)C_2}} } &
    \sbm{1 - \frac{2 R_1}{R_1 + \varepsilon}& 0 \\
      0&1 - \frac{2 R_2}{R_2 + \varepsilon}}
  } \text{ for } \varepsilon > 0 
\end{equation}
where $R = \sbm{R_1 & 0 \\ 0 & R_2}$. We see from this example that
the regularisation by shift-and-invert procedure does not make the
scattering system chain transformable if the original system is not
chain transformable.

\section{\label{ApplicationsSec} Applications}

We proceed to give two applications that illuminate the use of realisation
techniques for model synthesis. The first application concerns the state
space modelling of a passive Butterworth lowpass filter by chaining
the LC circuits described in Example~\ref{piTopologyExample}; see also
\cite[Section~13.2.5]{Antoulas:ALD:2005}.  The second, more
comprehensive application is the one considered in the introduction: an
acoustic transmission line, described by the Webster's
PDE~\eqref{eq:websterIntro}, is coupled to a load that is modelled by
an irrational impedance given in Eq.~\eqref{eq:acImpIntro}.

\subsection{\label{ButterworthSec} Passive Butterworth filter in Cauer $\pi$-topology}

\begin{figure}
    \centering
    \includegraphics[width=0.48\textwidth]{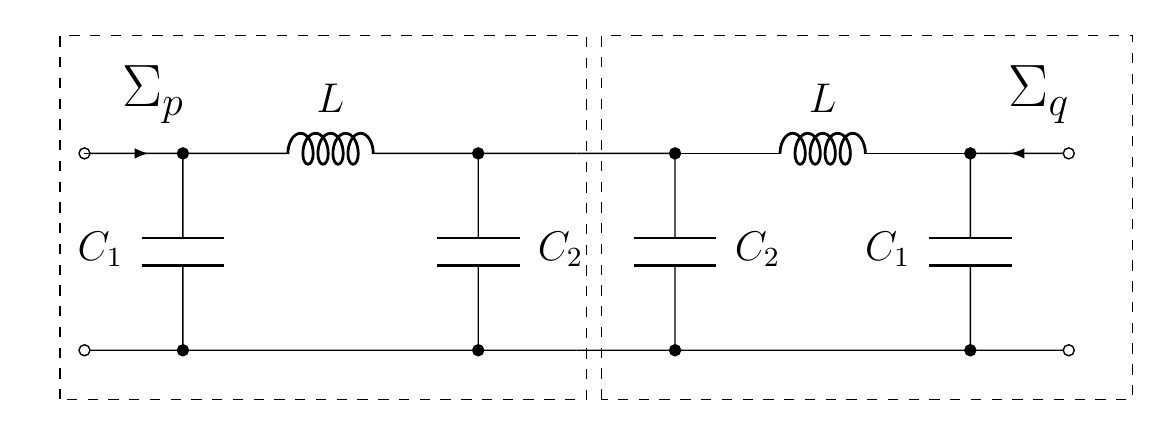}
    \includegraphics[width=0.48\textwidth]{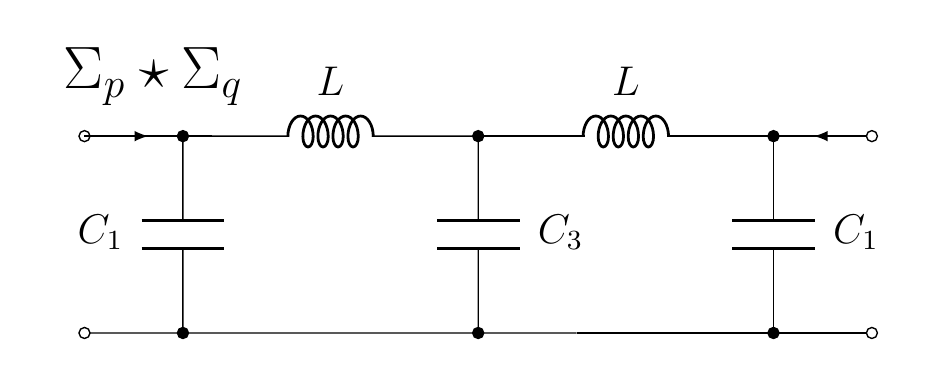}\\
    \caption{Coupling of two $\pi$-topology circuits whose state space
      realisations $\Sigma_p$ and $\Sigma_q$ are derived in
      Example~\ref{piTopologyExample} in scattering form (top
      panel). When the circuits are coupled corresponding
      \emph{formally} to the Redheffer star product $\Sigma_p \star
      \Sigma_q$, the two parallel capacitors of value $C_2$ can be
      regarded as a single capacitor of value $C_3 = 2 C_2$ (bottom
      panel).}
    \label{fig:RedhefferTwoPi}
\end{figure}
The impedance transfer function of the lossless circuit shown in
Fig.~\ref{fig:RedhefferTwoPi} is given by
\begin{equation} \label{eq:fullPiImpedance}
\begin{aligned}
  Z(s) & = \frac{1}{s \left(L C_1  s^2+1\right) \left( L C_1 C_3 s^2+2 C_1+C_3\right)}  \cdot \\
  \cdot & \bbm{L^2 C_1 C_3  s^4+  L (2   C_1 + C_3 )s^2+1 &  1 \\ 
    1 & L^2 C_1 C_3  s^4+  L (2   C_1 + C_3 )s^2+1    }.
\end{aligned}
\end{equation}
The resonant frequencies of the circuit are given by
\begin{equation}
  f_1 = \frac{1}{2 \pi \sqrt{L C_1}} \text{ and } f_2 =
  \frac{1}{2 \pi \sqrt{\tilde L \tilde C}}
\text{ where } \tilde L := \frac{L}{2} \text{ and }  \tilde C := \frac{2 C_1 C_3}{2 C_1 + C_3}.
\end{equation}
Since the $\pi$-topology circuit of Example~\ref{piTopologyExample} is
not properly impedance passive, it is not possible to compute the
realisation for the impedance in Eq.~\eqref{eq:fullPiImpedance} by
using Theorem~\ref{RedhefferScatteringThm} without regularisation.
However, for any $R= \sbm{R_0 & 0 \\ 0 & R_0}$ and $\varepsilon > 0$,
the Redheffer star product realisation of the two shift-and-invert
regularised component systems takes the form\footnote{Observe that we
  could use different $R$ for the component systems $\Sigma_p(R_1,
  \varepsilon)$ and $\Sigma_q(R_2, \varepsilon)$. However, when
  producing the feedback loop, the coupled scattering signals must
  have been defined using the same resistance matrix block, i.e.,
  $R_1= \sbm{R_{11} & 0 \\ 0 & R_{12}}$ and $R_2= \sbm{R_{21} & 0 \\ 0
    & R_{22}}$ with $R_{12} = R_{21}$.}  $\Sigma_p(R, \varepsilon)
\star \Sigma_q(R, \varepsilon) = \sbm{A_\varepsilon & B_\varepsilon
  \\ C_\varepsilon & D_\varepsilon}$ where $D_\varepsilon = \sbm{
  1-\frac{2}{\varepsilon+1} & 0\\ 0 & 1-\frac{2}{\varepsilon+1}}$,
$C_\varepsilon = B_\varepsilon^T = \sbm{\frac{ 1}{\varepsilon+R_0}
  \sqrt{\frac{2 R_0}{C_1}} & 0 & 0 & 0 & 0 & 0\\ 0 & 0 & 0 & 0 & 0 &
  \frac{1}{\varepsilon+R_0}\sqrt{\frac{2 R_0}{C_1 }}}$, and
\begin{equation}  \label{eq:scatteringTwoPiRedheffer}
    A_\varepsilon = \bbm{
      -\frac{1}{ R_0 C_1    {\left(\varepsilon+1\right)}} & \frac{1}{\sqrt{L C_1  }}  & 0 & 0 & 0 & 0\\
      -\frac{1}{\sqrt{L C_1  }}  & 0 & \frac{1}{\sqrt{L C_2  }}  & 0 & 0 & 0\\
      0 & -\frac{1}{\sqrt{L C_2  }}  & -\frac{1}{2 C_2   \varepsilon} & \frac{1}{2C_2\varepsilon}  & 0 & 0\\
      0 & 0 & \frac{1}{2C_2\varepsilon}  & -\frac{1}{2 C_2   \varepsilon} & \frac{1}{\sqrt{L C_2  }}  & 0\\
      0 & 0 & 0 & -\frac{1}{\sqrt{L C_2  }}  & 0 & \frac{1}{\sqrt{L C_1  }} \\
      0 & 0 & 0 & 0 & -\frac{1}{\sqrt{L C_1  }}  & -\frac{1}{ R_0 C_1    {\left(\varepsilon+1\right)}}
    }.
\end{equation}
There is a rank one symmetric matrix $A' := \lim_{\eta \to 0}{\eta
  A_\eta} \leq 0$ with eigenvalues $\sigma(A') := \{0, -1/C_2\}$,
satisfying
\begin{equation*}
  A_\varepsilon = A(\varepsilon) + \frac{A'}{\varepsilon} \quad \text{
    where } \quad A(0) = \lim_{\varepsilon \to 0}{A(\varepsilon)} \quad \text{
    exists. }
\end{equation*}

The realisation $\Sigma_\varepsilon := \Sigma_p(R, \varepsilon) \star
\Sigma_q(R, \varepsilon)$ has a six dimensional state space, yet the
circuit in Fig.~\ref{fig:RedhefferTwoPi} (bottom panel) has only five
components. The capacitor $C_3 = 2 C_2$ in the circuit of
Fig.~\ref{fig:RedhefferTwoPi} consists of the two parallel capacitors
of capacitance $C_2$, each of which corresponds to a separate
degree-of-freedom in the state space. Alternatively, one may think
that there is an extra pole in the transfer function of
$\Sigma_\varepsilon$ due to the non-vanishing regularisation parameter
$\varepsilon$.  In the physically realistic coupling without
regularisation (i.e., at the limit $\varepsilon \to 0+$), transfer of
charge may take place between each capacitor of value $C_2$, resulting
in infinite currents in internal conductors of zero resistance inside
$C_3$. We are dealing with a non-well-posed, yet completely virtual
feedback connection in Fig.~\ref{fig:RedhefferTwoPi} (bottom panel).

As a first step towards a more economic model, it is possible to give
a simplified version of $\Sigma_\varepsilon$ for $\varepsilon \approx
0$. Observing that $D_0 = -I$, $C_0 = B_0^T = \sbm{ \sqrt{\frac{2}{
      R_0 C_1}} & 0 & 0 & 0 & 0\\ 0 & 0 & 0 & 0 &\sqrt{\frac{2 }{R_0
      C_1 }}}$, and defining $A'_\varepsilon := A(0) +
\frac{A'}{\varepsilon}$, we first get the realisation
$\Sigma'_\varepsilon := \sbm{A'_\varepsilon & B_0 \\ C_0 & D_0}$ that
is scattering passive for all $\varepsilon > 0$. Indeed, all other
conditions of scattering conservativity in
Proposition~\ref{ScatteringConservativeProp} are satisfied, except for
the Liapunov equation that only holds as the inequality
\begin{equation*}
  A'_\varepsilon + \left ( A'_\varepsilon \right )^* + C_0^* C_0 = 
  A(0) + A(0)^* + \frac{2A'}{\varepsilon}  + C_0^* C_0 = \frac{2A'}{\varepsilon} \leq 0. 
\end{equation*}
The matrix $A'_\varepsilon$ has two real negative eigenvalues
$\lambda_1(\varepsilon)$ and $\lambda_2(\varepsilon)$ satisfying
$\lim_{\varepsilon \to 0+}{\lambda_1(\varepsilon)} = -\infty$ and
$\lim_{\varepsilon \to 0+}{\lambda_2(\varepsilon)} = 0$. Also,
$\lambda_1(\varepsilon) \approx 1/\varepsilon C_2$ as $\varepsilon
\approx 0$ which is a spurious, transient, and singular (i.e.,
proportional to $1/\varepsilon$) mode, associated with the
regularisation of the non-well-posed feedback loop.

We proceed to entirely removing the spurious eigenvalue
$\lambda_1(\varepsilon)$ by a structure preserving dimension
reduction.  Define
\begin{equation*}
     \tilde A 
     := \bbm{ -\frac{1}{ R_0 C_1 } & \frac{1}{\sqrt{L C_1  }}  & 0 & 0 & 0\\
             -\frac{1}{\sqrt{L C_1  }}  & 0 & \frac{1}{\sqrt{L C_3  }}  & 0 & 0\\
       0 & -\frac{1}{\sqrt{L C_3  }}  & 0  & \frac{1}{\sqrt{L C_3  }} & 0\\
       0 & 0 &  -\frac{1}{\sqrt{L C_3  }}  & 0 & \frac{1}{\sqrt{L C_1  }} \\
       0 & 0 & 0 & -\frac{1}{\sqrt{L C_1  }}  & -\frac{1}{ R_0 C_1 
     }} \text{ where }  C_3 = 2 C_2,
\end{equation*}
and $\tilde D := -I$, $\tilde C := \tilde B^T := \sbm{ \sqrt{\frac{2}{
      R_0 C_1}} & 0 & 0 & 0 & 0 \\ 0 & 0 & 0 & 0 &\sqrt{\frac{2 }{R_0
      C_1 }}}$. Then $\tilde \Sigma := \sbm{\tilde A & \tilde B
  \\ \tilde C & \tilde D }$ is a minimal scattering realisation for
the Butterworth filter described by Fig.~\ref{fig:RedhefferTwoPi}, and
it can be obtained by first applying the unitary similarity
transformation on $A_\varepsilon'$ that diagonalises the $2 \times 2$
singular block there. Finally, the row and column corresponding to the
singular mode $\lambda_1(\varepsilon)$ are plainly removed from the
original realisation $\Sigma_\varepsilon'$ to obtain $\tilde \Sigma$
without $1/\varepsilon$ dependency.

\begin{figure}
    \centering
    \includegraphics[width=0.40\textwidth]{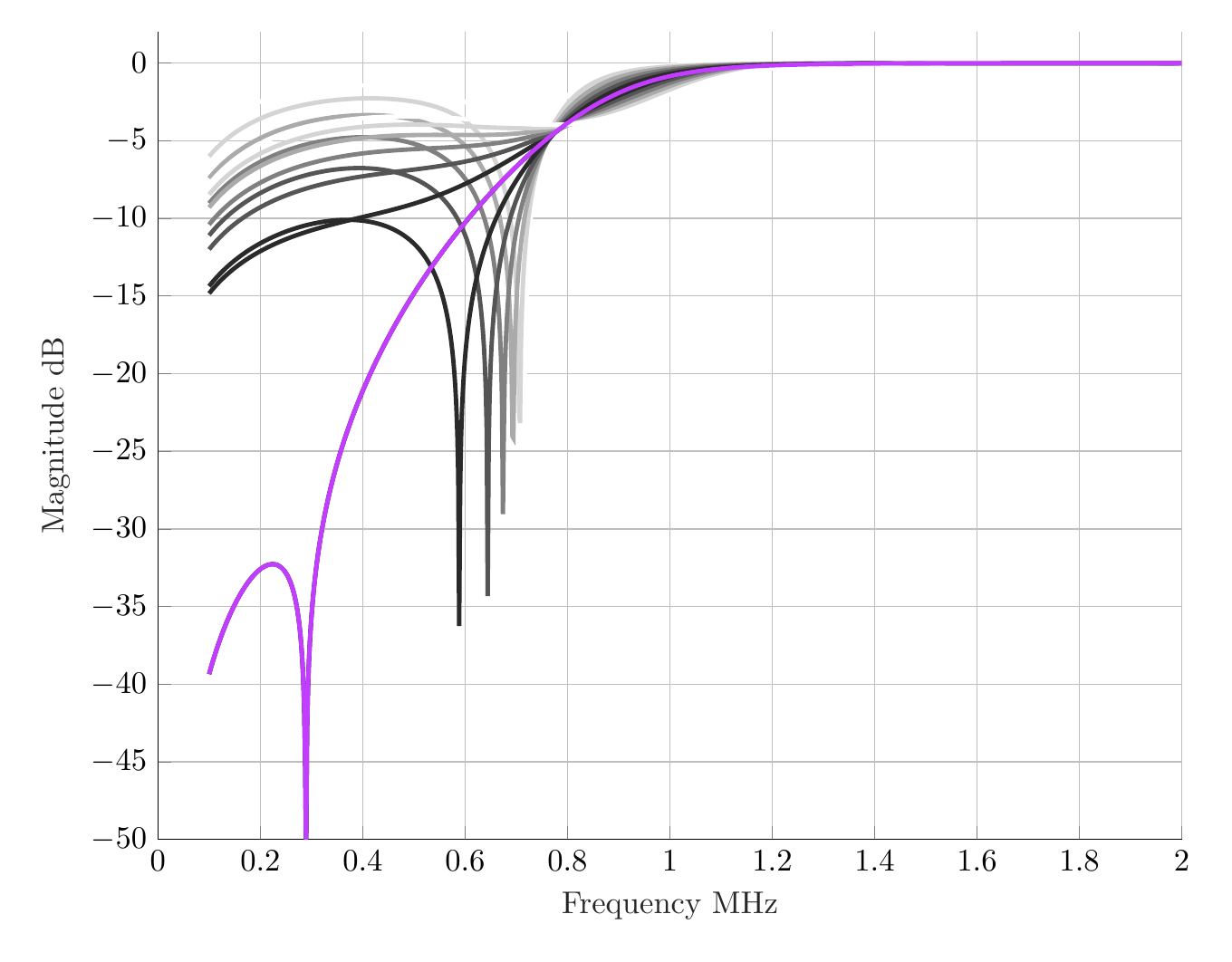}
    \includegraphics[width=0.40\textwidth]{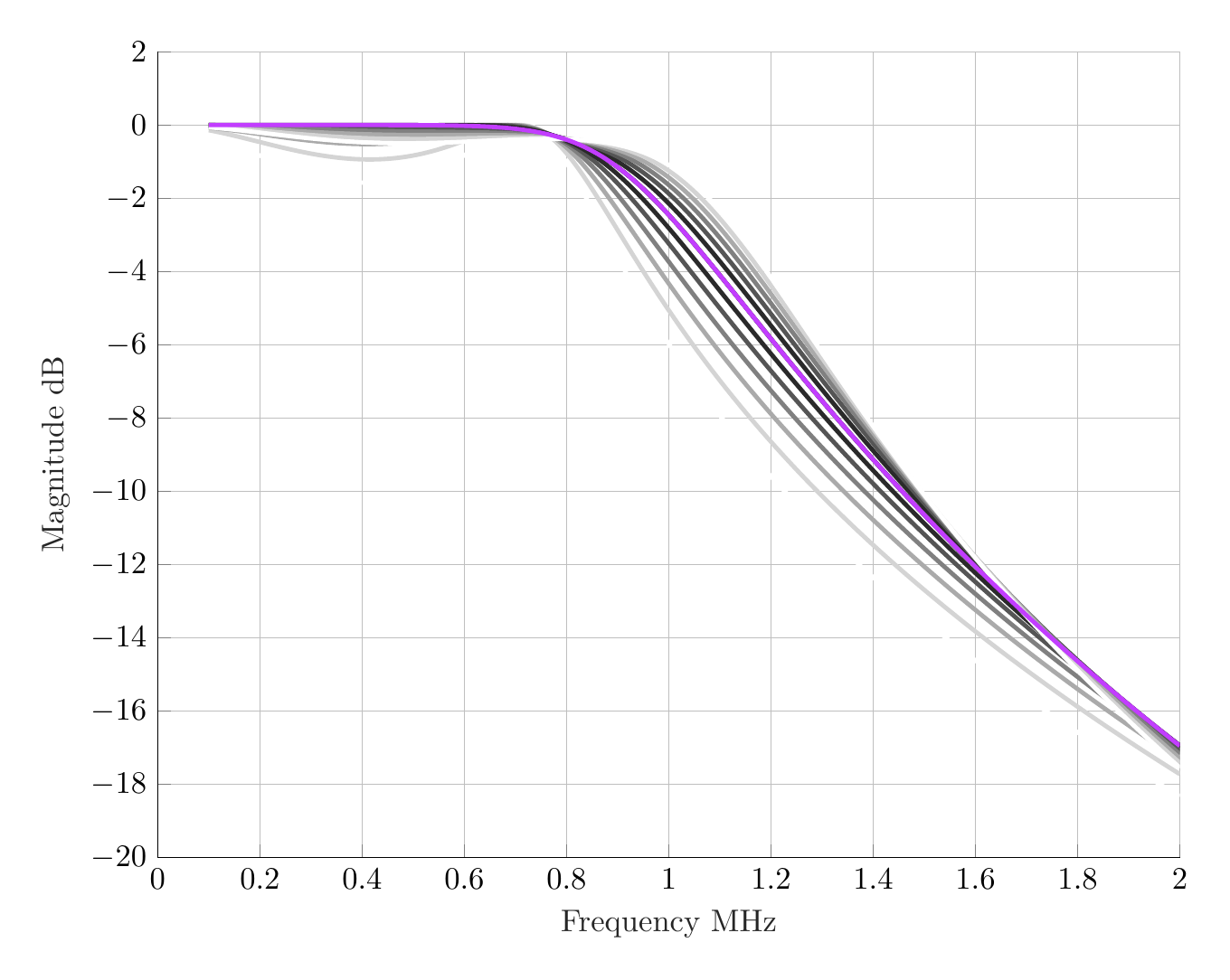}
    \caption[The $s$ parameters of $\Pi$-topology]{ Reflection and
      transmission parameters $s_{11}$ (left panel) and $s_{21}$
      (right panel) of the 5th order Butterworth low pass filter
      having the topology of Fig.~\ref{fig:RedhefferTwoPi}.  For the
      cut-off frequency $1 \, \mathrm{MHz}$ the approximate component
      values $C_1 = 2.2 \,\mathrm{nF}$, $C_3 = 6.8 \, \mathrm{nF}$,
      and $L = 14 \, \mu \mathrm{H}$ correspond to matching the $R_L =
      50 \, \Omega$ load at both ends. The different versions of
      $s_{11}$, $s_{21}$ have been computed from the Redheffer
      realisations $\Sigma_p(R, \varepsilon) \star \Sigma_q(R,
      \varepsilon)$ following Eq.~\eqref{eq:regularisedPiScattering},
      using the regularisation parameter value $\varepsilon = 1 \,
      \mathrm{n} \Omega$ and $R = R_0 \sbm{1 & 0 \\ 0 & 1}$ with
      $R_0/\Omega \in \{ 20, 25, \ldots , 50, \ldots 75, 80 \}$. The
      curve corresponding $R_0 = R_L$ show least ripple in the
      passband of $s_{21}$ as expected.}
\end{figure}

\subsection{\label{AcTransLineSec} Acoustic waveguide terminated to an irrational impedance}

We complete the article by modelling an acoustic waveguide that is
coupled from one end to an semi-infinite exterior space, modelled by a
piston model. Before forming the required Redheffer product in terms
of finite-dimensional systems, the acoustic part is spatially
discretised by FEM, and the load is rationally approximated by the
\Loewner 's interpolation method.  As an application of this, we compute
acoustic signals and frequency responses using anatomic data, acquired
by Magnetic Resonance Imaging (MRI) from a test subject during
production of vowel $[i]$ as explained in
\cite{A-A-H-J-K-K-L-M-M-P-S-V:LSDASMRIS}. 

\subsubsection{\label{FEMSubSec} Formulating the FEM system}

Given continuous and strictly positive functions $k_1, k_2: [0, L] \to
\R$, consider the partial differential equation for $\phi = \phi(t,
\chi)$ satisfying
\begin{equation}
    \label{eq:webster}
    \ddot{\phi} =  \frac{1}{k_1(\chi)}\frac{\partial }{\partial \chi}\left( k_2(\chi) \frac{\partial \phi}{\partial \chi} \right)
\quad \text{ for all } \quad \chi \in (0,L)  \quad \text{ and } \quad t \geq 0
\end{equation}
where $\ddot{\phi} := \tfrac{\partial^2 \phi}{\partial t^2}$ with the
boundary conditions
\begin{equation} \label{eq:webster_boundary}
\begin{aligned}
    - k_2(0) \frac{\partial \phi}{\partial \chi}(t,0) &= i_1(t) \quad \text{ and } \quad 
     k_2(L)\frac{\partial \phi}{\partial \chi}(t,L) &= i_2(t)  \quad \text{ for every } \quad t \geq 0
\end{aligned}
\end{equation}
for continuously differentiable input signals $i_1$ and $i_2$.  The
variational formulation of the problem is as follows: Find $\phi =
\phi(t,\chi)$, satisfying 
\begin{equation*}
  \phi \in C^2([0,\infty); L^2(0,L)) \cap C^1([0,\infty); H^1(0,L)) \cap C([0,\infty); L^2(0,L)) 
\end{equation*}
such that
\begin{equation} \label{eq:websterVariation}
\begin{aligned}
  & \frac{\partial^2}{\partial t^2} \int_0^L{k_1(\chi) \phi(t, \chi) v(\chi) \, \mathrm{d} \chi} + 
  \int_0^L{k_2(\chi) \frac{\partial \phi}{\partial \chi} (t, \chi) \frac{\partial v}{\partial \chi}(\chi) \, \mathrm{d} \chi} \\
  & = k_2(L) \frac{\partial \phi}{\partial \chi} (t, L) v(L) - k_2(0) \frac{\partial \phi}{\partial \chi} (t, 0) v(0) \\
  & =  v(L) i_2(t) +  v(0) i_1(t) 
\end{aligned}
\end{equation}
for all test functions $v \in H^1(0,L)$.

\subsubsection*{Cubic Hermite Finite Element spaces}

Treating the boundary condition such as
Eqs.~\eqref{eq:webster_boundary}--\eqref{eq:websterVariation}, it is
necessary to keep track of spatial derivatives. Instead of the usual
piecewise linear approximations, a convenient way of doing this is
using a Hermitian Finite Element (FE) space where these derivatives
appear as degrees-of-freedom; see, e.g., \cite{Kwon:FEM:1996},
\cite[Section~2.2.3]{Lapidus:NSPD:1999}.

Let $\mathcal{C}_h$ be a finite family of open, disjoint intervals
$(\chi_{i - 1}, \chi_i)$ with whose union is dense in $(0,L)$. The
notation and enumeration of the nodes $\chi_i$ of $\mathcal{C}_h$ is
given in Fig.~\ref{fig:notation}.
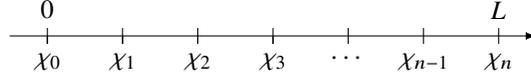
\begin{figure}
    \centering
    \begin{tikzpicture}
        \draw[-latex] (-0.5,0) -- (6.5,0);
        \foreach \x in {0,1,2,3,4,5,6}
        \draw[shift={(\x,0)},color=black] (0pt,3pt) -- (0pt, -3pt);
        \foreach \x in {0,1,2,3}
        \draw[shift={(\x,0)},color=black] (0pt,0pt) -- (0pt,-3pt) node[below] {$\chi_{\x}$};
        \draw[shift={(0,0)},color=black] (0pt,3pt) node[above] {$0$};
        \draw[shift={(6,0)},color=black] (0pt,3pt) node[above] {$L$};
        \draw[shift={(6,0)},color=black] (0pt,-3pt) node[below] {$\chi_n$};
        \draw[shift={(5,0)},color=black] (0pt,-3pt) node[below] {$\chi_{n-1}$};
        \draw[shift={(4,0)},color=black] (0pt,-4pt) node[below] {$\dots$};
    \end{tikzpicture}
    \caption{The notation related to the subdivision $\mathcal{C}_h$
      in the special case of an equidistant division of $[0,L]$.}
    \label{fig:notation}
\end{figure}
We make use of the finite-dimensional subspace
\begin{equation*}
  V_h := \{ w \in H^2(0,L) : w|_{K} \in P_3(K) \, \text{ for all } \, K \in \mathcal{C}_h \}
\end{equation*}
where $P_3(K)$ denotes the space of polynomials of degree $\leq 3$
restricted to $K \subset \R$.  The Galerkin method applied
Eq.~\eqref{eq:websterVariation} gives the following formulation: Find
$\phi_h \in C^2([0,\infty); V_{h})$ such that
\begin{equation}
  \label{eq:1}
   \frac{\partial^2}{\partial t^2} \int_0^L{k_1  \phi_h v_h \, \mathrm{d} \chi} + 
  \int_0^L{k_2  \frac{\partial \phi_h}{\partial \chi} \frac{\partial v_h}{\partial \chi} \, \mathrm{d} \chi} \\
   = v_h(L) i_2  + v_h(0) i_1 
\end{equation}
for all $v_h \in V_{h}$. We proceed to construct an explicit basis for $V_{h}$
so as to express Eq.~\eqref{eq:1} in matrix form.

The local basis functions (see Fig.~\ref{fig:lgbasis})
corresponding to each interval $(\chi_{i-1},\chi_i) \in \mathcal{C}_h$
for $i = 1, \ldots , n$, and they are given by
\begin{alignat*}{2}
    \varphi_i^1(\chi) &= 2 \ell_i(\chi)^3 - 3 \ell_i(\chi)^2 + 1, \quad
    &&\varphi_i^2(\chi) = -2\ell_i(\chi)^3 + 3 \ell_i(\chi)^2,\\
    \varphi_i^3(\chi) &= (\ell_i(\chi)^3 - 2 \ell_i(\chi)^2 + \ell_i(\chi))(\chi_i-\chi_{i-1}), \qquad
    &&\varphi_i^4(\chi) = (\ell_i(\chi)^3 - \ell_i(\chi)^2)(\chi_i-\chi_{i-1})
\end{alignat*}
where $\ell_i(\chi) := \frac{\chi - \chi_{i-1}}{\chi_i - \chi_{i-1}}$.  There are two
families of global basis functions corresponding to each \emph{interior node} $\chi_j$ for
$0 < j < n$, and they are given by
\begin{equation*}
    \psi_j^1(\chi) = \begin{cases}
        \varphi_j^2(\chi), & \text{if $\chi \in (\chi_{j-1},\chi_j]$,} \\
        \varphi_j^1(\chi), & \text{if $\chi \in (\chi_{j},\chi_{j+1})$,} \\
        0, & \text{otherwise,}
    \end{cases} \quad
    \psi_j^2(\chi) = \begin{cases}
        \varphi_j^4(\chi), & \text{if $\chi \in (\chi_{j-1},\chi_j]$,} \\
        \varphi_j^3(x), & \text{if $\chi \in (\chi_{j},\chi_{j+1})$,} \\
        0, & \text{otherwise.}
    \end{cases}
\end{equation*}
We enumerate these two families as 
\begin{align} \label{eq:basis}
    \psi_k(\chi) &= \begin{cases}
        \psi_k^1(\chi), & \text{if}\,\, k \in \{1,\dots,n-1\}, \\
        \psi_{k-n}^2(\chi), & \text{if}\,\, k \in \{n+1,\dots,2n-1\}
    \end{cases}
\end{align}
which is a basis for $V_{h} \cap H^1_0(0,L)$.  However, solutions
$\phi$ of Eq.~\eqref{eq:webster} are expected to have non-trivial
Dirichlet traces $\phi(t,0)$ and $\phi(t,L)$. Hence, we need to add
the functions
\begin{equation} \label{eq:basisTwoMore}
  \psi_{0} = \varphi_1^1  \quad \text{ and  } \quad  \psi_{n} = \varphi_n^2 
\end{equation} 
corresponding the \emph{end point nodes} $\chi_0 = 0$ and $\chi_n =
L$, to obtain the basis $\{\psi_k\}_{k = 0, 1, \ldots , 2n-1}$ for $V_{h}$;
see Fig.~\ref{fig:lgbasis}.

\subsubsection*{System of linear equations}

Since $\phi_h(t, \cdot) \in V_{h}$ for all $t \geq 0$, we get using
the basis of Eqs.~\eqref{eq:basis}--\eqref{eq:basisTwoMore}
\begin{equation*}
    \phi_h(\chi,t) = \sum_{k=0}^{2n-1}{w_k(t) \psi_k(\chi)}
\quad \text{ for } \quad x \in [0, L] \text{ and }  t \geq 0
\end{equation*}
where $w_k$ are some twice differentiable functions. Writing the
coefficient vector $w := \bbm{w_0 & \ldots & w_{2n-1}}^T \in \mathbb{R}^{2n}$
and the input matrix 
\begin{equation*}
  F :=  \bbm{1  & 0 & \ldots &  0 & 0 & 0 & \ldots &  0  \\
  0 & 0 & \ldots &  0 & 1  & 0 & \ldots &  0 }^T \in \mathbb{R}^{2n \times 2},
\end{equation*}
since $\varphi_1^1(0) = \varphi_n^2(L) = 1$; here nonzero entries are
in the first and $(n+1)$st positions.  Eq.~\eqref{eq:1} takes the form
\begin{equation} \label{eq:SecondOrderFEM}
    M \ddot{w}(t) + K w(t) = F \bbm{i_1(t) \\ i_2(t)},
\end{equation}
where $M = \bbm{m_{i j}}$, $K = \bbm{k_{i j}} \in \mathbb{R}^{2n\times
  2n}$ are given by
\begin{equation*}
  \begin{aligned}
    m_{ij} &= \int_0^L k_1 \psi_j \psi_i \,\mathrm{d}\chi \quad \text{ and } \quad
    k_{ij} &= \int_0^L{k_2  \frac{\partial \psi_j}{\partial \chi}  \frac{\partial \psi_i}{\partial \chi} \, \mathrm{d} \chi}.
  \end{aligned}
\end{equation*}
The matrix $M$ is always invertible since the functions $\psi_j$ are a
basis, but the matrix $K$ is never invertible since constant functions
are in $V_h$.  Defining the extended state trajectory $x = \sbm{w
  \\ \dot{w}}$, we get following
Corollary~\ref{thm:SpringConnectionCor} the equivalent form 
\begin{equation} \label{eq:FirstOrderFEM}
\begin{aligned}
  \dot{x}(t) & = A_i^{(1)}  x(t) +  B_i^{(1)} \bbm{i_1(t) & i_2(t)}^T \quad \text{ where } \\
  A_i^{(1)} & := \bbm{0 & K^{1/2} M^{-1/2} \\ - M^{-1/2} K^{1/2} & 0} \text{ and } B_i^{(1)} :=
  \bbm{0 \\ M^{-1/2} F}.
  \end{aligned}
\end{equation}

When modelling acoustics of a variable diameter tube, we have
$k_1(\chi) = \mathcal A(\chi)/c^2$ and $k_2(\chi) = \mathcal A(\chi)$
where $c$ is the speed of sound and $\mathcal A(\chi)$ is the
cross-sectional area at $\chi \in [0, L]$. The physical dimension of
the input signals $i_1$ and $i_2$ is volume velocity (given in
$\mathrm{m}^3/\mathrm{s}$). To read out the sound pressures (given in
$\mathrm{Pa}$) $p_1$ and $p_2$ from the system at the ends $\chi = 0$
and $\chi = L$ , we define
\begin{equation} \label{eq:FirstOrderFEMOut}
  \bbm{p_1(t) & p_2(t)}^T = C_i^{(1)} \xx(t) \quad\text{ with } C_i^{(1)} := \rho
  \bbm{0 & F^T M^{-1/2}}
\end{equation}
where the dimension of the zero matrix is $2 \times 2n$, and $\rho >
0$ denotes the density of the medium. The system defined by
Eqs.~\eqref{eq:FirstOrderFEM}--\eqref{eq:FirstOrderFEMOut} is
henceforth denoted by $\Sigma_i^{(1)} = \sbm{A_i^{(1)} & B_i^{(1)}
  \\ C_i^{(1)} & 0}$. Note that $\Sigma_i$ is not yet an impedance
conservative system, but impedance conservativity is achieved by
dividing $\rho$ in Eq.~\eqref{eq:FirstOrderFEMOut} evenly between
$B_i^{(1)}$ and $C_i^{(1)}$ as $\rho^{1/2}$, see
Theorem~\ref{thm:SpringConnection}.

\begin{figure}
    \centering
    \includegraphics[width=0.45\textwidth]{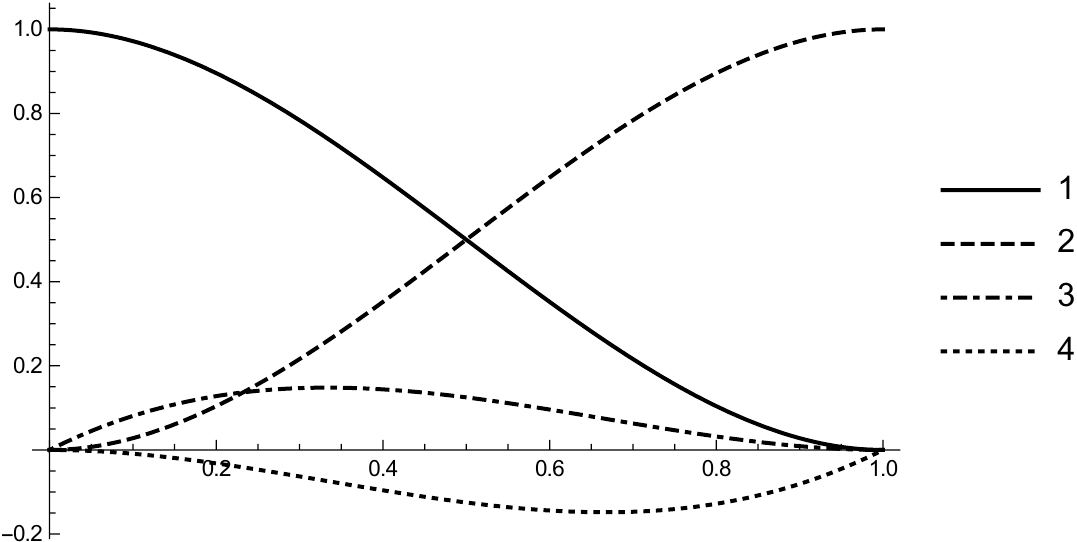}
    \includegraphics[width=0.45\textwidth]{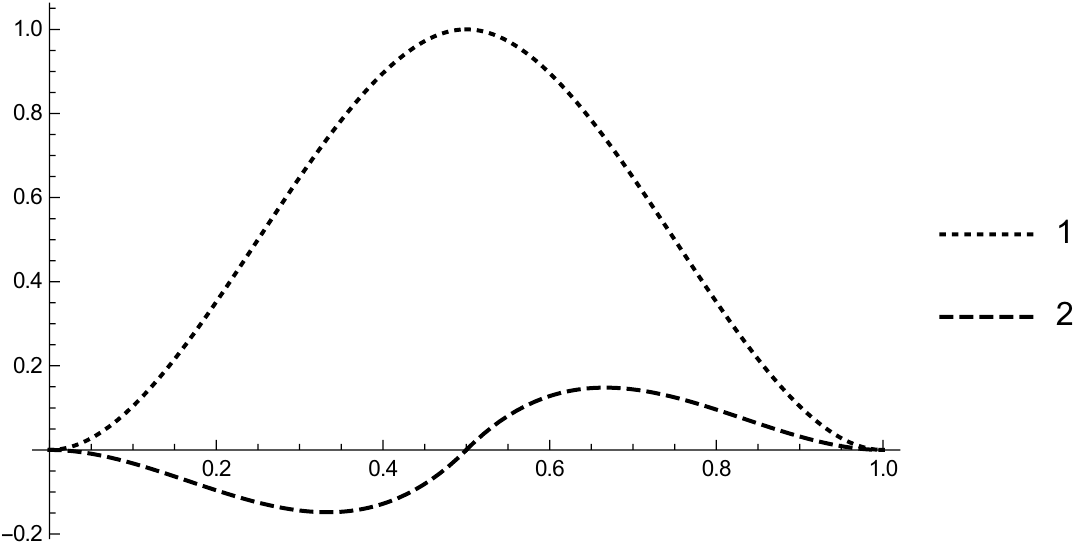}
    \caption{The local basis functions $\varphi_i^k$,
      $k\in\{1,2,3,4\}$, on the interval $(\chi_{i-1},\chi_{i})=(0,1)$ (left panel).
      The global basis functions $\psi_j^k$, $k\in\{1,2\}$,
      corresponding to the node $\chi_j=0.5$ (right panel). In the picture,
      $\chi_{j-1}=0$ and $x_{j+1}=1$.
 }
    \label{fig:lgbasis}
\end{figure}
 
\subsubsection{\label{LownerSubSec} L{\"o}wner interpolation of the exterior space impedance}

We proceed to construct a low-order model for the acoustics of an
unbounded half space in $\R^3$ as seen from a circular aperture of
radius $a > 0$. Following the work of Morse and
\Ingard~\cite{Morse:1968}, the acoustic impedance $Z(s)$ from the
piston model is given above in Eq.~\eqref{eq:acImpIntro}.  This
function satisfies $Z(\overline{s}) = \overline{Z(s)}$ and, hence,
$Z(\R) \subset \R$ for physical reasons.  We proceed in two steps: the
function $Z(s)$ is first approximated by a rational interpolant using
\Loewner's method, and then its McMillan degree is further reduced by
the Singular Value Decomposition (SVD).

For a given integer $m > 0$, we use two sets of interpolation points
$\{ \mu_i \}, \{ \lambda_j \} \subset \C_-$, each having $m$ elements
with $\{ \mu_i \} \cap \{ \lambda_j \} = \emptyset$. The interpolation
data from $Z(s)$ is represented in terms of two $m \times m$ matrices
$\Ll = \bbm{l_{ij}}$ and $\M = \bbm{m_{ij}}$ given by
\begin{equation*}
  l_{ij} = \frac{Z(\mu_i)-Z(\lambda_j)}{\mu_i - \lambda_j} \quad \text{ and } \quad 
m_{ij} = \frac{\mu_i Z(\mu_i)- \lambda_j Z(\lambda_j)}{\mu_i - \lambda_j}
\end{equation*}
for $1 \leq i, j \leq m$ where $2m$ is the total number of
interpolation points. Following \cite[Chapter 1]{Antoulas:IMR:2010},
we call $\Ll$ and $\M$ \emph{\Loewner{}} and \emph{shifted \Loewner{}
  matrices}, respectively.  Defining the vectors
\begin{equation*}
  b :=  \bbm{Z(\mu_1) & Z(\mu_2) & \ldots & Z(\mu_m)}^T
  \quad \text{ and } \quad c := \bbm{Z(\lambda_1) & Z(\lambda_2) & \ldots & Z(\lambda_m)}^T
\end{equation*}
we get the linear dynamical system, say $\tilde{\Sigma}_p$, given in descriptor form as
\begin{equation}\label{eq:LowUnreg}
\begin{cases}
  \Ll v'(t) & =  \M v(t) -  b i(t)  \\
  p(t)&  = c^T v(t)
\end{cases}
\end{equation}
whose transfer function is a rational interpolant of values of $Z(s)$
at points $\{ \mu_i \} \cup \{ \lambda_j \}$ if the matrix pencil
$\left (\Ll , \M \right)$ is regular; see
\cite[Theorem~1.9]{Antoulas:IMR:2010}, \cite[Section~4.5.2]{Antoulas:ALD:2005}.

To avoid complex arithmetics as in \cite[Appendix~A.2]{Ionita:DPM:2014}
while respecting the property $Z(\overline{s}) = \overline{Z(s)}$, we
use interpolation points satisfying in addition to $\{ \mu_i \} \cap
\{ \lambda_j \} = \emptyset$ and $\{ \mu_i \} \cup \{ \lambda_j \}
\subset \C \setminus{\R}$ the following conditions:
\begin{equation} \label{eq:InterpRestrictions}
  \begin{aligned}
    & \Ll \text{ is invertible}; \quad 
    \mu_i \neq \mu_j \text{ and } \lambda_i \neq \lambda_j \text{ for } i \neq j;  \text{ and }\\ 
    & \mu_{j+1} = \overline{\mu_j} \text{ and } \lambda_{j+1} = \overline{\lambda_j} \text{ for odd }  j = 1, 3, \ldots m-1
  \end{aligned}
\end{equation}
where $m$ even. An unitary change of coordinates given in
\cite[Appendix~A.2]{Ionita:DPM:2014} makes it possible to replace
matrices $\Ll , \M$ and vectors $b, c$ in Eq.~\eqref{eq:LowUnreg} by
their purely real counterparts that are henceforth denoted by the same
symbols.

After the construction of \emph{real} matrices $\Ll$ and $\M$ for
Eq.~\eqref{eq:LowUnreg}, the adverse effects of oversampling are
removed by reducing their order via SVD
following~\cite[Section~4.2]{Ionita:DPM:2014}.  More precisely, we
write $\Ll = USV^T $, and pick left and right singular vectors
corresponding to the $k$ largest singular values into $m \times k$
isometric matrices $U_k$ and $V_k$ with $k \ll m$.  The reduced order
$k \times k$ matrices $\Ll_k = U_k^T\Ll V_k $, $\M_k = U_k^T\M V_k$,
and the vectors $B_i^{(2)} := U_k^Tb$, $C_i^{(2)} := cV_k^T$ define
the system $\Sigma_i^{(2)} = \sbm{A_i^{(2)} & B_i^{(2)} \\ \tilde
  C_i^{(2)} & 0}$ through the dimension reduced equations derived from
Eq.~\eqref{eq:LowUnreg}
\begin{equation}
  \label{eq:LowRegSystem}
  \begin{cases}
    z'(t) = A_i^{(2)} z(t) + B_i^{(2)} i(t) \\ p(t) = C_i^{(2)} z(t)
  \end{cases}
\end{equation}
where $A_i^{(2)} := \Ll_k^{-1}\M_k$. The external reciprocal transform
of $\Sigma_i^{(2)}$ is a realisation related to the
Dirichlet-to-Neumann map that is used in resonance computations of
coupled Helmholtz systems; see, e.g., \cite{Araujo:ERC:2018}.

\subsubsection{\label{ResultsSubSec} Modelling vowel production}

To define the system $\Sigma_i^{(i)}$, the cross-section areas
$\mathcal A(\chi)$ for $\chi \in [0, L]$ are obtained from vocal tract
(VT) geometry of a test subject while producing the vowel sound [i]
through the process described in
\cite{Aalto:SurfExtraction:2013,Ojalammi:Segmentation:2017}.  The
length of the VT centreline is $L = 19.6 \, \mathrm{cm}$, and $\chi =
0, L$ denote the vocal folds and mouth opening positions,
respectively. For spatial discretisation, the equidistant subdivision
of $[0, L]$ into $n = 99$ subintervals is used, and $A(\chi)$ is
considered piecewise linear on this subdivision. The waveguide system
$\Sigma_i^{(1)}$ is produced as described above, and the dimension of
its state space is $4n = 396$.

The mouth opening area is used as the parameter $\mathcal A_0 =
\mathcal A(L)$ for the exterior space impedance $Z(s)$ in
Eq.~\eqref{eq:acImpIntro}. A finite-dimensional system is produced by
sampling $Z(s)$ at $2 m = 300$ interpolation points $\{ \mu_i \} \cup
\{ \lambda_j \}$ from inside the square $W = \{ s \in \C \, : \, 0
\leq \abs{\Re{s}}, \abs{\Im{s}} \leq 3\cdot 10^5 \,
\mathrm{rad}/\mathrm{s} \}$ while obeying the restrictions of
Eq.~\eqref{eq:InterpRestrictions}.  Some of the interpolation points
are near the zeroes of $Z(s)$, and the rest are uniformly distributed
random points. The frequencies under $48 \, \mathrm{kHz}$ get
accurately modelled in this way.  The low-order exterior space model
$\Sigma_i^{(2)}$ of McMillan degree $k = 16$ is obtained by dimension
reduction, and the estimated relative error of its transfer function,
compared to $|Z(s)|$, is under $3\cdot10^{-6} \,
\mathrm{kg}/\mathrm{m^4 \, s}$ in the frequency interval of interest.
The values $c = 343 \, \mathrm{m}/\mathrm{s}$ and $\rho = 1.225 \,
\mathrm{kg}/\mathrm{m}^3$ are used for both $\Sigma_i^{(1)}$ and
$\Sigma_i^{(2)}$.

To obtain the composite system $\Sigma_i = \sbm{A_i & B_i \\ C_i & 0}$
(where $A_i$ is a $412 \times 412$ matrix) and its time
discretisation, the following steps are taken:
\begin{enumerate}
\item The impedance passive systems $\Sigma_i^{(1)}$ and $\Sigma_i^{(2)}$ are
  regularised and externally Cayley transformed to obtain the
  scattering systems $\Sigma^{(1)}(R)$ and $\Sigma^{(2)}(R_2,
  \varepsilon)$ with $R= \sbm{R_1 & 0 \\ 0 & R_2}$ and $R_1 =R_2 = 1.1
  \cdot 10^6 \, \mathrm{kg}/\mathrm{m^4 \, s}$. The choice of the
  regularisation parameter $\varepsilon > 0$ is discussed
  below.\footnote{Note that only one of the systems $\Sigma^{(1)}(R)$
    and $\Sigma^{(2)}(R_2, \varepsilon)$ need be regularised for using
    the Redheffer star product. Applying the regularisation to the
    exterior space model $\Sigma_i^{(2)}$ has the additional advantage
    for introducing physically realistic resistive effects at the
    mouth opening. For this reason, we do not let $\varepsilon \to 0$
    as in Section~\ref{ButterworthSec} but choose it optimally based
    on measured formant data.}
\item The Redheffer star product $\Sigma(R_1, \varepsilon) :=
  \Sigma^{(1)}(R) \star \Sigma^{(2)}(R_2, \varepsilon)$ is computed to
  obtain the composite system in scattering form. The system
  $\Sigma(R_1, \varepsilon)$ is a scattering passive model for the VT
  where mouth at $\chi = L$ is coupled to the exterior space, and a
  virtual control surface right above vocal folds at $\chi = 0$ is
  coupled to a measurement port of load impedance $R_1$.
\item The inverse external Cayley transformation (with $R = R_1$) is
  used to obtain the impedance system $\Sigma_i(\varepsilon)$ from
  $\Sigma(R_1, \varepsilon)$.
\item For time domain simulations by Crank-Nicolson method, the
  internal Cayley transform $\phi_\sigma(\varepsilon) = \spm{A_\sigma
    & B_\sigma \\ C_\sigma & D_\sigma}$ of $\Sigma_i(\varepsilon)$ is
  computed using $\sigma = 88 \, 200$, corresponding the time
  discretisation parameter value $h = 22.7 \mu s$ and the sampling
  frequency $44 \, 100 \, \mathrm{Hz}$. 
\end{enumerate}
The transfer function $\CTF_i(s) = C_i \left (s - A_i \right ) ^{-1}
B_i$ of $\Sigma_i(\varepsilon)$ is the computed total acoustic
impedance of VT and the exterior space as seen at the vocal folds
position. The resonant frequencies $f_{R1}, f_{R2}, \ldots$,
corresponding to the lowest vocal tract formants $F_1, F_2, \ldots $,
are obtained from the eigenvalues $\lambda \in \sigma(A_i)$ by $R_j=
\Im{\lambda_j}/2 \pi$. 

The temporally discretised system $\phi_\sigma(\varepsilon)$ is
discrete time impedance passive by
Proposition~\ref{ScatteringPassiveProp}.  To simulate the vowel
production in time domain, it needs to be excited by a flow signal
waveform, using it as the input to the difference equations
Eq.~\eqref{eq:systemDynDisc} at the sampling frequency of $44 \, 100
\, \mathrm{Hz}$. A suitable flow waveform is provided by the
Liljencrantz--Fant (LF) pulse train at $f_o = 120 \, \mathrm{Hz}$
together with the synthesised signals shown in
Fig.~\ref{fig:VTSimulations} (left panel) (see also
\cite[Fig.~4]{M-A-M-A-V:MLBVFOVTA} 
estimated spectra (middle and right panel). Burg's method (with model
order 100) is used for the data in the middle panel, and an envelope
detector is used in the right panel. The spectral tilt over the
frequency interval $[200 \, \mathrm{Hz}, 6.4 \, \mathrm{kHz}]$ is
$\approx - 6.5 \, \mathrm{dB}/\mathrm{octave}$ in
Fig.~\ref{fig:VTSimulations} (middle panel). Considering that losses
to VT walls and the glottal opening at $\chi = 0$ have been neglected,
this value is reasonably consistent with the values $-9.0 \ldots
-9.2\, \mathrm{dB}/\mathrm{octave}$ that were given in
\cite[Table~2~on~p.15]{K-M-O:PPSRDMRI} 

\begin{figure}
    \centering
    \includegraphics[width=0.33\textwidth]{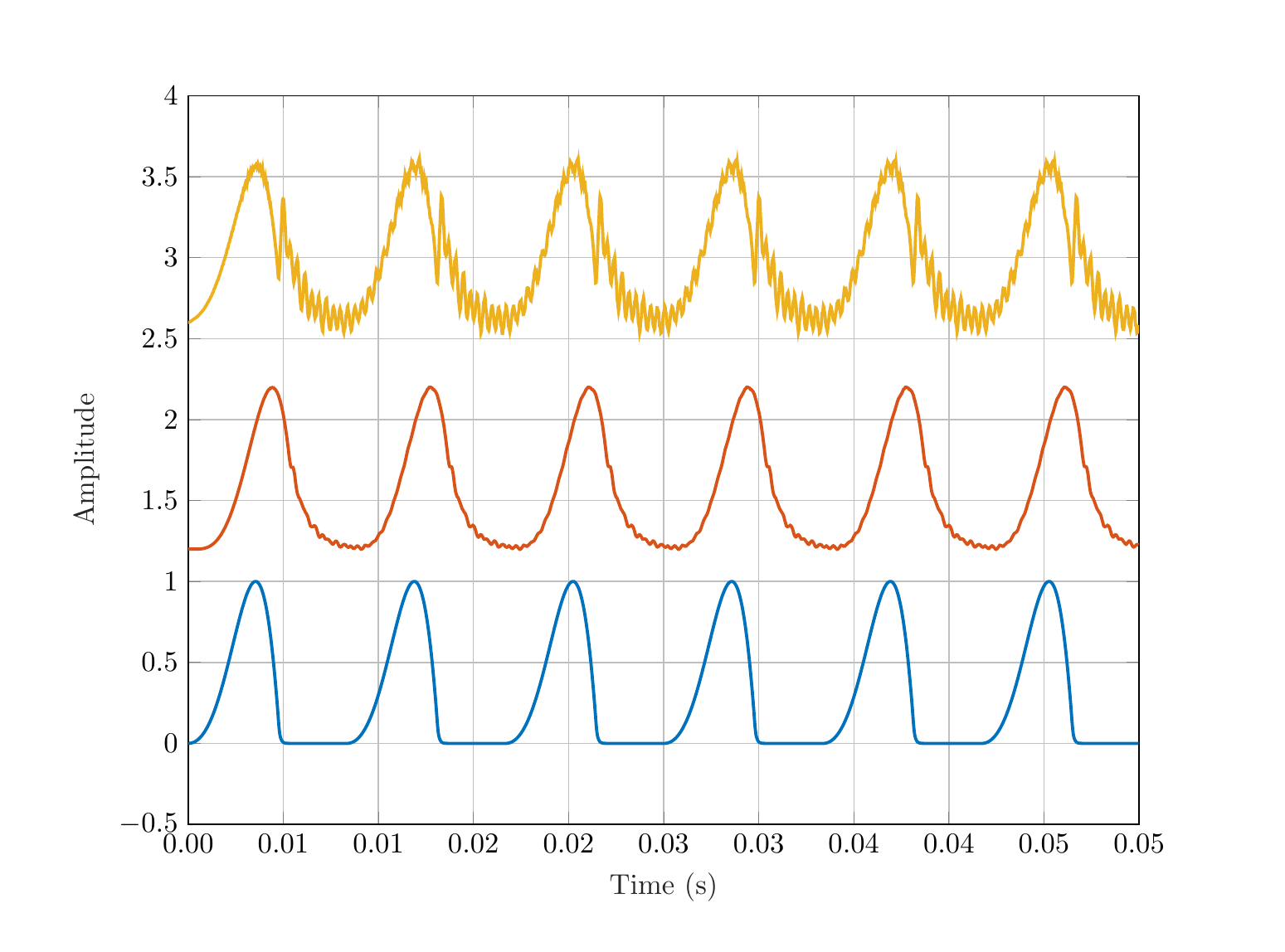}
    \includegraphics[width=0.33\textwidth]{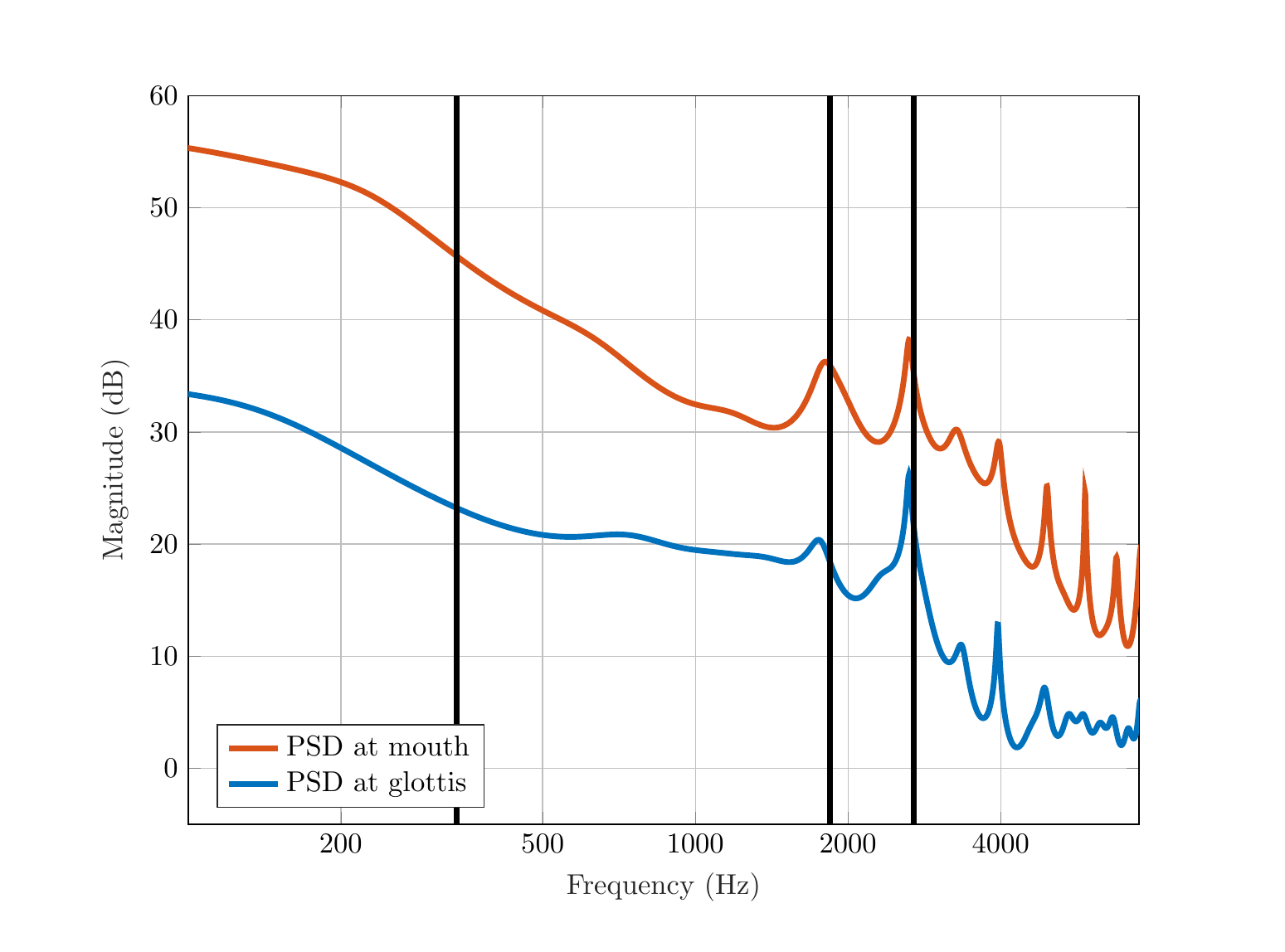}
    \includegraphics[width=0.29\textwidth]{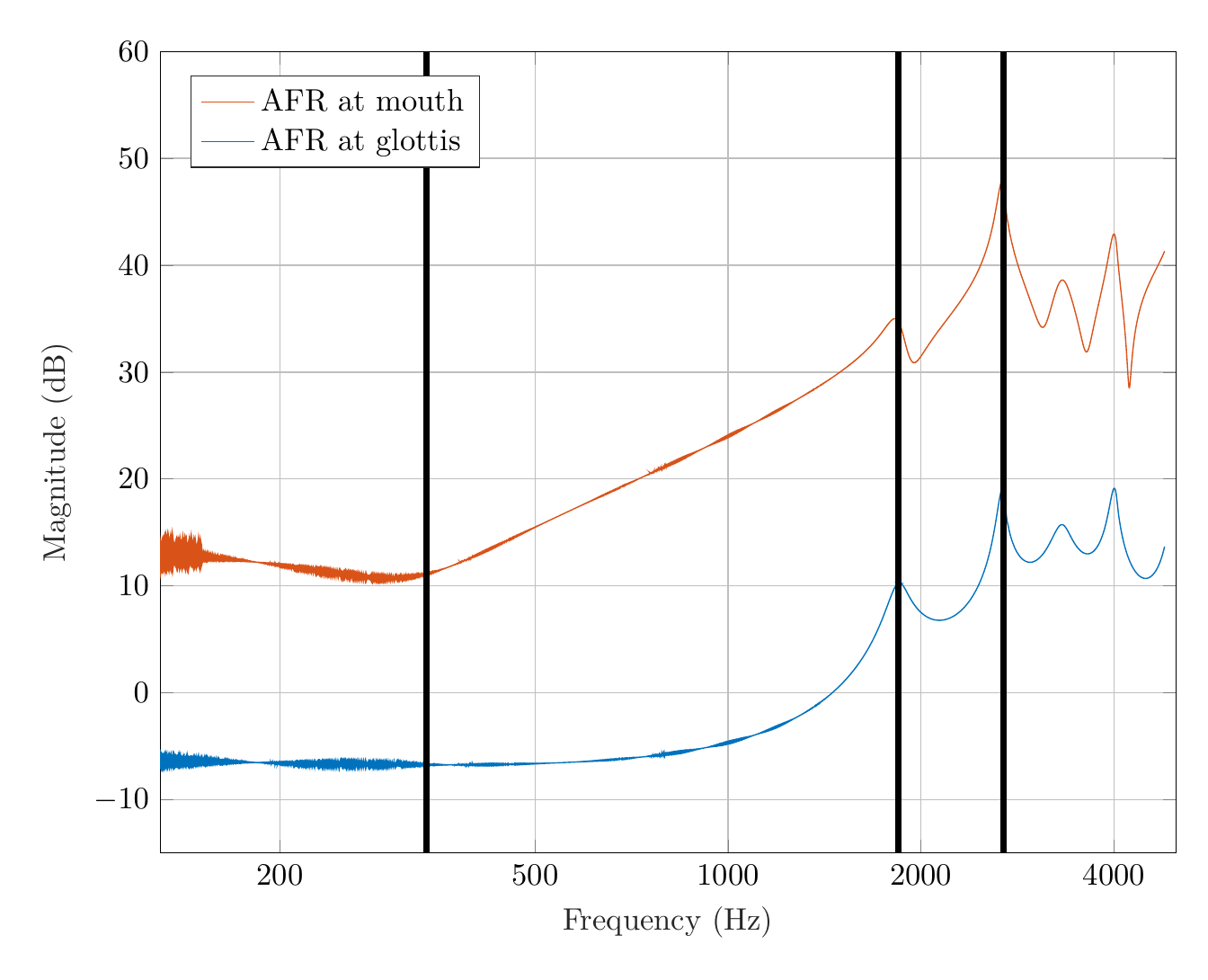}
    \caption{Liljencrantz--Fant (LF) glottal flow excitation waveform
      at $f_o = 120 \, \mathrm{Hz}$ (left panel, bottom) together with
      the acoustic pressure waveforms at mouth (middle) and at vocal
      folds (top) positions as computed from the model.  The envelopes
      of the power spectral densities (using LF excitation) at mouth
      and vocal folds positions (middle panel). Amplitude frequency
      responses (using constant amplitude logarithmic sweep) as the
      input (left panel). All the responses are given without absolute
      scale.}
    \label{fig:VTSimulations}
\end{figure}

We complete this work by discussing the accuracy of the low order
model $\Sigma_i(\varepsilon)$ in terms of measured data.  The lowest
resonant frequencies, computed from the eigenvalues of $A_i$, are
$f_{R1} = 338 \, \mathrm{Hz}$, $f_{R2} = 1845 \, \mathrm{Hz}$, and
$f_{R3} = 2693 \, \mathrm{Hz}$ using $\varepsilon = 0.194 \, Z_0 $
(with $Z_0 = \rho c/\mathcal A_0$) as the regularisation parameter for
$\Sigma_i^{(2)}$ that is tuned to match $f_{R1} \ldots f_{R3}$ to the
measured target data $F_1 \ldots F_3$ optimally.  As the target data,
we use the experimentally obtained formant values $F_1 \ldots F_3$
from vowel samples that were measured from the same test subject in
anechoic chamber. These values are $F_1 = 340 \pm 25 \, \mathrm{Hz} $,
$F_2 = 1840 \pm 40 \, \mathrm{Hz}$, and $F_3 = 2490 \pm 60\,
\mathrm{Hz}$ as can be read off from \cite[Fig.~9]{K-M-O:PPSRDMRI}.
Obviously, the computed resonant frequencies $f_{R2} , f_{R3}$ match
the peaks in spectrograms in Fig.~\ref{fig:VTSimulations}, extracted
from the simulated sound pressure signals at vocal folds and mouth
positions.

It is pointed out in~\cite{A-M-V:F} that not all formants can always
be observed in the power spectrum due to confounding factors; such
formant is called \textit{latent}. Since there is an underlying
resonance $f_{R1}$ corresponding to formant $F_1$ which is, hence,
classified as latent in Fig.~\ref{fig:VTSimulations}. Moreover, the
model produces no extra spurious resonances under $3 \, \mathrm{kHz}$,
i.e., resonant frequencies not accounted for by the measurement data
on vocal tract replicas \cite[Male\_i\_sweep.pdf in
  Repository~III]{A-M-M-K-S-A-S-V-G:OPENGLOT}, nor the speech
measurements during MRI \cite[Fig.~9]{K-M-O:PPSRDMRI}, except for
those that appear on the negative real axis.  Note that $\varepsilon$
can be understood as additional acoustic series resistance to the
exterior space model, and it appears that essentially only $f_{R1}$ is
sensitive to $\varepsilon$.

Discrepancies between $f_{Rj}$ and $F_j$ can be computed from
\begin{equation*}
  D_j = 12 \log_2{\frac{f_{Rj}}{F_j}} \quad \text{ in semitones for } j = 1, 2, 3, 
\end{equation*}
yielding $D_1 = -0.1$, $D_2 = 0.05$, and $D_3 = 1.4$, respectively. A
computational experiment based on 3D Helmholtz equation was reported
in \cite[Section~5.2]{A-A-H-J-K-K-L-M-M-P-S-V:LSDASMRIS} 
where the same 3D MRI data
was used but the exterior space model was trivial: the homogeneous
Dirichlet boundary condition was imposed at the mouth opening.  The
resulting discrepancies from the computational data of
\cite[Fig.~7]{A-A-H-J-K-K-L-M-M-P-S-V:LSDASMRIS} 
are $\tilde D_1 = -7.5$, $\tilde D_2
= 2.0$, and $\tilde D_3 = 2.9$ (accounting for the different speed of
sound $c = 350\, \mathrm{m}/\mathrm{s}$ used in
\cite{A-A-H-J-K-K-L-M-M-P-S-V:LSDASMRIS}. 

We conclude that the proposed 1D acoustics model with only $412$
degrees-of-freedom and an improved treatment of exterior space
acoustics produces much better match for $F_1 \ldots F_3$ from
experimental data, compared to the 3D Helmholtz FEM solver with
$\approx 10^5$ degrees-of-freedom.

\section{Conclusions}

Finite-dimensional realisations have been proposed as a framework for
practical numerical modelling of interconnected passive systems. The
use of the machinery was illuminated by examples from circuit
synthesis and acoustic waveguides.

Under some restrictive assumptions, up to six equivalent
reformulations were given in Section~\ref{TransformSec} for the same
underlying dynamical system.  It is in the nature of things that each
of the reformulations make some aspects quite transparent while
obscuring other aspects. For example, to write electro-mechanical
model equations in continuous time, one is likely to prefer impedance
passive formulations in terms of external current and voltage signals
satisfying Kirchhoff's laws in couplings. Continuous time scattering
passivity deals with power transmission and reflection parameters
satisfying the conservation of energy at the component interfaces, and
it is particularly useful for modelling feedback systems as in
Section~\ref{RedhefferSec}.  In itself, the characterisation of
passivity is easiest for continuous time in impedance setting and for
discrete time in scattering setting; see
Propositions~\ref{ImpPassiveProp}~and~\ref{ScatteringPassivePropDiscrete}.
It is desirable that spatial discretisation preserves the passivity of
the original system, and for some Finite Element discretisations this
follows from Theorem~\ref{thm:SpringConnection}. Finally, temporal
discretisation by Tustin's method, i.e., the internal Cayley
transformation, leads to scattering or impedance passive discrete time
systems as given in Proposition~\ref{ScatteringPassiveProp}.

A few words about generalisations that are peripheral to the point of
this
article. Propositions~\ref{ImpPassiveProp},~\ref{ScatteringConservativeProp},~\ref{ScatteringPassivePropDiscrete},
and \ref{ScatteringPassiveProp} are special cases of known
infinite-dimensional results on \emph{system nodes}.  Internal
transformations in Section~\ref{InternalCayleySec} have natural
generalisations to system nodes in \cite{OS:WPLS}, but the external
transformations in Section~\ref{ExtCayleySubSec} require more care
because they refer to the feedthrough operator $D$ in an essential
way. However, there is a straightforward generalisation to the
\emph{state-linear systems} (see~\cite{C-Z:IID}) where $A$ generates a
contraction semigroup on a Hilbert state space $X$, with $B, C^* \in
\BLO(U;X)$ and $D \in \BLO(U)$ for the signal Hilbert space
$U$. Considering
Propositions~\ref{AnyROKProp}~and~\ref{ExtCayleyPassivityProp}, the
external Cayley transform (with $R = I$) is given in
\cite[Theorem~5.2]{Staffans:PCCT:2002} for well-posed (but not
necessarily regular) impedance passive system nodes.  Furthermore,
Propositions~\ref{ExtCayleyPassivityProp}~and~\ref{ReciprocalPassivityProp}
can be generalised to system nodes by usual techniques whereas only a
weaker form of Theorem~\ref{ProperlyImpPassThm} holds even in the case
of regular systems where $D$ may be non-compact.
Theorems~\ref{RedhefferCascadeThm}~and~\ref{RedhefferScatteringThm}
generalise to state-linear system with finite-dimensional $U$, but
both $\Delta_1$ and $\Delta_2$ need be assumed invertible if $\dim{U}
= \infty$. A higher generalisation of the Redheffer star product is,
perhaps, easiest obtained by translating to discrete time by the
internal Cayley transformation. Finally, Standing
Assumption~\ref{SecondSA} is mostly a matter of convenience, and the
signal dimensions in splittings need not generally be the same with
the notable exception of the chain transformation where the
invertibility of $D_{21}$ is crucial.

Even though state space control is often motivated by the ease of
numerical linear algebra compared to the treatment of rational
transfer functions, the appeal of realisation techniques is diminished
by their inherent numerical burden in large problems. Simulation in
discrete time by the internal Cayley transform of $\Sigma = \sbm{ A &
  B \\ C & D }$ can always be arranged so that the matrices in
Eq.~\eqref{eq:IntCayleyDef} need not be computed but, instead, a
linear problem is solved at each time step. For long time simulations,
one would prefer pre-computing the matrices in
Eq.~\eqref{eq:IntCayleyDef} if storage is not a concern. So as to
external transformations and the Redheffer star product, some
resolvents of the feedthrough matrix $D$ of $\Sigma$, or its parts,
need be computed at the expense of loss of sparseness. This is a
trivial requirement for the examples in Section~\ref{ApplicationsSec}
but increasingly expensive when, e.g., coupling together 3D acoustic
subsystems at a common 2D boundary interface. Thus, some form of
economy should be exercised in the spatial discretisation of the
interfaces connecting the component systems.  Alternatively, a special
kind of dimension reduction could be used at the interface
degrees-of-freedom as was done in \cite{H-M-O:ESSEPFCS} for the
coupled symmetric eigenvalue problem.

\subsection*{Acknowledgements}

JK has received support from Academy of Finland (13312124,
13312340). JM has received support from Magnus Ehrnrooth
foundation. TG has received support from Finnish Cultural foundation. 


\end{document}